\documentclass[a4paper, 11pt, reqno]{amsart}

\usepackage{amsmath,amssymb,amscd,amsthm,amsxtra, esint}

\usepackage{mathtools}
\usepackage{xcolor}

\usepackage[linktoc=page,colorlinks=true,linkcolor=blue,citecolor=red]{hyperref}

\usepackage{enumitem}

\usepackage{marginnote}

\makeatletter
\renewcommand{\subjclassname}{%
  \textup{1991} Mathematics Subject Classification}
\@xp\let\csname subjclassname@1991\endcsname \subjclassname
\@namedef{subjclassname@2000}{%
  \textup{2000} Mathematics Subject Classification}
\@namedef{subjclassname@2010}{%
  \textup{2010} Mathematics Subject Classification}
  \@namedef{subjclassname@2020}{%
  \textup{2020} Mathematics Subject Classification}
\makeatother

\allowdisplaybreaks[2]

\newtheorem{theorem}{Theorem}[section]
\newtheorem{lemma}[theorem]{Lemma}
\newtheorem{proposition}[theorem]{Proposition}
\newtheorem{remark}[theorem]{Remark}
\newtheorem{example}{Example}
\newtheorem{definition}[theorem]{Definition}
\newtheorem{corollary}[theorem]{Corollary}
\newtheorem{innerassumption}{Assumption}
\newenvironment{assumption}[1]
  {\renewcommand\theinnerassumption{#1}\innerassumption}
  {\endinnerassumption}

\DeclareMathOperator*{\supp}{supp}

\newcommand{\noi}{\noindent}
\newcommand{\Z}{\mathbb{Z}}
\newcommand{\R}{\mathbb{R}}
\newcommand{\C}{\mathbb{C}}
\newcommand{\T}{\mathbb{T}}

\newcommand{\deff}{\overset{\textup{def}}{=}}

\let\Re=\undefined\DeclareMathOperator*{\Re}{Re}

\let\P= \undefined
\newcommand{\P}{\mathbf{P}}

\newcommand{\EE}{\mathcal{E}}
\newcommand{\CC}{\mathcal{C}}
\renewcommand{\L}{\mathcal{L}}

\newcommand{\NN}{\mathcal{N}}

\newcommand{\RR}{\mathcal{R}}
\newcommand{\K}{\mathbb{K}}
\newcommand{\F}{\mathcal{F}}

\newcommand{\eps}{\varepsilon}
\newcommand{\ft}{\widehat}
\newcommand{\wt}{\widetilde}
\newcommand{\cj}{\overline}
\newcommand{\dx}{\partial_x}

\newcommand{\dt}{\partial_t}
\newcommand{\les}{\lesssim}

\newcommand{\jb}[1]
{\langle #1 \rangle}

\newcommand{\ind}{\mathbf 1}

\renewcommand{\S}{\mathcal{S}}

\newcommand{\ZZ}{\mathcal{Z}}

\newcommand{\M}{\mathcal{M}}

\newcommand{\N}{\mathbb{N}}
\renewcommand{\SS}{\mathbb{S}}
\newcommand{\Sb}{\mathbf{S}}

\newtheorem*{ackno}{Acknowledgements}

\numberwithin{equation}{section}
\numberwithin{theorem}{section}

\begin{document}
\baselineskip = 15pt

\title[Well-posedness for modulated dispersive PDEs]
{Remarks on nonlinear dispersive PDEs with rough dispersion management}

\author{Tristan Robert}

\address{\small{Université de Lorraine, CNRS, IECL, F-54000 Nancy, France}}

\email{tristan.robert@univ-lorraine.fr}



\begin{abstract}
In this work, we study the Cauchy problem for a class of dispersive PDEs where a rough time coefficient is present in front of the dispersion. Under minimal assumptions on the occupation measure of this coefficient, we show that for the large class of semilinear dispersive PDEs whose well-posedness theory relies on linear estimates of Strichartz or local smoothing type, one has the same well-posedness theory with or without the modulation. We also show a regularization by noise type of phenomenon for rough modulations, namely, large data global well-posedness in the focusing mass-critical case for the modulated equation. Under rougher assumptions on the modulation, we show that one can also transfer the well-posedness theory based on multilinear Fourier analysis from the original dispersive PDE to the modulated one. In the case of the NLS equation on $\mathbb{R}^d$ and $\mathbb{T}^d$, this covers all the sub-critical and critical regularities, thus completing and extending the various results currently available in the literature. At last, in the case of the periodic Wick-ordered cubic NLS, we show an even stronger form of regularization by noise, namely well-posedness in the critical Fourier-Lebesgue space for rough modulations.
\end{abstract}

%
\maketitle

\vspace{-5mm}

\tableofcontents

\section{Introduction}
\subsection{Setting}
We consider abstract dispersive PDEs with modulated dispersion:
\begin{align}\label{EQ}
\begin{cases}{\displaystyle \dt u+\frac{d W}{dt}\L u+ \NN(u)=0,}\\
u_{|t=0}=u_0
\end{cases}
~~u:(t,x)\in\R\times \M\to \K,
\end{align}
where $\M$ is some boundaryless manifold to be specified, $\K\in\{\R;\C\}$, $\L$ is a skew-adjoint operator, and $\NN$ is a nonlinear function of $u$ possibly depending also on its derivatives. The modulation function $W_t:\R\to\R$ in \eqref{EQ} is a given continuous but not necessarily differentiable function, and its time derivative in \eqref{EQ} above is in general interpreted in the sense of distribution. Actually, to avoid the issue of ill-defined products due to a potential lack of regularity in time, we will always work at the level of the mild (Duhamel) formulation of \eqref{EQ}, namely
\begin{align}\label{Duhamel}
u(t)=e^{-(W_t-W_0)\L}u_0-\int_0^te^{-(W_t-W_\tau)\L}\NN(u(\tau))d\tau,
\end{align}
which indeed can be made sense of as soon as $W_t$ is continuous.

The model \eqref{EQ} has been widely studied in the particular case of the nonlinear Schr\"odinger (NLS) equation, namely $\L=-i\Delta$ and $\NN(u)=|u|^{p-2}u$ for some $p\ge 2$. In this context, when $\M=\R$, the NLS equation with modulation \eqref{EQ} appears as a model for the propagation of a signal in an optical fibre with dispersion management; due to their fruitful applications to the improvement on the transmission rate in optical fibres, there is a wide physical literature on dispersion-managed solitons, see e.g. \cite{GT1,GT2,KH,Ku,LK,MM,MMetal,Zetal}, the reviews \cite{PhysRev1,PhysRev2}, and references therein. In turn, this sparked more rigorous mathematical investigations \cite{ASS,CHL1,CHL2,CLL,CL,DBD,DT,HL1,HL2,KMZ,Murphy,MVH1,MVH2}. In these works, $\frac{dW_t}{dt}$ is typically assumed to be regular enough, e.g. periodic and piece-wise constant with a finite number of discontinuities. In this context one is often interested in effective models as the period goes to zero, in either the so-called strong \cite{CHL1,CHL2,CL,HL1,HL2,KM,KMZ,Murphy,MVH1,MVH2} or fast \cite{ASS,CMVH} dispersion regimes, depending if the amplitude of the modulation also blows-up as the period goes to zero. 

In contrast, in \cite{DBD,DT,Marty1}, the case of a random dispersion is considered, and the effective model appearing as a scaling limit for optical fibres with random dispersion management is given by \eqref{EQ} with $W_t$ being a Brownian motion. This motivated the study of \eqref{EQ} with \emph{rough} modulation $W_t$. Note that our abstract result in Theorem~\ref{THM:noreg}~(i) below applies to all three settings, i.e. the usual setting $W_t=t$, the case of $\frac{dW_t}{dt}$ periodic and piece-wise constant with a finite number of discontinuities, or the case of rougher $W_t$; see Remark~\ref{REM:A0} below.

Let us now describe the literature regarding the NLS equation with \emph{rough} modulation. In \cite{DBD}, de Bouard and Debussche studied the modulated NLS equation \eqref{EQ} on $\M=\R^d$ and with Brownian modulation. Using stochastic properties of $W_t$, they established global well-posedness of \eqref{EQ} in $L^2(\R^d)$ for mass-subcritical nonlinearities $\NN(u)=|u|^{2\sigma}u$, $0<\sigma<\frac2d$. In dimension $d=1$, the mass-critical case $\sigma=2$ was considered by Debussche and Tsutsumi \cite{DT}, who proved a \emph{regularization by noise} phenomenon: the mass-critical NLS equation \eqref{EQ} with white noise dispersion is globally well-posed in $L^2(\R)$ even for large data and focusing nonlinearity. This is in sharp contrast with the deterministic focusing mass-critical NLS equation \eqref{EQ}, for which finite time blow-up occurs for data with mass bigger than that of the ground state \cite{Merle}. Then Duboscq and Réveillac \cite{DR} proved a path-wise version of the results in \cite{DBD,DT} by establishing path-wise dispersive and Hardy-Littlewood-Sobolev inequalities adapted to $W_t$ a (fractional) Brownian motion in place of $t$. Following the development of deterministic regularization by noise in the context of ODEs by Catellier and Gubinelli \cite{CaG}, Chouk and Gubinelli \cite{CG1}, and Galeati and Gubinelli \cite{GaG2} established a deterministic version of the results in \cite{DBD,DT,DR}, under appropriate assumptions on the occupation measure of $W_t$; see Definition~\ref{DEF:occup} and Example~\ref{EX:noreg}~(i) below for a comparison between these results and our abstract Theorem~\ref{THM:noreg} below. Very recently, Tanaka \cite{Tanaka} studied the modulated NLS equation on tori $\M=\T^d$. See also \cite{CGr} for a study of the modulated NLS equation on quantum graphs, and the very recent preprint \cite{Goel} discussing Strichartz estimates for the Schr\"odinger equation with both a potential and white noise dispersion.

Numerical simulations of the modulated NLS equation \eqref{EQ} with Brownian modulation \cite{BDBD,Maierhofer,Marty1,Marty2} suggest that a stronger regularization by noise phenomenon should occur in this case. Namely, it seems like the $1d$ NLS equation with white noise dispersion is well-posed in $L^2(\R)$ for mass-supercritical power nonlinearities $|u|^{2\sigma}u$, $\sigma \in (2,4]$. However, in the case of the periodic quintic NLS equation \eqref{EQ} with Brownian modulation, Stewart \cite{Stewart} showed that the same semilinear ill-posedness in $L^2(\T)$ as in the deterministic case \eqref{EQ:deterministic} by Kishimoto \cite{Kishimoto} also holds. This indicates that the regularization by noise phenomenon numerically conjectured in \cite{BDBD,Maierhofer,Marty1,Marty2} is not systematic. A strong regularization by noise is observed by Chouk, Gubinelli, Li, Li, and Oh in \cite{CG2} for KdV, Benjamin-Ono and the intermediate long wave equation. They showed well-posedness at \emph{any} regularity provided that $W_t$ is a suitably \emph{irregular} deterministic modulation. We generalize and extend this result in \cite{ModulatedNonRes} by showing that this phenomenon is typical in \emph{strongly non-resonant} dispersive PDEs\footnote{Loosely speaking, a dispersive PDE with dispersion relation $\varphi$ and algebraic nonlinearity $u^{m_1}\cj{u}^{m_2}$, $m_j\in\N$, is \emph{strongly non-resonant} if the associated resonance function $\Omega(\xi,\xi_1,\dots,\xi_{m_1},\eta_0,\eta_1,\dots,\eta_{m_2}):=\sum_{j=1}^{m_1}\varphi(\xi_j)-\sum_{j=0}^{m_2}\varphi(\eta_j)$ , defined for frequencies satisfying the convolution constraint $\sum_{j=1}^{m_1}\xi_j-\sum_{j=0}^{m_2}\eta_j=0$, enjoys a lower bound of the form $|\Omega(\vec{\xi},\vec{\eta})|\gtrsim \|(\vec{\xi},\vec{\eta})\|_{\infty}^\theta$ for some $\theta>0$ on the support of the nonlinearity. This is the case e.g. for the (fractional) KdV or Benjamin-Ono equations (restricted to mean zero initial data), or Wick ordered cubic fractional NLS equation with enough dispersion, but not the usual NLS equation. We refer to \cite{ModulatedNonRes} for the precise definition and for a discussion of the classes of dispersive PDEs which satisfy it or not.}. Thus these results are in sharp contrast with the counter-example provided by Stewart that we discussed above, and raise the question of observing such a phenomenon for models which are \emph{not} strongly non-resonant, such as the quintic periodic NLS equation studied in \cite{Stewart}. We give such examples in Theorem~\ref{THM:NLSFL} below and in \cite{ModulatedKP}; see Remark~\ref{REM:reg} below. We point out that as a by-product of the proof of Theorem~\ref{THM:noreg} below, we show that such regularization by noise can only come from \emph{multilinear} estimates, since linear space-time estimates à la Strichartz for the modulated semi-group are bounded above and below by that of the deterministic semi-group; see Remark~\ref{REM:sharp}. We refer to the introduction of \cite{ModulatedNonRes} and references therein for a further discussion and history on regularization by noise phenomena for ODEs and PDEs, in particular dispersive PDEs; see also \cite{JEDP} for an overview of these results.

\subsection{Assumptions}
In the following, we will use different assumptions on the irregularity of $W_t$, expressed through the regularity of its occupation measure $\mu_{[0,T]}$ measured in some appropriate topology.\\~~\\

\begin{definition}\label{DEF:occup}
\rm ~~
\begin{itemize}
\item Let $W:\R\to \R$ be a continuous path. For all time intervals $I\subset \R$, its \textbf{occupation measure} is defined as the measure on Borel sets $A\in\mathcal{B}(\R)$
\begin{align}\label{mu}
\mu_I(A):=\mathrm{Leb}\big(t\in I,~W_t\in A\big)=\int_I\mathbf{1}_A(W_t)dt.
\end{align}
\item In the following we assume that $\mu_I\ll dz$ for all intervals $I\subset \R$. Then ${\displaystyle\frac{d\mu_I}{dz}}$ is the \textbf{local time} of $W_t$ in $I$.
\end{itemize}
\end{definition}

The relation between regularity properties of a deterministic or stochastic path $W_t$ and its occupation measure $\mu$ have been investigated since the 70s, and we refer to the review \cite{Review} by Geman and Horowitz on these earlier results. Different topologies have been used to measure the regularity of $\mu$: while H\"older regularity is discussed in \cite{Review}, more recently the use of Fourier-Lebesgue regularity has been proved useful by Catellier and Gubinelli \cite{CaG} in the context of ODEs under rough additive perturbations $W_t$, or Chouk-Gubinelli \cite{CG1}, also with Li, Li and Oh  \cite{CG2} in the context of modulated dispersive PDEs. Even more recently, the work of Romito and Tolomeo \cite{RT} studies Besov regularity of $\mu$ under similar assumptions as in \cite{Review} on $W_t$. The common feature of these works is that \emph{regularity} of $\mu$ implies local \emph{irregularity} of $W_t$. Actually, the notion of prevalence allows to show some form of the converse, namely paths $W_t$ in some H\"older spaces satisfy ``a.e.'' a regularity property for their occupation measures; see the recent works \cite{GaG1,GaG2} by Galeati and Gubinelli where this is investigated.

 Below, we will assume either H\"older or Fourier-Lebesgue regularity à la Catellier-Gubinelli for the occupation measure $\mu$ of the path $W_t$ in \eqref{EQ}. Therefore we will make one of the following assumptions on the irregularity of $W_t$:

\begin{assumption}{\textbf{(A0)}}\label{A0}
\rm $W_t$ is a continuous function with local time $$\frac{d\mu_{[0,\cdot]}}{dz}(\cdot)\in L^\infty_{\textrm{loc}}(\R\times\R).$$
 \end{assumption}
\begin{assumption}{\textbf{(A0)*}}\label{A0*}
\rm $W_t$ is a continuous function with local time $$\frac{d\mu_{[0,\cdot]}}{dz}(\cdot)\in C^0_t\big(\R;L^\infty_{\textrm{loc}}(\R)\big).$$
\end{assumption}
\begin{assumption}{\textbf{(A0)$_{(\rho,\gamma)}$}}\label{A0gr}
\rm $W_t$ is a continuous function with local time $$\frac{d\mu_{[0,\cdot]}}{dz}(\cdot)\in C^\gamma_t\big(\R;\F L^{\rho,\infty}(\R)\big).$$
\end{assumption}
Before stating the assumptions on the model \eqref{EQ}, let us give some comments on the set of assumptions above regarding the regularity of the occupation measure $\mu$ of $W_t$.
\begin{remark}\label{REM:A0}
\rm~~\\
\textup{(i)} The definition of the Fourier-Lebesgue space $\F L^{\rho,\infty}(\R)$ appearing in assumption \ref{A0gr} above is recalled in Section~\ref{SEC:Prelim} below.\\
\textup{(ii)} The assumption \ref{A0gr} corresponds to the definition of $(\rho,\gamma)$-irregularity introduced in \cite{CaG} and also used in \cite{CG1,CG2}.\\
\textup{(iii)} From the standard embedding $\F L^{\rho,\infty}(\R) \hookrightarrow L^\infty(\R)$ when $\rho>1$, we have the following implications between the different assumptions above: $$\Big(\textrm{\ref{A0gr}},~~\gamma\ge 0,~~\rho>1\Big)\Longrightarrow\textrm{\ref{A0*}}\Longrightarrow\textrm{\ref{A0}}.$$
\textup{(iv)} When $W_t=t$, one has ${\frac{d\mu_{[0,t]}}{dz}(z)=\mathbf{1}_{[0,t]}(z)}$, which clearly satisfies \ref{A0} but not \ref{A0*}. Moreover, from $$\Big|\widehat{\frac{d\mu_{[0,t]}}{dz}}(\xi)\Big|=\Big|\frac{e^{it\xi}-1}{i\xi}\Big| \les |t|^\gamma|\xi|^{1-\gamma}$$ for any $\gamma\in [0,1]$, one finds that $W_t=t$ also satisfies \ref{A0gr} if and only if $\gamma\in [0,1]$ and $\rho\le 1-\gamma$. The same holds for more regular one-to-one maps, see \cite[Proposition 1.4]{CG1}.\\
\textup{(v)} Similarly, when $\frac{dW_t}{dt}$ is $T$-periodic and piecewise constant with a finite number of discontinuities, i.e. ${\displaystyle \frac{dW_t}{dt}=\sum_{k=1}^{K}\mathbf{1}_{(t_{k-1},t_k)}\dot{w}_{k-1}}$ for some partition $\{t_k\}_{k=0}^K$ of $[0,T]$, provided that $\dot{w}_k\neq 0$ for all $k$, one finds for $t\in [0,T]$ that ${\displaystyle \frac{d\mu_{[0,t]}}{dz}(z)=\sum_{k=1}^K\mathbf{1}_{[w_{k-1},w_{k-1}+\dot{w}_{k-1}(t_k-t_{k-1}))\cap[0,t]}(z)}$ for some $w_k\in\R$. Thus $W_t$ also satisfies \ref{A0} in this case.\\
\textup{(vi)} On the contrary, it follows from \cite{Review} that as soon as \ref{A0*} is satisfied, $W_t$ is not locally Lipschitz continuous at \emph{any} $t$. Similarly, for any $\eta\in(0,1)$, the assumption \ref{A0gr} for some $\rho>1$ and $\gamma>1-\eta$ implies that $W_t$ is not $\eta$-H\"older continuous at any $t$, see \cite[Theorem 1.5]{CG1}.
\end{remark}

We now turn to the dispersive properties of equation \eqref{EQ}. We will make use of one of the following assumptions on the domain and the dispersion.

\begin{assumption}{\textbf{(A1)}}\label{A1}
\rm $\M$ is a connected abelian Lie group: $\M=\R^{d_1}\times \T^{d_2}$ for some $d_1,d_2\in\N$, and $\widehat{\M}=\R^{d_1}\times\Z^{d_2}$ denotes its Pontryagin dual. Then $\L$ is a skew-adjoint operator given by a Fourier multiplier $\varphi\in C^1(\widehat{\M};\R)$ such that for any $u\in \S(\M)$ and $\xi\in\widehat{\M}$, $\widehat{\L u}(\xi)=i\varphi(\xi)\widehat{u}(\xi)$.
\end{assumption}
\begin{assumption}{\textbf{(A1)*}}\label{A1*}
\rm $\M$ is a closed Riemannian manifold, $\L$ is a skew-adjoint operator with compact resolvent and discrete spectrum $\{i\lambda_j\}_{j\in\N}$, $\lambda_j\in\R$, $|\lambda_j|\to \infty$, and associated orthonormal basis of eigenfunctions $\{\psi_j\}_{j\in\N}$ of $L^2(\M)$. Then $\widehat{\M}=\N$, and for any $u\in \S(\M)$ and $\xi\in\N$, $\widehat{u}(\xi)=\langle u,\psi_\xi\rangle_{L^2}$ and $\widehat{\L u}(\xi)=i\lambda_\xi\widehat{u}(\xi)$.
\end{assumption}

Regarding the nonlinearity in \eqref{EQ}, we will make the following mild assumption. In the following, $H^\infty(\M)=\cap_{s\in\R}H^s(\M)$ and $H^{-\infty}(\M)=\cup_{s\in\R}H^s(\M)$, where $H^s(\M)$ are the $L^2$-based Sobolev spaces; see Subsection~\ref{SUBS:space} below.
\begin{assumption}{\textbf{(A2)}}\label{A2}
\rm The nonlinearity $\NN$ satisfies $\NN(0)=0$, and maps $H^{\infty}(\M)$ to $H^{-\infty}(\M)$.
\end{assumption}

We will also eventually use a further assumption regarding the conservation laws associated with \eqref{EQ}.

\begin{assumption}{\textbf{(A3)}}\label{A3}
\rm  \ref{A2} holds, and the nonlinearity $\NN$ enjoys the following cancellation property:
\begin{align}\label{cancellation}
\forall\varphi\in H^\infty(\M),~~\Re\langle \varphi,\NN(\varphi)\rangle_{H^\infty\times H^{-\infty}} = 0.
\end{align}
\end{assumption}
\begin{remark}\label{REM:A2}
\rm
(i) For smooth solutions to the deterministic equation 
\begin{align}\label{EQ:deterministic}
\begin{cases}
\dt u +\L u +\NN(u)=0,\\
u(t=0)=u_0,
\end{cases}
\end{align}
the property \eqref{cancellation} is equivalent to the conservation of the mass $t\mapsto \|u(t)\|_{L^2}^2$. Indeed, if $u\in C([0,T);H^\infty(\M))$ is a smooth solution to \eqref{EQ:deterministic} on $[0,T)$, then since $\NN$ acts on smooth functions, we see on \eqref{EQ:deterministic} that $u\in C^1((0,T);H^{-\infty}(\M))$ and therefore $t\mapsto \|u(t)\|_{L^2}^2$ is $C^1$ on $(0,T)$, with
\begin{align*}
\frac{d}{dt}\|u(t)\|_{L^2}^2&=2\Re\langle u(t),\dt u(t)\rangle_{H^\infty\times H^{-\infty}}\\
&=2\Re\langle u(t),\L u(t)\rangle_{L^2}+2\Re\langle u(t),\NN(u(t))\rangle_{H^\infty\times H^{-\infty}} = 0
\end{align*}
for any $t>0$, due to the skew-adjointness of $\L$ and \eqref{cancellation}. The mass also has a right derivative at $t=0$, which also vanishes. Conversely, if the mass is conserved, then the previous computation evaluated at $t=0$ gives $2\Re\langle u_0,\NN(u_0)\rangle_{H^\infty\times H^{-\infty}}=0$ for all $u_0\in H^\infty(\M)$, which is \eqref{cancellation}. See also \cite{Ozawa} for related discussions on rigorous proofs of conservation laws for non-regular solutions.\\
However, in Theorem~\ref{THM:noreg} below, we only assume well-posedness of \eqref{EQ:deterministic} in $H^s(\M)$ for some $s\in\R$, but not necessarily propagation of regularity. Namely, even if $u_0\in H^\infty(\M)$, we only know a priori that the solution $u$ to \eqref{EQ:deterministic} is in $C([0,T];H^s(\M))$ for some $T(u_0)>0$ but not necessarily in $C([0,T];H^\infty(\M))$, which prevents us to deduce \eqref{cancellation} directly from the conservation of mass for \eqref{EQ:deterministic}. This is why we assume \eqref{cancellation} instead. Of course, in the concrete examples given in Example~\ref{EX:noreg} below, we can always check \eqref{cancellation} directly, even if the propagation of regularity could fail somehow (for example, in the case of a linear operator $\L$ with coefficients of finite regularity).
\\
(ii) In the case where $\NN$ does not depend on the derivatives of $u$, then \eqref{cancellation} can be replaced by $\Re\big( \cj{z}\NN(z)\big) = 0$ for all $z\in\C$.\\
(iii) Even if the corresponding deterministic equation \eqref{EQ:deterministic} is Hamiltonian with a conserved energy, we do not expect it to be conserved anymore for \eqref{EQ}. This is why we are only concerned with globalisation using invariance of the $L^2$ norm, hence assumption \ref{A3}.
\end{remark}

\subsection{Statement of the results}
Equipped with the set of assumptions from the previous subsection, we can now state our results on the Cauchy problem for the abstract modulated equation \eqref{EQ}.~~\\
\subsubsection{\textbf{Well-posedness through space-time estimates}}~~\\
Informally, our first result shows that for any semilinear dispersive PDE treated by a fixed-point argument in appropriate space-time norms, the same well-posedness results hold for both \eqref{EQ} and \eqref{EQ:deterministic}.

To state our result more precisely, we use the following notations: for a set of spatial functions $X(\M)\subset H^{-\infty}(\M)$ and $T>0$, we define the sets of space-time functions
\begin{align}\label{Spq}
S^{p,+}_T(X):= L^p([0,T];X(\M))\qquad\text{and}\qquad S^{p,-}_T(X)= X(\M;L^p([0,T])).
\end{align}
In \eqref{Spq} above, the spaces $S^{p,+}_T(X)$ are used to encode estimates of Strichartz type, which involve spaces like $L^p([0,T];W^{s,q}(\M))$, taking the time norm first; while the spaces $S^{p,-}_T(X)$ are used for estimates of local smoothing (Kato) type, such as the ones used for the gKdV equation in Subsection~\ref{SUBS:Example1} below, where the order of integration is reversed in spaces like $L^q(\M;L^p([0,T]))$.

For a finite family of spaces $\{S^{p_j,\iota_j}(X_j)\}_j$ and a regularity exponent $s\in\R$, we define for $T>0$ the Strichartz/local smoothing type space\footnote{Here and in \eqref{Strichartznonlinear} below, for $s\in\R$, the space $\Sb^{-s}_T$ is simply defined as \eqref{Sb} with $-s$ in place of $s$; in particular, no duality is involved in the $L^{p_j}$ or $X_j$ spaces.}
$$\Sb^s_T:=C([0,T];H^s(\M))\bigcap_j \jb{D}^{-s}S^{p_j,\iota_j}_T(X_j),$$
endowed with the norm
\begin{align}\label{Sb}
\|u\|_{\Sb^s_T}:=\|u\|_{L^\infty_T H^s}+\max_{j}\big\|\jb{D}^su\big\|_{S^{p_j,\iota_j}_T(X_j)}.
\end{align}

\begin{theorem}\label{THM:noreg}
Assume that \ref{A2} and either \ref{A1} or \ref{A1*} hold, and that the deterministic equation \eqref{EQ:deterministic} is locally well-posed in $H^s(\M)$ for some $s\in\R$ via (local in time) linear space-time (Strichartz/local smoothing) estimates. Namely, given $s\in\R$, there is a finite set $\mathcal{E}\subset (2,\infty)\times E\times\{\pm\}$ of admissible exponents and there are $n,C_0>0$ and a constant $\eta_{\EE}(s)\ge 0$ such that for any $(p,\zeta,\iota)\in\mathcal{E}$, there are spatial function spaces $X_{\zeta}\subset H^{-\infty}(\M)$ such that for any $T\in(0,1]$, there is a constant $C_{p,\zeta,\iota}(T)>0$ non-decreasing in $T$ satisfying
\begin{align}\label{Strichartz}
\big\|e^{-t\L}u_0\big\|_{S^{p,\iota}_T(X_{\zeta})}\le C_{p,\zeta,\iota}(T)\|u_0\|_{L^2}
\end{align}
for any $u_0\in L^2(\M)$,
and moreover the nonlinear estimate
\begin{align}\label{Strichartznonlinear}
\Big|\int_0^T\int_\M \big(\NN(u)-\NN(v)\big)\cj{w}dxdt\Big|\le C_0T^{\eta_\EE(s)}\|u-v\|_{\Sb^s_T}\big(\|u\|_{\Sb^s_T}+\|v\|_{\Sb^s_T}\big)^n \|w\|_{\Sb^{-s}_T}
\end{align}
holds for any $T\in(0,1]$ and any $u,v\in \Sb^s_T(\M)$ and $w\in \Sb^{-s}_T(\M)$.
Then the following hold.\\
\textup{(i)} If \ref{A0} holds, the corresponding modulated equation \eqref{EQ} is also locally well-posed in $H^s(\M)$. More precisely, if $\eta_\EE(s)>0$, then for any $u_0\in H^s(\M)$, there is $T=T(\|u_0\|_{H^s})>0$ such that there is a unique mild solution $u\in C([0,T);H^s(\M))\cap X^s_T$ to \eqref{EQ} on $[0,T)$, where $X^s_T$ is a solution space adapted to the modulated linear flow, defined in Definition~\ref{DEF:UpVpW} below. In the critical case $\eta_\EE(s)=0$, small data well-posedness holds: for any $T>0$, there is $\delta(T)>0$ such that for any $u_0\in H^s(\M)$ satisfying $\|u_0\|_{H^s}\leq \delta$, then there is a unique mild solution $u\in C([0,T);H^s(\M))\cap X^s_T$ to \eqref{EQ} on $[0,T)$. If \ref{A0*} also holds, this can be upgraded to large data local well-posedness as in the sub-critical case.

At last, the flow map $u_0\in H^s(\M)\mapsto u\in C([0,T);H^s(\M))\cap X^s_T$ is locally Lipschitz continuous.\\
\textup{(ii)} If \ref{A3} holds, and either \ref{A0*}, or \ref{A0} with $\eta_\EE(0)>0$, also hold, then in case $s=0$ the solution extends globally in time for any initial data in $L^2(\M)$.
\end{theorem}

Before giving some examples of nonlinear dispersive PDEs satisfying the assumptions in Theorem~\ref{THM:noreg}, we first make some comments on these results.

\begin{remark}\label{REM:A0*}
\rm~~\\
\textup{(i)} Clearly, \eqref{Strichartz}-\eqref{Strichartznonlinear} allow to run a fixed point argument on the Duhamel formulation of \eqref{EQ:deterministic} in a suitable ball of $\Sb^s_T$ for $T>0$ small enough depending on $u_0\in H^s(\M)$, thus showing that indeed the deterministic equation \eqref{EQ:deterministic} is well-posed in $H^s(\M)$.\\
\textup{(ii)} Only \ref{A0} is assumed in Theorem~\ref{THM:noreg}~(i). In particular, this contains the case $W_t=t$ as pointed out in Remark~\ref{REM:A0}~(ii). Of course Theorem~\ref{THM:noreg}~(i) is a tautology in this case due to the previous remark. This also covers the case of periodic and piece-wise constant $\frac{dW_t}{dt}$, see Remark~\ref{REM:A0}~(v), and the case of rougher $W_t$ in the sense of having a more regular local time. This means that Theorem~\ref{THM:noreg} does not distinguish between $W_t=t$ and more irregular $W_t$, in sharp contrast to Theorems~\ref{THM:NLScrit} and particularly~\ref{THM:NLSFL} below. Note that this is however in some sense optimal: see Remark~\ref{REM:sharp} below.\\
\textup{(iii)} As pointed out in Remark~\ref{REM:A0}~(i), \ref{A0*} is less restrictive than \ref{A0}$_{(\rho,\gamma)}$ for $\rho>1$. For example, if $W_t$ is a sample path of a fractional Brownian motion with Hurst parameter $H\in (0,1)$, then \ref{A0}$_{(\rho,\gamma)}$ holds a.s. for $\gamma\in[\frac12,1]$ and $\rho<\frac{1-\gamma}{H}$; see \cite[Theorem 1.4]{CaG}. Note that \cite[Theorem 1.4]{CaG} is a positive result, which by itself does not exclude that fBm paths could be $(\rho,\gamma)$-irregular for larger values of $\rho$; that this is a.s. not the case follows from \cite[Lemma 5.9]{GaG2}, see also the discussion in \cite[Remark 3.8]{GaG2}. 
On the other hand,\ref{A0*} holds for any $H\in(0,1)$, see \cite[Theorem 8.1]{Berman73} or \cite{Xiao}.\\
\textup{(iv)} From Remark~\ref{REM:A0}~(iii) and~(iv) we see that the assumptions on $W_t$ in Theorem~\ref{THM:noreg}~(i) and~(ii) are in some sense optimal.
\end{remark}
\begin{remark}
\rm
Theorem~\ref{THM:noreg}~(ii) can be seen as a \textit{noiseless regularization by noise} phenomenon. Indeed this result shows that due to the irregularity of $W_t$ in \eqref{EQ}, expressed through the regularity assumption \ref{A0*} on its occupation measure, we have that even for focusing mass-critical equations, one can get \textit{large data} global well-posedness. This of course is false in general for the deterministic equation \eqref{EQ:deterministic}, and thus the assumption \ref{A0*} in Theorem~\ref{THM:noreg}~(ii) is minimal to see this regularization effect. This phenomenon was first observed in \cite{DT} then \cite{CG1,DR} in the context of the quintic NLS equation on $\M=\R$ with (fractional) Brownian motion $W_t$. Theorem~\ref{THM:noreg}~(ii) generalizes this observation to other mass-critical equations and to rough but deterministic modulations. See Examples~\ref{EX:noreg} below for examples of dispersive PDEs where Theorem~\ref{THM:noreg}~(i) and~(ii) apply. Similar regularization by noise effects, namely lack of blow-up under stochastic perturbations of dispersive PDEs, also appear for some models with \emph{non-conservative} multiplicative noise; see e.g. \cite{BRZ,HRSZ,HRZ} and Remark~\ref{REM:CW} below. However, in the case of the focusing mass-supercritical NLS equation with a \emph{conservative} multiplicative noise, de Bouard and Debussche \cite{dBD02,dBD05} showed that the noise can instead lead to blow-up with positive probability, almost instantaneously, for \emph{any} initial data.

Note at last that the special role of $s=0$ in Theorem~\ref{THM:noreg}~(ii) is due to the fact that apart from the $L^2$ norm, we do not expect other conservation laws for \eqref{EQ} because of the time-dependent coefficient in front of the dispersion, as pointed out in Remark~\ref{REM:A2}.
\end{remark}

\begin{remark}
\rm
The proof of Theorem~\ref{THM:noreg} relies on the transfer of the linear space-time estimate \eqref{Strichartz} from the equation \eqref{EQ:deterministic} to the modulated equation \eqref{EQ}, in the same spirit as in e.g. \cite{Zhang}. Note that under \ref{A0}, $W_t$ can be very rough. In particular, as far as Schr\"odinger equations with non constant rough \emph{spatial} coefficients are concerned, it is known that there is a threshold of regularity for the spatial coefficients, below which Strichartz estimates come with a loss of derivative \cite{BP,FS,ST}; see also the recent works \cite{DM,DW,GUZ,MZ,TV} where perturbations of the Laplacian by rough random potentials are considered. It is interesting to note that \emph{time} coefficients are allowed to be much more singular than what is possible for \emph{spatial} coefficients, and that actually their roughness \emph{improves} the behaviour of (local) space-time estimates; see \eqref{StrichartzW}-\eqref{CW2}-\eqref{CW3} and Remark~\ref{REM:CW} below.
\end{remark}

\begin{remark}
\rm
The local solution space $X^s_T$ in Theorem~\ref{THM:noreg}~(i) is a $U^2$-type space adapted to the modulated linear evolution $e^{-W_t\L}$; see Subsection~\ref{SUBS:space} below for definitions and properties of these spaces. It replaces the space $D^W(H^s)$ used in \cite{CG1,CG2}. In particular the duality relation of Proposition~\ref{PROP:UpVp}~(iii) below can be seen as the endpoint case $\gamma=\frac12$ of the nonlinear Young integral theory of \cite[Theorem 2.3]{CG1} for $f\in C^\gamma_t\textrm{Lip}_M(H^s)$ and $g\in C^\gamma_tH^s$; see also \cite{CG2,GaReview} for further results based on nonlinear Young integration.

These spaces are nowadays heavily used to study the Cauchy problem for nonlinear dispersive PDEs at scaling-critical regularity. Here, even at sub-critical regularity, we crucially exploit their atomic structure, which is particularly compatible with assumptions \ref{A0}, \ref{A0*} or \ref{A0gr} on the occupation measure \eqref{mu}. On the contrary, it is not clear at all how one could use the standard $X^{s,b}$ spaces à la Bourgain when $W_t\not\equiv t$, due to a lack of information on the Fourier transform of ${\displaystyle t\mapsto\frac{d\mu_{[0,t]}}{dz}}$ with respect to the \emph{time} variable. Indeed, assumption \ref{A0gr} provides a control on the Fourier transform of the local time in the $z$ variable, together with H\"older regularity with respect to the endpoint $t$ of the time interval, but no information on its Fourier transform in the $t$ variable, which is what would be required to compare $X^{s,b}$-type norms adapted to the modulated flow $e^{-W_t\L}$ with the standard ones. This choice of spaces also allows to track more efficiently the dependence on the time interval. Though this does not play such a big role in this work, it is crucial when performing a short-time analysis as in \cite{ModulatedNonRes,ModulatedKP}.
\end{remark}

\begin{remark}
\rm
It is also possible to extend the admissible space-time estimates \eqref{Strichartz} of Theorem~\ref{THM:noreg} to an anisotropic setting. This would allow for example to take into account the local smoothing estimate for the Schrödinger equation in $\R^d$ when $d\ge 2$.
\end{remark}

We give below some examples of semilinear dispersive PDEs of the form \eqref{EQ:deterministic} satisfying the assumptions of Theorem~\ref{THM:noreg} and we refer to Subsection~\ref{SUBS:Example1} below for further discussion on the validity of these assumptions for each model.

\begin{example}\label{EX:noreg}
\rm~~

\begin{enumerate}
\item \textit{\textbf{The modulated nonlinear Schr\"odinger (NLS) equation on $\M=\R^d$:}}\\
Our first example is the modulated NLS equation
\begin{align}\label{NLS}
i\dt u -\frac{dW}{dt}\Delta u \pm |u|^{2m}u=0
\end{align}
for\footnote{The fact that we treat a nonlinearity $|u|^pu$ with $p$ restricted to be an even integer is not necessary and only made to simplify the presentation, but the same arguments as in \cite{Caz} allow to treat a more general nonlinearity $f(u)$ with suitable polynomial growth and $f(0)=0$.} some $m\in\N^*$. This model satisfies assumptions \ref{A1}, \ref{A2}, and \ref{A3}; see Subsection~\ref{SUBS:Example1} below. In this case, since the deterministic NLS equation is locally well-posed in $H^s(\R^d)$ in the whole (sub-) critical range $s\ge \max(s_c,0)$, where $s_c=\frac{d}{2}-\frac1m$, via Strichartz estimates \eqref{Strichartz} with $S^{p,+}_T(X)=L^p_TL^q_x$, Theorem~\ref{THM:noreg} shows local well-posedness of the modulated NLS equation \eqref{EQ:deterministic} also in the whole (sub-) critical range $s\ge \max(s_c,0)$. This recovers and improves the results of \cite{DBD,DT,CG1,DR,GaG2}. Indeed, compared to \cite{DBD,DT,DR}, here the assumption on $W_t$ is completely non-random, similarly as what is explained by Galeati and Gubinelli in \cite[Theorem 3.13]{GaG2} regarding the result of Duboscq and Réveillac \cite{DR}. The assumption \ref{A0} allows to relax the assumptions in \cite{DBD,DT,DR,GaG2} treating the case of (fractional) Brownian modulations\footnote{In particular, as mentioned in Remark~\ref{REM:A0*}, our results hold for a.e. path $W_t$ of a fractional Brownian motion of any Hurst parameter $H\in(0,1)$.}, and the assumptions of Chouk and Gubinelli in \cite[Theorems 1.9.1\& 1.10]{CG1} regarding the cases $d=1,m=1,2$, and $d=2,m=1$. Moreover, Theorem~\ref{THM:noreg} improves the result of \cite[Theorem 1.9.1]{CG1} regarding the case $d=2,m=1$ to well-posedness in the whole (sub-) critical range $s\ge 0$.

At last, when $-\Delta$ is perturbed by a suitable potential as in \cite{Goldberg,Goel}, then (the proof of) Theorem~\ref{THM:noreg} recovers the results of \cite{Goel} establishing the Strichartz estimate for the modulated equation from the one of \cite{Goldberg} for the standard Schr\"odinger equation with potential.\\
\item \textit{\textbf{The modulated NLS equation on $\M=\T^d$:}}\\
We still consider \eqref{NLS}, but with spatial domain $\M=\T^d$. This satisfies again assumptions \ref{A1}-\ref{A2}-\ref{A3}. In this case, under assumption \ref{A0}, Theorem~\ref{THM:noreg} recovers the sub-critical well-posedness of the periodic NLS equation \eqref{NLS} in $H^s(\T^d)$, $s>\max(s_c,0)$. In \cite{Tanaka}, Tanaka studies the same model, but under assumption \ref{A0gr} for some $\gamma,\rho\in(0,1]$, and shows well-posedness in $H^s(\T^d)$ for $s>s(\rho)=\max(\frac{d}2-\frac{\rho}m,s_c)$. Note in particular that $s(\rho)>s_c$ if and only if $\rho<1$. Since \ref{A0gr} implies \ref{A0} when $\rho>1$ as mentioned in Remark~\ref{REM:A0}~(iii), Theorem~\ref{THM:noreg} thus recovers and improves the result of Tanaka in this regime. In case $0<\rho\le 1$, assumptions \ref{A0gr} and \ref{A0} are not comparable anymore, but in the particular case where $W_t$ is a $H$-fBm, as mentioned in Remark~\ref{REM:A0*}~(iii), \ref{A0} is satisfied for any $H\in(0,1)$, and thus our Theorem~\ref{THM:noreg} improves the result of Tanaka by treating the full sub-critical range of regularity. Note however that assumptions \ref{A0} and \ref{A0gr} are genuinely different in nature: there exist \emph{differentiable} paths which are $\rho$-irregular for some $\rho<1$ but do not admit a local time -- for instance the integral in time of a (fractional) Brownian motion, see the discussion in \cite[Example iv., p. 2442]{GaG2} -- so that Tanaka's result applies to some modulations which are not covered by Theorem~\ref{THM:noreg}. In Appendix~\ref{SEC:appendix} below, we generalize assumption~\ref{A0} into~\ref{Atheta} to recover the case of singular occupation measures (i.e. without the need for a local time), and recover Tanaka's range of well-posedness under slightly weaker assumptions.

Under only the minimal assumption \ref{A0}, our proof misses\footnote{Actually, our proof shows well-posedness when $s=0$ and $d=m=1$ but only for \emph{small} initial data, even though the $1d$ cubic NLS is mass-subcritical.} the case $s=0$ when $d=m=1$. This is recovered if one assumes the slightly stronger assumption \ref{A0*}, see Remark~\ref{REM:cubicNLS} after the proof of Theorem~\ref{THM:noreg}. In the particular case $d=1,m=2$, Theorem~\ref{THM:noreg} also extends the result of Stewart \cite[Theorem 1]{Stewart} by treating the case of deterministic modulations, including the case of a.e. paths of Brownian motion. In particular, this result is sharp in terms of range of regularity $s$ for the periodic quintic NLS in view of the counter-example provided by Stewart.

As for the scaling-critical case $s=\max(s_c,0)$, the existing scaling-critical theory for the deterministic NLS equation on tori relies on multilinear Fourier analysis refining the simpler use of linear Strichartz estimates, and is thus not covered by Theorem~\ref{THM:noreg}. See Theorems~\ref{THM:NLScrit} and~\ref{THM:NLSFL} below where we address this issue. \\

\item \textit{\textbf{The modulated NLS equation on closed Riemannian manifolds:}}\\
We now consider the case where $\M$ is a closed (smooth, compact, boundaryless) $d$-dimensional Riemannian manifold. Then \eqref{NLS} satisfies assumptions \ref{A1*} and \ref{A3}. Under \ref{A0}, Theorem~\ref{THM:noreg} provides local well-posedness of \eqref{NLS} in $H^s(\M)$ for $s>\frac{d}{2}-\frac1{2m}$ similarly as in \cite{BGT1,ST} for the deterministic NLS equation on $\M$. Note that improvement on this range of regularity requires global information on the geometry of $\M$, see e.g. \cite{BGT2,BGT3,Herr}, together with suitable multilinear Fourier type analysis going beyond the assumptions of Theorem~\ref{THM:noreg}. See Theorem~\ref{THM:NLScrit} and Remark~\ref{REM:NLSsphere} below.\\

\item \textit{\textbf{The modulated quintic generalized Korteweg-de Vries (gKdV) equation on $\M=\R$}}\\
Our last example is the modulated quintic gKdV equation
\begin{align}\label{gKdV}
\dt u + \frac{dW}{dt}\dx^3 u \pm u^4\dx u = 0.
\end{align} 
In this case $\L = \dx^3$, $\NN(u)=\pm\frac1{5}\dx(u^{5})$ and $\M=\R$. This satisfies the assumptions \ref{A1}-\ref{A3}. The corresponding deterministic equation \eqref{EQ:deterministic} is mass-critical and locally well-posed in $H^s(\R)$ for any $s\ge 0$, and globally for small enough initial data \cite{KPV}, while finite-time blow-up occurs for large enough initial data in the focusing case \cite{Merle}. Thus under assumption \ref{A0} we also get local well-posedness in $H^s(\R)$, $s\ge 0$, for the modulated quintic gKdV equation \eqref{gKdV}. Moreover, under \ref{A0*}, this gives another example of mass-critical equation where we see large data global well-posedness for the modulated equation.
\end{enumerate}
\end{example}
~~\\

\subsubsection{\textbf{Further results for the NLS equation with noisy dispersion}}~~\\

Theorem~\ref{THM:noreg} above is only stated in a setting where the deterministic equation \eqref{EQ:deterministic} is well-posed through \emph{linear} space-time estimates. Thus we can wonder if equations whose well-posedness relies on \emph{multilinear} refinements of space-time estimates (e.g. bilinear Strichartz estimates) also enjoy the transfer of multilinear estimates from the deterministic equation \eqref{EQ:deterministic} to the noisy one \eqref{EQ}. This is the purpose of our next result, specialized to the case of the modulated NLS equation on tori at scaling critical regularity.

\begin{theorem}\label{THM:NLScrit}
Assume that $\M=\T^d$ for some $d\ge 1$, $\L=i\Delta$, $\K=\C$, and $\NN(u)=|u|^{2m}u$ for some $m\in\N^*$.
Then, if $m>\frac2d$, and \ref{A0gr} holds for some $\gamma>\frac12$ and $\rho>1$, the modulated periodic NLS \eqref{EQ} is locally well-posed in $H^s(\T^d)$ for any $s\ge s_c=\frac{d}2-\frac1m$. More precisely, for any $R>0$, there exists $T(R)\sim \jb{R}^{-\frac{2m}\gamma}>0$ (with $\gamma$ as in \ref{A0gr}) such that for any $u_0\in H^s(\T^d)$ with $\|u_0\|_{H^{s}}\le R$, there exists a unique mild solution $u\in C([0,T);H^s(\T^d))\cap X^s_T$ to \eqref{EQ} on $[0,T)$. Moreover, the flow map $u_0 \in H^s(\T^d)\mapsto u\in C([0,T);H^s(\T^d))\cap X^s_T$ is locally Lipschitz continuous.
\end{theorem}

\begin{remark}
\rm
Note that the local time of existence in Theorem~\ref{THM:NLScrit} above only depends on the size of the initial data in $H^s(\T^d)$, even in the scaling-critical case $s=s_c$. This suggests that under \ref{A0gr} with $\gamma>\frac12$ and $\rho>1$, $H^{s_c}(\T^d)$ becomes sub-critical for the modulated periodic NLS equation \eqref{EQ}. 
\end{remark}

\begin{remark}
\rm 
The results of Theorem~\ref{THM:NLScrit} for the modulated periodic NLS \eqref{EQ} are similar to the ones for the usual periodic NLS \eqref{EQ:deterministic} in \cite{HTT,Wang}. From the restrictions on $(m,d)$, the only cases not covered by Theorem~\ref{THM:NLScrit} are: $d=1$, $m=1,2$, and $d=2$, $m=1$. The cases $d=1$, $m=2$ and $d=2$, $m=1$ correspond to the mass-critical cases. Due to the $\eps$ loss of derivatives in the Strichartz estimate \eqref{StrichartzLoc} in the endpoint case $p=\frac{2(d+2)}d$, well-posedness in $L^2(\T^d)$ in these cases for the unmodulated equation \eqref{EQ:deterministic} remains a challenging open problem; see \cite{HerrKwak} for a recent improvement on \eqref{StrichartzLoc} giving a sharp logarithmic loss in the case $d=2$, $p=4$. The same loss of derivatives cannot be overcome using the modulation in \eqref{EQ}; see again \cite{Stewart} for a counter-example in the case $d=1$, $m=2$, and the discussion in Remark~\ref{REM:sharp}. The case $d=m=1$ is treated in \cite[Theorem 1.8]{CG1} under less restrictive assumptions on $\mu$ than that of Theorem~\ref{THM:NLScrit}, due to the sub-critical nature of the $1d$ cubic NLS in $L^2(\T)$. See also Theorem~\ref{THM:NLSFL} below.

Theorem~\ref{THM:NLScrit} also extends the results of \cite{Tanaka} in the case $\rho>1$ which only covered the sub-critical range of regularity for the well-posedness, missing the critical one covered by Theorem~\ref{THM:NLScrit}.
\end{remark}

The results of Theorems~\ref{THM:noreg}~(i) and~\ref{THM:NLScrit} show that one has the same range of regularity exponent $s$ where well-posedness holds for both the deterministic equation \eqref{EQ:deterministic} and the modulated one \eqref{EQ}, including the scaling critical case $s=s_c$. However, in view of the discussion in the introduction, we expect that in some cases the rough modulation in \eqref{EQ} can actually improve the well-posedness theory compared to that of the deterministic equation \eqref{EQ:deterministic}. Note that this is by no means systematic, since there are counter-examples showing that for some models, it is not possible to improve on the well-posedness theory of \eqref{EQ} compared to that of \eqref{EQ:deterministic}; again we point to the work of Stewart \cite{Stewart} on this point. The point of Theorem~\ref{THM:noreg} is to show that locally in time, the space-time estimates hold for both the deterministic equation and the modulated one with same range of integrability exponents $(p,q)$. And this is sharp since we can also show that the space-time norm for the modulated dispersion is bounded below by that of the standard dispersion; see Remark~\ref{REM:sharp} below. This shows that the modulation cannot help to improve the range of admissible exponents $(p,q)$ in \eqref{Strichartz}. Thus any improvement on the well-posedness theory (in terms of larger spaces for the initial data) can only come from multilinear estimates in $X^{s,b}$-type norms. Our last result provides such an example.

\begin{theorem}\label{THM:NLSFL}
Let $\M=\T$, $\L=i\dx^2$, $\K=\C$, and ${\displaystyle \NN(u)=\pm\big(|u|^2-2\|u\|_{L^2}^2\big)u}$. Let $s\ge 0$ and $r\in[1,\infty]$. Assume that \ref{A0gr} holds for some $\gamma>\frac12$ and $\rho>\frac1{r'}$. Then the periodic Wick-ordered cubic modulated NLS \eqref{EQ} is locally well-posed in $\F L^{s,r}(\T)$. More precisely, for any $R>0$, there exists $T(R)\sim \jb{R}^{-\frac2\gamma}>0$ such that for any $u_0\in \F L^{s,r}(\T)$ with $\|u_0\|_{\F L^{s,r}}\le R$, there exists a unique mild solution $u\in C([0,T);\F L^{s,r}(\T))\cap X^{s,r}_T$ to \eqref{EQ} on $[0,T)$. Moreover, the flow map $u_0 \in \F L^{s,r}(\T)\mapsto u\in C([0,T);\F L^{s,r}(\T))\cap X^{s,r}_T$ is locally Lipschitz continuous.
\end{theorem}

The definition of the Fourier-Lebesgue space $\F L^{s,r}(\T)$ is recalled in Section~\ref{SEC:Prelim} below. In this more general Banach scale, the scaling critical regularity is $s_c=-\frac1r$, $r\in[1,\infty]$. Thus Theorem~\ref{THM:NLSFL} is in the same spirit as Theorem~\ref{THM:NLScrit}, showing that the modulated equation \eqref{EQ} is also well-posed at sub-critical and critical regularity. Note however that the usual cubic Wick-ordered periodic NLS \eqref{EQ:deterministic} is only known to be locally well-posed in $\F L^{s,r}(\T)$ for any $s\ge 0$ and $r\in[1,\infty)$ \cite{GrunrockHerr}; see also \cite{OhWangFLNF,OhWangFLGWP}. Since $L^2(\T) \varsubsetneq \F L^{0,r}(\T)$ when $r>2$, the ``Wick renormalization'' $|u|^2u-2\|u\|_{L^2}^2u$ of the usual cubic nonlinearity $|u|^2u$ is necessary to have existence of solutions for data outside of $L^2(\T)$, as pointed out in \cite{GuoOh}, while in $L^2$ solutions of the usual cubic equation are converted to solution of the Wick-ordered equation via the gauge transform $u(t)\mapsto e^{it2\|u_0\|_{L^2}^2}u(t)$ due to the conservation of the $L^2$ norm. Thus, in the case $r=2$, Theorem~\ref{THM:NLSFL} recovers \cite[Theorem 1.8]{CG1} and addresses the case $d=m=1$ missing in Theorem~\ref{THM:NLScrit} above. 

In the scale of Fourier-Lebesgue spaces, local well-posedness in the critical space $\F L^{0,\infty}(\T)$ is still a challenging open problem for the deterministic equation \eqref{EQ:deterministic}. Thus Theorem~\ref{THM:NLSFL} shows a strong regularization effect of the noisy dispersion in \eqref{EQ} compared to its deterministic counterpart \eqref{EQ:deterministic} since the well-posedness result of Theorem~\ref{THM:NLSFL} includes the critical space $\F L^{0,\infty}(\T)$.

We also point out that, by basic properties of the Fourier transform, any finite Radon measure $u_0\in\mathcal{M}(\T)$ belongs to $\F L^{0,\infty}(\T)$, so that Theorem~\ref{THM:NLSFL} provides in particular local existence, uniqueness and stability for measure initial data. However, our argument only propagates the $\F L^{0,\infty}$ regularity of the initial data, and it is an interesting open question whether one could achieve local well-posedness in $\mathcal{M}(\T)$ itself, namely also guaranteeing that $u(t)$ remains a finite measure for $t\in[0,T)$.

Note also that this improvement on the well-posedness for the modulated equation \eqref{EQ} compared to the deterministic one \eqref{EQ:deterministic} is somehow surprising in view of the counter-example  provided by Stewart \cite{Stewart} showing that such an improvement is not possible in the case of the quintic periodic NLS.

\begin{remark}\label{REM:reg}
\rm
In \cite{ModulatedNonRes}, we show an even stronger regularization effect for modulated dispersive PDEs enjoying a strong non-resonant condition, similar to the result in \cite{CG2}. Namely, for these models, well-posedness holds at \emph{any} (in particular super-critical) regularity $s$ provided that $W_t$ is \emph{irregular} enough (i.e. its occupation measure is regular enough). This in particular applies to the periodic Wick-ordered cubic fractional NLS with slightly more dispersion than the usual NLS, namely $\L=iD^\alpha$ for $\alpha>2$ instead of $\alpha=2$. Still, Theorem~\ref{THM:NLSFL} shows a regularization by noise effect for a model which is \emph{not} strongly non-resonant.
\end{remark}

\textbf{Structure of the paper}
After recalling the notations and definitions of the initial data and solution spaces in Section~\ref{SEC:Prelim}, we give the proof of Theorem~\ref{THM:noreg} in Section~\ref{SEC:Noreg}, and in particular we detail how the various models of Example~\ref{EX:noreg} fit into the framework of Theorem~\ref{THM:noreg} in Subsection~\ref{SUBS:Example1}. Then we give the proofs of Theorems~\ref{THM:NLScrit} and~\ref{THM:NLSFL} in Sections~\ref{SEC:NLScrit} and~\ref{SEC:NLSFL}, respectively.
\begin{ackno}

\rm 
The author would like to thank Anne de Bouard, Arnaud Debussche, Tadahiro Oh, Nikolay Tzvetkov, and Yuzhao Wang for helpful discussions. The author is also grateful to the anonymous referees, whose careful reading and numerous remarks greatly improved the manuscript. The author was partially supported by the ANR project Smooth ANR-22-CE40-0017 and the PEPS JCJC project BFC n°272076.

\end{ackno}

\section{Preliminaries}\label{SEC:Prelim}
\subsection{Notations}\label{SUBS:notations}
For complex numbers $a\in\C$, we write $a^-=\cj a$ and $a^+=a$. For positive numbers $A$ and $B$, we write $A\les B$ if there is a constant $C>0$, independent of the various parameters, such that $A\le CB$; $C=10^{42}$ should work throughout this text. We also write $A\sim B$ if both $A\les B$ and $B\les A$. We also write $A\vee B = \max(A,B)$ and $A\wedge B=\min(A,B)$.

In the following, we use $N,N_j,M\in 2^{\N}$ to denote dyadic integers in $2^{\N}=\{2^m,~~m\in \N\}$. For $N\in 2^\N$ a dyadic integer, we first define $\P_{\le N}$ to be a smooth version of the Dirichlet projection onto the frequencies $\{|\xi|\leq N\}\subset \widehat{\M}$. Namely, take a smooth even non-increasing cut-off $\psi_0\in C_0^{\infty}(\R)$ satisfying $\supp \psi_0 \subset [-1,1]$ and $\psi_0\equiv 1$ on $[-1/2,1/2]$, then $\P_{\leq N}$ is defined as the Fourier multiplier with symbol $\psi_0\big(N^{-1}\xi\big)$, $\xi\in\ft\M$. Then we define
$$\P_N = \begin{cases}
\P_{\le 1},& N=1,\\
\P_{\le N}-\P_{\le \frac{N}2},& N\ge 2.
\end{cases}
$$
In other words, $\{\P_N\}_{N\in2^\N}$ is a standard (smooth) Littlewood-Paley partition of unity: for any $u\in H^{-\infty}(\M)$ it holds $u=\sum_{N\in2^\N}\P_N u$, which allows to characterize function spaces, see below.

More generally, for $S\subset\ft\M$ a set of frequencies, we define the projector $\P_S$ as the Fourier multiplier with symbol $\mathbf{1}_S(\xi)$, $\xi\in\ft\M$.

\subsection{Properties associated with the noise}
In this whole manuscript, we assume that $W:\R\to\R$ is continuous, and we write $\mu$ for its occupation measure \eqref{mu}. We will always assume either \ref{A0}, \ref{A0*} or \ref{A0gr} holds, and in particular that $d\mu_{[0,T]}$ is absolutely continuous with respect to the Lebesgue measure $dz$ for any $T>0$. We will make heavy use of the following \textbf{occupation time formula}; see \cite[Theorem 6.4]{Review}.
\begin{proposition}
For every non-negative measurable function $F:\R\to [0,\infty)$ and $T>0$, it holds
\begin{align}\label{occupation}
\int_0^TF(W_t)dt = \int_\R F(z)\frac{d\mu_{[0,T]}}{dz}(z)dz.
\end{align}
\end{proposition}

\subsection{Function spaces}\label{SUBS:space}
\subsubsection{\textbf{Spatial function spaces}}
Hereafter, we will use the Sobolev and Besov spaces $W^{s,p}(\M)$ and $B^{s}_{p,q}(\M)$, $s\in\R$, $1\leq p,q\leq \infty$, which are defined via the norms
\begin{align*}
\|u\|_{W^{s,p}} \deff \big\|\jb{D}^su\big\|_{L^p},
\end{align*}
where $\jb{D}^s$ is the Fourier multiplier with symbol $\jb{\xi}^s$, $\xi\in\ft\M$, 
and
\begin{align*}
\|u\|_{B^s_{p,q}} \deff \Big(\sum_{N\ge 1}N^{sq}\big\|\P_N u\big\|_{L^{p} }^q\Big)^{\frac1q}.
\end{align*}

We recall the following properties of Besov spaces; see e.g. \cite{BCD} in case $\M=\R^d$ and \cite{SNLW} in case $M$ is a smooth closed Riemannian manifold.
\begin{lemma}\label{LEM:Besov}
Let $\M=\R^d$ or $\M$ be any compact Riemannian manifold of dimension $d$ without boundary. \\
\textup{(i)} For any $s\in\R$ we have $B^s_{2,2}(\M)=H^s(\M)$, and more generally for any $2\le p <\infty$ and $\eps>0$ we have
\begin{align*}
\|u\|_{B^s_{p,\infty}(\M)}\les \|u\|_{W^{s,p}(\M)}\les \|u\|_{B^s_{p,2}(\M)} \les \|u\|_{B^{s+\eps}_{p,\infty}(\M)}.
\end{align*}
\textup{(ii)} Let $s\in\R$ and $1\leq p_1\leq p_2\leq \infty$ and $q\in [1,\infty]$. Then for any $f\in B^s_{p_1,q}(\M)$ we have
\begin{align*}
\|f\|_{B^{s-d\big(\frac1{p_1}-\frac1{p_2}\big)}_{p_2,q}(\M)}\les \|f\|_{B^s_{p_1,q}(\M)}.
\end{align*}
\textup{(iii)} Let $\alpha,\beta\in \R$ with $\alpha+\beta>0$ and $p_1,p_2,q_1,q_2\in [1,\infty]$ with
\begin{align*}
\frac1p = \frac1{p_1}+\frac1{p_2} \qquad\text{ and }\qquad \frac1q=\frac1{q_1}+\frac1{q_2}.
\end{align*}
 Then for any $f\in B^{\alpha}_{p_1,q_1}(\M)$ and $g\in B^{\beta}_{p_2,q_2}(\M)$, we have $fg \in B^{\alpha\wedge \beta}_{p,q}(\M)$, and moreover it holds
\begin{align*}
\|fg\|_{B^{\alpha\wedge\beta}_{p,q}(\M)}\les \|f\|_{B^{\alpha}_{p_1,q_1}(\M)}\|g\|_{B^{\beta}_{p_2,q_2}(\M)}.
\end{align*}
\end{lemma}

Finally, for $s\in\R$ and $r\in[1,\infty]$, we define the Fourier-Lebesgue space $\F L^{s,r}(\M)$ via the norm $$\|u\|_{\F L^{s,r}} = \|\jb{\xi}^s \widehat{u}\|_{L^r(\widehat{\M})}.$$
In particular $H^s(\M)=W^{s,2}(\M)=\F L^{s,2}(\M)$.

\subsubsection{\textbf{Adapted function spaces}}
In this subsection we recall the definition and properties of $U^2/V^2$-type spaces associated with the flow of the linear dispersive equation with modulated dispersion.

These spaces were first used in \cite{KT05,KT07} to study the Cauchy problem for some dispersive PDEs, and we refer to \cite{KTV,HHK,HTT,Schippa} and in particular \cite{CHT} for the proofs of their properties that we list below. We give the necessary adaptations to the case of Banach-valued functions as in \cite{KTV} since we consider functions of $t$ with values in $\F L^{s,r}(\M)$.

\begin{definition}\label{DEF:UpVp}
\rm
Let $s\in\R$, $r\in[1,\infty]$, $1\le p<\infty$, and $I=[a,b)\subset \R$ with $-\infty\le a<b\le\infty$. Let then $\mathcal{Z}(I)$ be the collection of finite non-decreasing sequences $\{ t_k \}_{k=0}^K$ in $I$ with $t_0=a$ and $t_K=b$; in other words, $\mathcal{Z}(I)$ is the set of finite partitions of $I$.\\
\textup{(i)} We define $V^p(I)\F L^{s,r}(\M)$ as the space of functions $u: I \to \F L^{s,r}(\M)$ such that $u(t)$ has a limit as $t\searrow a$ and also ${\displaystyle \lim_{t\nearrow b}u(t)=0}$, endowed with the norm
\begin{align*}
\| u \|_{V^p\F L^{s,r}} = \sup_{ \{t_k \}_{k=0}^K \in \mathcal{Z}(I)}\Big(  \sum_{k=1}^K \|u(t_k) - u(t_{k-1})\|_{\F L^{s,r}}^p \Big)^{\frac1p}.
\end{align*}
Note that $V^p(I)\F L^{s,r}(\M)$ is nothing but a closed subspace of the classical space of functions $u:I\to\F L^{s,r}(\M)$ with finite $p$-variation, singled out by the normalization conditions at the endpoints of $I$ in Definition~\ref{DEF:UpVp}~(i).\\
\textup{(ii)} We define a $U^p(I)\F L^{s,r}(\M)$-atom to be a piecewise continuous function $A:I\to \F L^{s,r}(\M)$ such that
\begin{align*}
A(t) = \sum_{k=1}^{K}\ind_{[t_{k-1},t_{k})}\phi_{k}
\end{align*}
for some partition $\{t_{k}\}_{k=0}^K\in\ZZ(I)$ and collection $\{\phi_{k}\}_{k=1}^K\in \F L^{s,r}(\M)^K$ such that ${\displaystyle \sum_{k=1}^K\|\phi_{k}\|_{\F L^{s,r}}^p\le 1}$. Note that $A$ is right-continuous with ${\displaystyle \lim_{t\searrow a}A(t)=0}$ and has a limit as $t\nearrow b$.

Then $U^p(I)\F L^{s,r}(\M)$ is the space of functions $u:I\to \F L^{s,r}(\M)$ such that
\begin{align*}
u = \sum_{j=0}^\infty \lambda_j A_j,
\end{align*}
with $\{\lambda_j\}_j\in\ell^1(\N;\K)$, and $A_j$ are $U^p(I)\F L^{s,r}(\M)$-atoms, endowed with the norm
\begin{align*}
\|u\|_{U^p(I)\F L^{s,r}}=\inf\Big\{\sum_{j=0}^\infty|\lambda_j|,~~u=\sum_{j=0}^\infty\lambda_jA_j,~~A_j\text{ are $U^p(I)\F L^{s,r}$-atoms}\Big\}.
\end{align*}

\end{definition}

These spaces enjoy the following properties; again, we refer to \cite{KTV,HHK,HTT,CHT,Schippa} for the proofs.
\begin{proposition}\label{PROP:UpVp0}
The following properties hold:\\
\textup{(i)} $U^p(I)\F L^{s,r}(\M)$ and $V^p(I)\F L^{s,r}(\M)$ are Banach spaces.\\
\textup{(ii)} If $u\in U^p([-\infty,\infty))\F L^{s,r}(\M)$, $a>-\infty$, and $u(a)=0$, then $\mathbf{1}_I(t)u\in U^p(I)\F L^{s,r}(\M)$ and $\|u\|_{U^p(I)\F L^{s,r}}=\|\mathbf{1}_I u\|_{U^p([-\infty,\infty))\F L^{s,r}}$.\\
Similarly, if $u\in V^p([-\infty,\infty))\F L^{s,r}(\M)$, $b<\infty$, and $u(b)=0$, then $\mathbf{1}_I(t)u\in V^p(I)\F L^{s,r}(\M)$ and $\|u\|_{V^p(I)\F L^{s,r}}+\|u(a)\|_{\F L^{s,r}}=\|\mathbf{1}_I u\|_{V^p([-\infty,\infty))\F L^{s,r}}$.\\
\textup{(iii)} For any $1\le p<q<\infty$, there are the continuous embeddings $
U^p(I)\F L^{s,r}(\M)\subset U^q \F L^{s,r}(\M)\subset L^\infty(I;\F L^{s,r}(\M))$ and $V^p(I)\F L^{s,r}(\M)\subset V^q(I)\F L^{s,r}(\M)\subset L^\infty(I;\F L^{s,r}(\M))$.\\
\textup{(iv)} If $u\in U^p(I)\F L^{s,r}(\M)$ with $u(b)=0$, then $u\in V^p(I)\F L^{s,r}(\M)$ and $\|u\|_{V^p(I)\F L^{s,r}}\les \|u\|_{U^p(I)\F L^{s,r}}$.\\
Similarly, for any $1\le p<q<\infty$, if $u\in V^p(I)\F L^{s,r}(\M)$ is right-continuous with $u(a)=0$, then $u\in U^q(I)\F L^{s,r}(\M)$ and $\|u\|_{U^q(I)\F L^{s,r}}\les \|u\|_{V^p(I)\F L^{s,r}}$.
\end{proposition}

\begin{definition}
\rm
We define the Banach space $$DU^2(I)\F L^{s,r}(\M)=\big\{\dt u\,:\,u\in U^2(I)\F L^{s,r}(\M)\big\},$$
where the derivative is taken in the sense of distributions, endowed with the norm $\|f\|_{DU^2(I)\F L^{s,r}}=\|u\|_{U^2(I)\F L^{s,r}}$. The condition $u(a)=0$ for $u\in U^p(I)\F L^{s,r}(\M)$ guarantees that for $f\in DU^2(I)\F L^{s,r}(\M)$ there is a unique choice of $u\in U^2(I)\F L^{s,r}(\M)$ such that $f=\dt u$.
\end{definition}

The main result regarding this space is the duality relation between $DU^2(I)\F L^{s,r}(\M)$ and $V^2(I)\F L^{-s,r'}(\M)$; see \cite{KTV} for the Banach-valued version.

\begin{proposition}\label{PROP:duality}
For all $f\in L^1(I;\F L^{s,r}(\M))\subset DU^2(I)\F L^{s,r}(\M)$, it holds
\begin{align*}
\|f\|_{DU^2(I)\F L^{s,r}}=\sup_{\|v\|_{V^2(I)\F L^{-s,r'}}\le 1}\Big|\int_I\int_\M f(t,x)\cj{v}(t,x)dxdt\Big|.
\end{align*}
\end{proposition}

As mentioned in the introduction, the adapted solution space $X^{s,r}_T$ for $u$ solving \eqref{EQ} is the replacement of the solution space $D^W(H^s)$ used by Chouk and Gubinelli in the context of modulated Schr\"odinger and KdV equations; see \cite[Definition 2.2]{CG1}. 
\begin{definition}\label{DEF:UpVpW}
\rm
Let $s\in\R$, $r\in[1,\infty]$, and $1 \leq p <\infty$. We define
\begin{align*}
X^{s,p,r}(I) = \big\{u:I\times\M\to\K : e^{W_t\L}u\in U^p(I)\F L^{s,r}(\M)\big\}
\end{align*}
and 
\begin{align*}
Y^{s,p,r}(I) = \big\{ u: I \times \M \to \K : e^{W_t\L} u \in V^p(I) \F L^{s,r}(\M) \big\}
\end{align*}
with norms given by
\begin{align*}
\|u\|_{X^{s,p,r}(I)}=\|e^{W_t\L}u\|_{U^p(I)\F L^{s,r}}\text{ and }\|u\|_{Y^{s,p,r}(I)}=\|e^{W_t\L}u\|_{V^p(I)\F L^{s,r}}.
\end{align*}

When $I=[0,T)$ for some $T>0$, we simply write $X^{s,p,r}(I)=X^{s,p,r}_T$ and $Y^{s,p,r}(I)=Y^{s,p,r}_T$. 

Finally, when $p=2$, we simply write $X^{s,r}_T$ and $Y^{s,r}_T$, and when $r=2$, we just write $X^s_T$, $Y^s_T$.

\end{definition}

\begin{remark}
\rm
Since $U^2\F L^{s,r}(\M)\subset L^\infty(\R;\F L^{s,r}(\M))$ continuously, we have that for any $T>0$, $C([0,T);\F L^{s,r}(\M))\cap X^{s,r}_T $ is a closed subspace of $C([0,T);\F L^{s,r}(\M))$.
\end{remark}

\begin{proposition}\label{PROP:UpVp}
Let $s\in\R$ and $r\in[1,\infty]$. Then the following hold:\\
\textup{(i)} We have the continuous embeddings
\begin{align*}
X^{s,r}(I)\subset Y^{s,r}(I)\subset X^{s,p,r}(I)
\end{align*}
for any $2<p<\infty$.

\noi
\textup{(ii)} For any $T>0$ and $\phi\in \F L^{s,r}(\M)$, we have $e^{-(W_t-W_0)\L}\phi\in X^{s,r}_T$ and
\begin{align*}
\|e^{-(W_t-W_0)\L}\phi\|_{X^{s,r}_T}\le \|\phi\|_{\F L^{s,r}}.
\end{align*}

\noi
\textup{(iii)} For any $T>0$ and $F\in L^1([0,T);\F L^{s,r}(\M))$, we have $t\mapsto \ind_{[0,\infty)}(t)\int_0^te^{-(W_t-W_s)\L}F(s)ds\in X^{s,r}_T$ and
\begin{align*}
\Big\|\int_0^te^{-(W_t-W_s)\L}F(s)ds\Big\|_{X^{s,r}_T}\le \sup_{\substack{w\in Y^{-s,r'}_T\\\|w\|_{Y^{-s,r'}_T}\le 1}}\Big|\int_0^T\int_\M F(t,x)\cj w(t,x)dxdt\Big|.
\end{align*}

\noi
\textup{(iv)} (Transference principle) If for some spatial function space $X$, some $s\in\R$, $r\in[1,\infty]$, $p\in[1,\infty)$, $\iota\in\{\pm\}$, and any $T\in(0,1]$, there is $C_{X,s,r,p,\iota}(T)>0$ non-decreasing in $T$ such that
\begin{align*}
\|e^{-W_t\L}\phi\|_{S^{p,\iota}_T(X)}\le  C_{X,s,r,p,\iota}(T)\|\phi\|_{\F L^{s,r}}
\end{align*}
holds for any $\phi\in\F L^{s,r}(\M)$, then
\begin{align}\label{transference}
\|u\|_{S^{p,\iota}_T(X)}\le C_{X,s,r,p,\iota}(T)\|u\|_{X^{s,p,r}_T}
\end{align}
for any $u\in X^{s,p,r}([0,T))$.
\end{proposition}

\begin{proof}
The proof of (i) is straightforward from the definition of $X^{s,r}_T$ and $Y^{s,r}_T$ together with the embedding $V^2(I)\subset U^p(I)$ of Proposition~\ref{PROP:UpVp0}~(iv). Note that the latter requires $u$ to be right-continuous with $u(a)=0$, but this is not a restriction here. Indeed, given $v\in V^2(I)\F L^{s,r}(\M)$, we may first replace $v$ by its right-continuous modification, which changes $v$ at (at most countably many) discontinuity points, hence does not increase its $V^2$ norm nor affect any of the space-time integrals $\int_I\int_\M f\cj{v}\,dxdt$ appearing in the duality arguments below. Then, writing $v=v(a)\ind_I+\big(v-v(a)\ind_I\big)$, the first term is a multiple of a one-step $U^p$-atom, with $\|v(a)\ind_I\|_{U^p\F L^{s,r}}=\|v(a)\|_{\F L^{s,r}}\le\|v\|_{V^2\F L^{s,r}}$ in view of the normalization ${\displaystyle\lim_{t\nearrow b}v(t)=0}$ in Definition~\ref{DEF:UpVp}~(i), while the second term vanishes at $t=a$ and has the same $V^2$ norm as $v$, so that Proposition~\ref{PROP:UpVp0}~(iv) applies to it.

Similarly, (ii) follows directly from the definition of the space $X^{s,r}_T$.

The property (iii) is fundamental and established e.g. in \cite{HHK,HTT} for Hilbert-valued functions of $t$. Here we give a straightforward adaptation of the proof to the case of $\F L^{s,r}(\M)$-valued function. 
Indeed, extending ${\displaystyle t\mapsto \int_0^te^{W_s\L}F(s)ds}$ by ${\displaystyle \int_0^Te^{W_s\L}F(s)ds}$ for $t\ge T$, since $t\mapsto e^{W_t\L}F(t)\in L^1([0,T];\F L^{s,r}(\M))$, we have by a $\F L^{s,r}(\M)$-valued Riemann sum approximation that ${\displaystyle t\mapsto \int_0^te^{W_s\L}F(s)ds}$ is in $U^2([0,T))\F L^{s,r}(\M)$, and so $e^{W_t\L}F(t)\in DU^2([0,T))\F L^{s,r}(\M)$ with ${\displaystyle e^{W_t\L}F(t)=\dt\Big(\int_0^te^{W_s\L}F(s)ds\Big)}$. This shows that
$$\Big\|\int_0^te^{-(W_t-W_s)\L}F(s)ds\Big\|_{X^{s,r}_T}=\big\|e^{W_t\L}F\big\|_{DU^2([0,T))\F L^{s,r}},$$
and the conclusion of (iii) follows from Proposition~\ref{PROP:duality}.

As for (iv), we have from the atomic structure of $U^p\F L^{s,r}(\M)$ (Definition~\ref{DEF:UpVp}) and the definition of $X^{s,p,r}(I)$ (Definition~\ref{DEF:UpVpW}) that if
\begin{align*}
u=\sum_{j=0}^\infty\lambda_ja_j\text{ with }a_j(t)=\sum_{k=1}^{K_j}\ind_{[t_{j,k-1},t_{j,k})}e^{-W_t\L}\phi_{j,k}
\end{align*}
for some $\{t_{j,k}\}_{k=0}^{K_j-1}\in\ZZ(T)$ and $\{\phi_{j,k}\}_{k=1}^{K_j}\in \F L^{s,r}(\M)^{K_j}$ such that ${\displaystyle \sum_{k=1}^{K_j}\|\phi_{j,k}\|_{\F L^{s,r}}^p=1}$, and some $\{\lambda_j\}_j\in\ell^1(\N)$, then
\begin{align*}
\|u\|_{S^{p,+}(X)}&\le\sum_{j=0}^\infty|\lambda_j|\Big(\sum_{k=1}^{K_j}\int_{t_{j,k-1}}^{t_{j,k}}\big\| e^{-W_t\L}\phi_{j,k}\big\|_{X}^pdt\Big)^\frac1p\\
&\le C_{X,s,r,p,\iota}(T)\sum_{j=0}^\infty|\lambda_j|\Big(\sum_{k=1}^{K_j}\|\phi_{j,k}\|_{\F L^{s,r}}^p\Big)^\frac1p\\
&\le C_{X,s,r,p,\iota}(T)\sum_{j=0}^\infty|\lambda_j|.
\end{align*}
We have a similar bound if we replace $L^p([0,T];X(\M))$ by $X(\M;L^p([0,T]))$. This proves \eqref{transference} in view of Definition~\ref{DEF:UpVpW} and Definition~\ref{DEF:UpVp}~(ii).
\end{proof}

\section{Local well-posedness for general models: proof of Theorem~\ref{THM:noreg}}\label{SEC:Noreg}

We now present the proof of Theorem~\ref{THM:noreg}.

\subsection{Case (i)}
We assume that \ref{A0} holds, and that we have the space-time linear estimate \eqref{Strichartz} together with the non-linear estimate \eqref{Strichartznonlinear}. We first treat the case of Strichartz estimates, namely \eqref{Strichartz} holds for some $(p,\zeta,\iota)\in\EE$ with $\iota=+$ (recall that $p<\infty$ by assumption). This implies that for any $u_0\in L^2(\M)$, the map ${\displaystyle z\mapsto \|e^{z\L}u_0\|_{X_\zeta}^p}$ is integrable on $[0,T]$, in particular measurable and positive. Then we can use the occupation time formula \eqref{occupation} together with \ref{A0} and \eqref{Strichartz}. Indeed, letting $I_W(T):=[W_0-2\|W\|_{L^\infty_T},W_0+2\|W\|_{L^\infty_T}]$, for all $t\in[0,T]$ it holds $W_t\in I_W(T)$, we thus get that for any $(p,\zeta,+)\in\mathcal{E}$:
\begin{align*}
\big\|e^{-(W_t-W_0)\L}u_0\big\|_{L^p_TX_\zeta}^p&=\int_0^T\mathbf{1}_{|W_t-W_0|\le 2\|W\|_{L^\infty_T}}\big\|e^{-(W_t-W_0)\L}u_0\big\|_{X_\zeta}^p dt\\
&= \int_{I_W(T)}\big\|e^{-z\L}e^{W_0\L}u_0\big\|_{X_\zeta}^p d\mu_{[0,T]}(z)\\
&\le \big\|\frac{d\mu_{[0,\cdot]}}{dz}(\cdot)\big\|_{L^\infty([0,T]\times I_W(T))}\big\|e^{-z\L}e^{W_0\L}u_0\big\|_{L^p_{I_W(T)}X_\zeta}^p\\
&\le \big\|\frac{d\mu_{[0,\cdot]}}{dz}(\cdot)\big\|_{L^\infty([0,T]\times I_W(T))}C_{p,\zeta,+}(|I_W(T)|)^p\|e^{W_0\L}u_0\|_{L^2}^p\\
&= \big\|\frac{d\mu_{[0,\cdot]}}{dz}(\cdot)\big\|_{L^\infty([0,T]\times I_W(T))}C_{p,\zeta,+}(|I_W(T)|)^p\|u_0\|_{L^2}^p,
\end{align*}
where in the last step we used that $e^{W_0\L}$ is an isometry of $L^2(\M)$, $\L$ being skew-adjoint.
Similarly, if $L^p_TX_\zeta$ is replaced by $X_\zeta L^p_T$, we get for any fixed $x\in\M$
\begin{align*}
\big\|e^{-(W_t-W_0)\L}u_0(x)\big\|_{L^p_T}^p&=\int_0^T\mathbf{1}_{I_W(T)}(W_t)\big|e^{-(W_t-W_0)\L}u_0(x)\big|^p dt\\
&= \int_{I_W(T)}\big|e^{-z\L}e^{W_0\L}u_0(x)\big|^p d\mu_{[0,T]}(z)\\
&\le \big\|\frac{d\mu_{[0,\cdot]}}{dz}(\cdot)\big\|_{L^\infty([0,T]\times I_W(T))}\big\|e^{-z\L}e^{W_0\L}u_0(x)\big\|_{L^p_{I_W(T)}}^p.
\end{align*}
Taking the $\frac{1}p$-th power then the $X_\zeta$-norm, together with \eqref{Strichartz} this yields
\begin{align*}
\big\|e^{-(W_t-W_0)\L}u_0(x)\big\|_{X_\zeta L^p_T}&\le \big\|\frac{d\mu_{[0,\cdot]}}{dz}(\cdot)\big\|_{L^\infty([0,T]\times I_W(T))}^\frac1p\big\|e^{-z\L}e^{W_0\L}u_0(x)\big\|_{X_\zeta L^p_{I_W(T)}}\\
&\le \big\|\frac{d\mu_{[0,\cdot]}}{dz}(\cdot)\big\|_{L^\infty([0,T]\times I_W(T))}^\frac1p C_{p,\zeta,-}(|I_W(T)|)\|e^{W_0\L}u_0\|_{L^2}\\
&= \big\|\frac{d\mu_{[0,\cdot]}}{dz}(\cdot)\big\|_{L^\infty([0,T]\times I_W(T))}^\frac1p C_{p,\zeta,-}(|I_W(T)|)\|u_0\|_{L^2},
\end{align*}
using again the skew-adjointness of $\L$ in the last step.

All in all we obtain
\begin{align}\label{StrichartzW}
\big\|\jb{D}^se^{-(W_t-W_0)\L}u_0\big\|_{S^{p,\iota}_TX_\zeta}\le C_{p,\zeta,\iota}^W(T)\|u_0\|_{H^s},
\end{align}
for all $(p,\zeta,\iota)\in\mathcal{E}$, where
\begin{align}\label{CW}
C_{p,\zeta,\iota}^W(T)=\big\|\frac{d\mu_{[0,\cdot]}}{dz}(\cdot)\big\|_{L^\infty([0,T]\times I_W(T))}^\frac1p C_{p,\zeta,\iota}(|I_W(T)|).
\end{align}

From the definition \eqref{Sb} of $\Sb^s_T$, the transference principle of Proposition~\ref{PROP:UpVp}~(iv) and the embeddings of Proposition~\ref{PROP:UpVp0}~(ii), we deduce that
\begin{align}\label{StrichartzX}
\|u\|_{\mathbf{S}^s_T}\le C_W(T)\|u\|_{X^{s,p_\mathcal{E},2}_T},
\end{align}
where
\begin{align}\label{CW2}
C_W(T)=\max_{(p,\zeta,\iota)\in\EE}C_{p,\zeta,\iota}^W(T),
\end{align}
and
\begin{align*}
p_\mathcal{E}=\min \big\{p> 2,~~\exists (\zeta,\iota)\in E\times\{\pm\},~~(p,\zeta,\iota)\in\mathcal{E}\big\}.
\end{align*}
Recall that $p_\mathcal{E}>2$ from our assumption that $\mathcal{E}\subset (2,\infty)\times E\times\{\pm\}$ is finite.

From the nonlinear estimate \eqref{Strichartznonlinear}, the linear estimate \eqref{StrichartzX}, and the embeddings of Propositions~\ref{PROP:UpVp0}~(ii) and~\ref{PROP:UpVp}~(iv), we get 
\begin{align}\label{nonlineartransfer}
&\Big|\int_0^T\int_\M \big(\NN(u)-\NN(v)\big)\cj w dxdt\Big|\notag\\
&\qquad\le CT^{\eta_\EE(s)} \|u-v\|_{\mathbf{S}^s_T}\big(\|u\|_{\mathbf{S}^s_T}+\|v\|_{\mathbf{S}^s_T}\big)^{n}\|w\|_{\Sb^{-s}_T}\notag\\
&\qquad\le CT^{\eta_\EE(s)}C_W(T)^{n+2} \|u-v\|_{X^{s,p_\mathcal{E},2}_T}\big(\|u\|_{X^{s,p_\mathcal{E},2}_T}+\|v\|_{X^{s,p_\mathcal{E},2}_T}\big)^{n}\|w\|_{X^{-s,p_\mathcal{E},2}_T}\notag\\
&\qquad\le CT^{\eta_\EE(s)}C_W(T)^{n+2} \|u-v\|_{X^s_T}\big(\|u\|_{X^s_T}+\|v\|_{X^s_T}\big)^{n}\|w\|_{Y^{-s}_T}.
\end{align}
Note that in the last step we used that $p_\EE>2$ to estimate $\|w\|_{X^{-s,p_\mathcal{E},2}_T}\les \|w\|_{Y^{-s}_T}$.

Then, for $R,T>0$, we define the ball 
$$B_R(T):=\big\{u\in C([0,T);H^s(\M))\cap X^s_T,~~\|u\|_{X^s_T}\le R\big\},$$
which is closed in $X^s_T$, hence a complete metric space when endowed with the metric inherited from $\|\cdot\|_{X^s_T}$.

Given $u_0\in H^s(\M)$, we set $R=2A\|u_0\|_{H^s}$ for some constant $A\ge \max(1,C_W(T))$ to be chosen later, and consider the map
$$\Phi_{u_0} : u\in B_R(T)\mapsto e^{-(W_t-W_0)\L}u_0 - \int_0^te^{-(W_t-W_{t'})\L}\NN(u)(t')dt'.$$

By Proposition~\ref{PROP:UpVp}~(ii), we have
\begin{align*}
\big\|\Phi_{u_0}(u)-\Phi_{v_0}(v)\big\|_{X^s_T}&\leq \|u_0-v_0\|_{H^s}+\Big\|\int_0^te^{-(W_t-W_{t'})\L}\big(\NN(u)-\NN(v)\big)dt'\Big\|_{X^s_T}.
\end{align*}

Now, in order to be able to use \eqref{nonlineartransfer}, we wish to invoke Proposition~\ref{PROP:UpVp}~(iii) to estimate the nonlinear term above. However, the right-hand side of the estimate of Proposition~\ref{PROP:UpVp}~(iii), given by Proposition~\ref{PROP:duality}, only holds if the forcing term is in $L^1([0,T);H^s(\M))$. Therefore, we first write
\begin{align*}
\Big\|\int_0^te^{-(W_t-W_{t'})\L}\big(\NN(u)-\NN(v)\big)dt'\Big\|_{X^s_T}=\lim_{N\to\infty}\Big\|\int_0^te^{-(W_t-W_{t'})\L}\P_{\leq N}\big(\NN(u)-\NN(v)\big)dt'\Big\|_{X^s_T}
\end{align*}
by Definition~\ref{DEF:UpVpW} of $X^s_T$, and it is then enough to show that for each fixed $N$, $\P_{\leq N}\big(\NN(u)-\NN(v)\big)\in L^1([0,T),H^s(\M))$. By duality, 
\begin{align*}
\Big\|\P_{\leq N}\big(\NN(u)-\NN(v)\big)\Big\|_{L^1_T H^s}&\leq\sup_{\|w\|_{L^\infty_T H^{-s}}\leq 1}\Big|\int_0^T\int_\M\big(\NN(u)-\NN(v)\big) \cj{\P_{\leq N}w}dxdt\Big|,
\end{align*}
and using \eqref{Strichartznonlinear} and \eqref{StrichartzX}, we find
\begin{align*}
&\Big|\int_0^T\int_\M\big(\NN(u)-\NN(v)\big) \cj{\P_{\leq N}w}dxdt\Big|\\
&\hspace*{2cm}\leq C_0T^{\eta_\EE(s)}\|u-v\|_{\Sb^s_T}\big(\|u\|_{\Sb^s_T}+\|v\|_{\Sb^s_T}\big)^n\|\P_{\leq N}w\|_{\Sb^{-s}_T}\\
&\hspace*{2cm}\leq C_0T^{\eta_\EE(s)}C_W(T)^{n+1}\|u-v\|_{X^s_T}\big(\|u\|_{X^s_T}+\|v\|_{X^s_T}\big)^n\|\P_{\leq N}w\|_{\Sb^{-s}_T}.
\end{align*}
Thus it is enough to show that 
\begin{align*}
\|\P_{\leq N}w\|_{\Sb^{-s}_T}\les_N \|w\|_{L^\infty_T H^{-s}},
\end{align*}
i.e. for all $(p,\zeta,\iota)\in\EE$,
\begin{align}\label{StrichartzWeak}
\|\jb{D}^{-s}\P_{\leq N}w\|_{S^{p,\iota}_T(X_\zeta)}\les_N \|w\|_{L^\infty_T H^{-s}}.
\end{align}
In case $\iota=+$, since $X_\zeta\subset H^{-\infty}(\M)$, we have
\begin{align*}
\|\P_{\leq N}w\|_{L^p_T(X_\zeta)}\les_N T^\frac1p \|\P_{\leq N}w\|_{L^\infty_T L^2_x},
\end{align*}
which gives \eqref{StrichartzWeak}. Similarly, using also Minkowski's integral inequality since $p>2$ due to $\EE\subset (2,\infty)\times E\times\{\pm\}$ finite, we have
\begin{align*}
\|\P_{\leq N}w\|_{X_\zeta L^p_T}\les_N \|\P_{\leq N}w\|_{L^2_x L^p_T}\les_N T^\frac1p \|\P_{\leq N}w\|_{L^\infty_T L^2_x},
\end{align*}
which is also enough for \eqref{StrichartzWeak}. This shows that $\P_{\leq N}\big(\NN(u)-\NN(v)\big)\in L^1([0,T];H^s(\M))$ for any fixed $N$.

Therefore, we are entitled to use Proposition~\ref{PROP:UpVp}~(iii), which together with \eqref{nonlineartransfer} gives
\begin{align*}
&\Big\|\int_0^te^{-(W_t-W_{t'})\L}\P_{\leq N}\big(\NN(u)-\NN(v)\big)dt'\Big\|_{X^s_T}\\
&\hspace*{2cm}\leq \sup_{\|w\|_{Y^{-s}_T}\leq 1}\Big|\int_0^T\int_\M \big(\NN(u)-\NN(v)\big)\cj{\P_{\leq N}w}dxdt\Big|\\
&\hspace*{2cm}\leq C T^{\eta_\EE(s)}C_W(T)^{n+2}\|u-v\|_{X^s_T}\big(\|u\|_{X^s_T}+\|v\|_{X^s_T}\big)^n\|\P_{\leq N}w\|_{Y^{-s}_T}.
\end{align*}

Since $\P_{\leq N}$ is uniformly (in $N$) bounded on $Y^{-s}_T$ in view of Definition~\ref{DEF:UpVpW}, putting everything together yields
\begin{align}\label{les_hirondelles_senvolent}
\big\|\Phi_{u_0}(u)-\Phi_{v_0}(v)\big\|_{X^s_T}&\le \|u_0-v_0\|_{H^s}+CT^{\eta_\EE(s)}C_W(T)^{n+2}(2R)^{n}\|u-v\|_{X^s_T}
\end{align}
for any $u,v\in B_R(T)$.

\bigskip
\textit{Sub-critical case: $\eta_\EE(s)>0$.}\\
\medskip
First, in case where the equation \eqref{EQ:deterministic} is sub-critical, we have $\eta_\EE(s)>0$, so that we can choose $A$ such that $R\ge \max(1,\|u_0\|_{H^s})$, then $T\sim R^{-\frac{n}{\eta_\EE(s)}}\in(0,1)$ and $C_W(T)\le C_W(1)$, which guarantee that $\Phi_{u_0}$ maps $B_R(T)$ to $B_R(T)$. With our choice of $R$ and $T$ as above, we deduce that $\Phi_{u_0}$ is also a contraction on $B_R(T)$, thus having a unique fixed point $u\in B_R(T)$. If $v\in C([0,T);H^s(\M))\cap X^s_T$ is another solution, with $R'=\max\big(R,\|v\|_{X^s_T}\big)$ and $\tau=\tau(R')\le T$, the argument above shows that $v\equiv u$ on $[0,\tau]$. But then iterating this argument on $[\tau,2\tau]$, etc., shows that $v\equiv u$ on the whole time interval where $u$ is defined, showing uniqueness in the whole class $C([0,T);H^s(\M))\cap X^s_T$. The estimate \eqref{les_hirondelles_senvolent} also provides the Lipschitz dependence of the flow map on the initial data.

\bigskip
\textit{Critical case: $\eta_\EE(s)=0$.}\\
\medskip
In the case of critical equations, one has $\eta_\EE(s)=0$, so the previous argument cannot exploit smallness of time to treat large initial data. Instead, we prove small data well-posedness: fix any $T>0$. Then from the previous estimates, for initial data satisfying $\|u_0\|_{H^s}\leq \delta$, we find on the ball $B_{2\delta}$ of radius $2\delta$ in $C([0,T),H^s(\M))\cap X^s_T$:
\begin{align*}
\big\|\Phi_{u_0}(u)\big\|_{X^s_T}&\le \delta+CC_W(T)^{n+2}(2\delta)^{n+1}\leq 2\delta
\end{align*} 
provided that
\begin{align*}
\delta=\delta(T)\leq (CC_W(T)^{n+2}2^{n+1})^{-\frac1n}.
\end{align*}
This shows that $\Phi_{u_0}$ maps $B_{2\delta}$ to itself, and a similar argument shows that it is a contraction on this ball. This gives small data local well-posedness on $[0,T)$, provided that $\|u_0\|_{H^s}\leq \delta(T)$ as above.

\subsection{Case (ii)}

If now \ref{A0*} holds instead of \ref{A0}, the same discussion as in the previous subsection still holds and leads again to \eqref{les_hirondelles_senvolent}. However, from \eqref{CW} and the fact that now ${\displaystyle T\mapsto\big\|\frac{d\mu_{[0,T]}}{dz}\big\|_{L^\infty_z}}$ is continuous\footnote{Note that since the occupation measure $\mu_{[0,T]}$ is supported in $\big[-\|W\|_{L^\infty([0,T])},\|W\|_{L^\infty([0,T])}\big]$, the $C^0_tL^\infty_{\textrm{loc}}$ regularity provided by \ref{A0*} is indeed enough to grant the continuity of the above map.} and vanishes at $T=0$, we see that 
\begin{align}\label{CW3}
C_W(T)\underset{T\searrow 0}{\longrightarrow}0.
\end{align}
Thus, even if $\eta_\EE(s)=0$ in the critical case, we can choose again $A$ such that $R\ge \max(1,\|u_0\|_{H^s})$, and then $T=T(\|u_0\|_{H^s})\in(0,1)$ small enough guaranteeing again that $\Phi_{u_0}$ defines a contraction on $B_R(T)$.

Finally, if $s=0$ and if \ref{A3} also holds, we claim that the mass $t\mapsto\|u(t)\|_{L^2}^2$ is a conservation law for \eqref{EQ}.
Taking this claim for granted, then from the previous argument we have sub-critical local well-posedness in the sense that for any $u_0\in L^2(\M)$, there is $T=T(\|u_0\|_{L^2})>0$ and a unique solution $u\in C([0,T];L^2(\M))\cap X^0_T$. But since $\|u(T)\|_{L^2}=\|u_0\|_{L^2}$, we can iterate the local well-posedness on $[T,2T]$, then $[kT,(k+1)T]$ for any $k\in\N$. This shows global well-posedness of \eqref{EQ} in this case.

It remains to prove the claim. To be able to use \eqref{cancellation}, we need a smooth approximation of \eqref{EQ}. For $N\in\N\cup\{\infty\}$, we define the truncated dynamics
\begin{align}\label{Galerkin}
\begin{cases}
\dt u^N + \L u^N + \P_{\leq N}\NN(\P_{\leq N}u^N) = 0,\\
u^N(0)=\P_{\leq N}u_0.
\end{cases}
\end{align}
Let us detail how the argument of the previous case provides uniform (in $N\in\N\cup\{\infty\}$) local well-posedness for \eqref{Galerkin}. Since the projections $\P_{\leq N}$ are bounded on $L^2(\M)$ and on the spaces $\Sb^{0}_T$ and $Y^0_T$, uniformly in $N$, the estimates \eqref{StrichartzW}-\eqref{nonlineartransfer} hold for the truncated nonlinearity $\P_{\leq N}\NN(\P_{\leq N}\cdot)$ with constants uniform in $N$. We then run the fixed point argument for \eqref{Galerkin} in the ball $B_R(T)\subset X^0_T$, with $R=2A\|u_0\|_{L^2}\ge 2A\|\P_{\leq N}u_0\|_{L^2}$ for some $A\ge\max(1,C_W(1))$, and with $T>0$ chosen small enough so that $CT^{\eta_\EE(0)}C_W(T)^{n+2}(2R)^{n}\le\frac12$: this is possible with a choice of $T=T(\|u_0\|_{L^2})>0$ \emph{independent of $N$}, either under \ref{A0} with $\eta_\EE(0)>0$, or under \ref{A0*} with $\eta_\EE(0)=0$ by using \eqref{CW3}. This provides a solution $u^N\in B_R(T)$ of \eqref{Galerkin} on $[0,T]$, unique in $C([0,T];L^2(\M))\cap X^0_T$, for all $N\in\N\cup\{\infty\}$, on a common time interval. Moreover, $\P_{\leq 10N}u^N=u^N$, therefore $u^N\in C([0,T];H^\infty(\M))$. Letting $v^N(t):=e^{W_t\L}u^N(t)$, we have also $v^N\in C([0,T];H^\infty(\M))$ and
\begin{align*}
\dt v^N = e^{W_t\L}\P_{\leq N}\NN(\P_{\leq N}e^{-W_t\L}v^N).
\end{align*}
Thus from \ref{A2} we see that $v^N\in C^1([0,T];H^\infty(\M))$, and we can compute
\begin{align*}
\frac{d}{dt}\|u^N(t)\|_{L^2}^2 &= \frac{d}{dt}\|v^N(t)\|_{L^2}^2 = 2\Re \langle v^N(t),e^{W_t\L}\P_{\leq N}\NN(\P_{\leq N}e^{-W_t\L}v^N(t))\rangle_{H^\infty\times H^{-\infty}}\\
& = 2\Re \langle \P_{\leq N}u^N(t),\NN(\P_{\leq N}u^N(t))\rangle_{H^\infty\times H^{-\infty}} = 0
\end{align*}
due to \eqref{cancellation} under \ref{A3}. This shows that the mass is conserved for \eqref{Galerkin}.

Then, writing $\NN(\P_{\leq N}u^N)-\NN(u) = (\NN(\P_{\leq N}u^N)-\NN(\P_{\leq N}u))+(\NN(\P_{\leq N}u)-\NN(u))$ and proceeding as for \eqref{les_hirondelles_senvolent}, we find
\begin{align*}
\|u^N-u\|_{X^0_T}&\leq \|u^N-\P_{\leq N}u\|_{X^0_T}+\|\P_{>N}u\|_{X^0_T}\\
&\leq \|\P_{>N}u_0\|_{L^2}+\frac12\|u^N-u\|_{X^0_T}+C\|\P_{>N}u\|_{X^0_T}.
\end{align*}
This shows that $u^N\to u$ in $X^0_T$ as $N\to\infty$. In particular, convergence of $\|u^N(t)\|_{L^2}=\|\P_{\leq N}u_0\|_{L^2}$ to $\|u(t)\|_{L^2}$ finally shows the conservation of mass for \eqref{EQ} under \ref{A3}.

\begin{remark}\label{REM:CW}
\rm
The gain \eqref{CW3} from the noise $W_t$ under assumption \ref{A0*} is similar in spirit to the use of the refined rescaling transform used to study mass-critical equations with multiplicative noise in \cite{BRZ,HRSZ}. While the latter relies heavily on stochastic properties of the noise, here this gain is entirely deterministic and only relies on the irregularity property of $W_t$.
\end{remark}

\begin{remark}\label{REM:sharp}
\rm
In the proof of Theorem~\ref{THM:noreg} above, the transfer of the linear space-time estimate \eqref{Strichartz} for the deterministic equation \eqref{EQ:deterministic} to the modulated one \eqref{EQ} simply follows from the occupation time formula \eqref{occupation}:
\begin{align*}
\int_0^T\big\|e^{-W_t\L}u\big\|_X^pdt = \int_{I_W(T)}\big\|e^{-z\L}u\big\|_X^p\frac{d\mu_{[0,T]}}{dz}(z)dz,
\end{align*}
from which we got an upper bound in the proof of Theorem~\ref{THM:noreg}.

In turn, for any subinterval $J\subset I_W(T)$, this also provides the \emph{lower} bound
\begin{align*}
\int_0^T\big\|e^{-W_t\L}u\big\|_X^pdt \ge \Big(\inf_{z\in J}\frac{d\mu_{[0,T]}}{dz}(z)\Big)\int_{J}\big\|e^{-z\L}u\big\|_X^p dz.
\end{align*}
Note that one cannot expect the infimum to be positive on the whole of $I_W(T)$, since the local time vanishes at the boundary of its support, but a lower bound on a small interval $J$ around $z=0$ is enough for our purpose. Indeed, for the standard Brownian motion, the local time at zero $\frac{d\mu_{[0,T]}}{dz}(0)=L_T(0)$ has the same law as $\frac12\max_{[0,T]}W_t$ \cite{RY}, which is a.s. positive, and $z\mapsto \frac{d\mu_{[0,T]}}{dz}(z)$ admits an a.s. continuous version. Thus, almost surely, there are $\eps=\eps(W,T)>0$ and $c=c(W,T)>0$ such that $\frac{d\mu_{[0,T]}}{dz}\ge c$ on $J=[-\eps,\eps]$, whence
\begin{align*}
\int_0^T\big\|e^{-W_t\L}u\big\|_X^pdt \ge c\int_{-\eps}^{\eps}\big\|e^{-z\L}u\big\|_X^p dz.
\end{align*}
Therefore, if the modulated linear flow satisfied a space-time estimate \eqref{StrichartzW} on $[0,T]$ for some pair of exponents $(p,q)$, then the deterministic flow $e^{-z\L}$ would satisfy the corresponding local-in-time estimate \eqref{Strichartz} on the (fixed, non-trivial) time interval $[-\eps,\eps]$. Since the range of admissible exponents in the \emph{local-in-time} estimate \eqref{Strichartz} does not depend on the length of the time interval (for the Schr\"odinger equation on $\R^d$, for instance, this follows from scaling and time translation invariance), one \emph{cannot} get an improvement of the range of exponents in \eqref{Strichartz} for the modulated equation compared to the deterministic one. This shows the sharpness of Theorem~\ref{THM:noreg} in case of a Brownian modulation. In the case of the Schr\"odinger equation with white noise dispersion, this goes against the intuition that from the usual dispersion inequality ${\displaystyle \|e^{i(W_t-W_\tau)\Delta}u\|_{L^\infty_x}\les |W_t-W_\tau|^{-\frac{d}2}\|u\|_{L^1_x}}$ and the scaling $W_t-W_\tau\sim \sqrt{t-\tau}$, we would have expected a better (i.e. more integrable in 0) time decay, locally. Thus, through the usual $TT^*$ argument via Hardy-Littlewood-Sobolev inequality, we would have expected a better range of exponents $(p,q)$ for the local in time Strichartz estimate with white noise dispersion. The previous argument shows that it is not possible to have a local in time modulated Strichartz estimate with exponents $(p,q)$ outside of the usual admissible region $\frac2p+\frac{d}{q}\geq \frac{d}{2}$.
\end{remark}

\subsection{Proofs of Example~\ref{EX:noreg}}\label{SUBS:Example1}
In this subsection we detail how the various models presented in Example~\ref{EX:noreg} fit into the framework of Theorem~\ref{THM:noreg}. While the estimates presented here are very classical (for NLS in $\R^d$, slightly less so for NLS on $\T^d$ as they involve the multiscale spaces \eqref{XNLSper} below, which is not the form usually stated in the literature), the point of detailing them is to show how they fit into the framework \eqref{Strichartz}-\eqref{Strichartznonlinear}. Moreover, the computations of Subsubsection~\ref{SUBSUB:NLS} will serve as a basis for the multilinear analysis performed in Sections~\ref{SEC:NLScrit} and~\ref{SEC:NLSFL}.

\subsubsection{\textbf{NLS equation on $\R^d$}}
In this case, we have $\M=\R^d$, $\L=i\Delta$, and $\NN(u)=\pm i|u|^{2m}u$ with $m\in\N^*$. In particular \eqref{NLS} satisfies assumptions \ref{A1} and \ref{A3}. The deterministic NLS equation \eqref{NLS} is locally well-posed in $H^s(\R^d)$ for any $s\ge \max(s_c,0)=\max(\frac{d}{2}-\frac1m,0)$. Since the case $s>\frac{d}2$ can be treated simply using the algebra property of $H^s(\R^d)$ (i.e. there is no need for Strichartz estimates in this case), we focus on the case $\max(s_c,0)\le s<\frac{d}2$. We mainly follow the arguments in \cite{Caz,LP} for the sub-critical case, and \cite{CW} for the critical case.

Then \eqref{Strichartz} corresponds to the Strichartz estimate (see e.g. \cite[Theorem 2.3.3]{Caz} or \cite[Theorem 4.2]{LP})
\begin{align}\label{StrichartzNLS}
\big\|e^{it\Delta}u_0\big\|_{L^p_tL^q_x}\le c\|u_0\|_{L^2}
\end{align}
for any $p,q\in[2,\infty]$ satisfying the admissibility condition
\begin{align}\label{admissibility}
\frac{2}{p}+\frac{d}{q}=\frac{d}{2} \text{ and }
\begin{cases} 
q\in[2,\infty],~~d=1;\\
q\in[2,\infty),~~d=2;\\
q\in[2,\frac{2d}{d-2}),~~d\ge 3.
\end{cases}
\end{align}
Namely, we have $E=[2,\infty]$, $X_q=L^q_x$ for $q\in E$, so that $S^{p,+}_TX=L^p_TL^q_x$ and $\wt\EE=\{(p,q,+),~~(p,q)\text{ admissible}\}$. Actually, to deal with the scaling critical case, we use the refined space 
\begin{align}\label{SpqRd}
X_q(\R^d)=B^{0}_{q,2}(\R^d),
\end{align}
so that \eqref{Strichartz} still follows from \eqref{StrichartzNLS} in this case:
\begin{align*}
\big\|e^{it\Delta}u_0\big\|_{S^{p,+}_TX_q} &= \big\|e^{it\Delta}\P_N u_0\big\|_{L^p_T\ell^2_N L^q_x}\\
&\le \big\|e^{it\Delta}\P_N u_0\big\|_{\ell^2_N L^p_TL^q_x}\\
&\le c\|\P_N u_0\|_{\ell^2_N L^2_x}\sim \|u_0\|_{L^2_x},
\end{align*}
where we used Minkowski's inequality in the second step. Since we assumed $\EE$ to be finite in Theorem~\ref{THM:noreg}, we will restrict $\wt\EE$ to a suitable finite set $\EE$ in the following.

Since $\NN$ is multilinear in our case (recall that $m\in\N^*$), showing \eqref{Strichartznonlinear} boils down to estimating
\begin{align*}
\int_0^T\int_{\R^d}\prod_{j= 0}^{2m+1}u_j^{\pm_j}dxdt = \int_0^T\int_{\R^d}\sum_{N_0,...,N_{2m+1}}\prod_{j=0}^{2m+1}\P_{N_j}u_j^{\pm_j}dxdt
\end{align*}
for any $u_0\in\Sb^{-s}_T$ and $u_j\in\Sb^{s}_T$, $j=1,\dots,2m+1$, and any choice of signs $\pm_j$ (recall from Subsection~\ref{SUBS:notations} the notations $u^+=u$ and $u^-=\cj{u}$ for complex conjugates). Then, since our analysis does not depend on the signs $\pm_j$, by symmetry we can assume that $N_1=\max(N_1,\dots,N_{2m+1})$.
Note that from the assumption on $N_1$ and the convolution constraint when writing the product on the Fourier side using Plancherel theorem, we get the conditions $N_0\les N_1$ and $N_1\sim \max(N_0,\max(N_2,\dots,N_{2m+1}))$. In particular, if $N_0\ge \max(N_2,\dots,N_{2m+1})$, \eqref{Strichartznonlinear} is reduced to estimating
\begin{align}\label{multiRd0}
\int_0^T\int_{\R^d}\sum_{N_1\sim N_0}\Big(\prod_{j=0}^1\P_{N_j}u_j^{\pm_j}\Big)\Big(\prod_{j=2}^{2m+1}\P_{\le N_0}u_j^{\pm_j}\Big)dxdt,
\end{align}
while in the case $N_0\le \max(N_2,\dots,N_{2m+1})$ need to estimate
\begin{align}\label{multiRd1}
\int_0^T\int_{\R^d}\sum_{N_1\sim N_2}\Big(\prod_{j=1}^2\P_{N_j}u_j^{\pm_j}\Big)\Big(\P_{\le N_2}u_0^{\pm_0}\prod_{j=3}^{2m+1}\P_{\le N_2}u_j^{\pm_j}\Big)dxdt.
\end{align}

We first treat the case $s>\max(s_c,0)$. In this case, taking $0<\delta<s-s_c<\frac1m$, we simply restrict $\wt\EE$ to
\begin{align*}
\EE = \Big\{\big(\frac{2m}{1-m\delta},\frac{2dm}{dm-2+2m\delta},+\big)\Big\},
\end{align*}
so that ${\displaystyle \frac{d}{q}=s_c+\delta<s}$ and $p>2m$. Then in the case $N_0\ge N_2$, \eqref{Strichartznonlinear} follows from estimating \eqref{multiRd0} using H\"older, Sobolev  and Cauchy-Schwarz inequalities:
\begin{align*}
&\Big|\int_0^T\int_{\R^d}\sum_{N_0\sim N_1}\Big(\prod_{j=0}^1\P_{N_j}u_j^{\pm_j}\Big)\Big(\prod_{j=2}^{2m+1}\P_{\le N_0}u_j^{\pm_j}\Big)dxdt\Big|\\
&\qquad\le \Big\|\sum_{N_0\sim N_1}\Big(\prod_{j=0}^1\big\|\P_{N_j}u_j\big\|_{L^2_x}\Big)\Big(\prod_{j=2}^{2m+1}\big\|\P_{\le N_0}u_j\big\|_{L^\infty_x}\Big)\Big\|_{L^1_T}\\
&\qquad\les \Big\|\|u_0\|_{H^{-s}_x}\|u_1\|_{H^s_x}\prod_{j=2}^{2m+1}\|u_j\|_{W^{s,q}_x}\Big\|_{L^1_T}\\
&\qquad\les \|u_0\|_{L^\infty_T H^{-s}_x}\|u_1\|_{L^\infty_T H^s_x}\prod_{j=2}^{2m+1}\|u_j\|_{L^{2m}_TW^{s,q}_x}\\
&\qquad\les T^{m\delta}\|u_0\|_{L^\infty_T H^{-s}_x}\|u_1\|_{L^\infty_T H^s_x}\prod_{j=2}^{2m+1}\|u_j\|_{L^{p}_TB^{s}_{q,2}},
\end{align*}
where $(p,q,+)\in\EE$, and in the last step we used H\"older's inequality and Lemma~\ref{LEM:Besov}. Similarly, \eqref{multiRd1} is estimated by
\begin{align*}
&\Big|\int_0^T\int_{\R^d}\sum_{N_1\sim N_2}\Big(\prod_{j=1}^2\P_{N_j}u_j^{\pm_j}\Big)\Big(\P_{\le N_2}u_0^{\pm_0}\prod_{j=3}^{2m+1}\P_{\le N_2}u_j^{\pm_j}\Big)dxdt\Big|\\
&\qquad\les \Big\|\sum_{N_1\sim N_2}\|\P_{N_1}u_1\|_{L^2_x}\|\P_{N_2}u_2\|_{L^\infty_x}N_2^{s}\|\P_{\le N_2}u_0\|_{H^{-s}}\prod_{j=3}^{2m+1}\|u_j\|_{L^\infty_x}\Big\|_{L^1_T}\\
&\qquad\les \Big\|\sum_{N_1\sim N_2}N_1^s\|\P_{N_1}u_1\|_{L^2_x}N_2^{s}\|\P_{N_2}u_2\|_{L^q_x}\|u_0\|_{L^\infty_T H^{-s}}\prod_{j=3}^{2m+1}\|u_j\|_{W^{s,q}_x}\Big\|_{L^1_T}\\
&\qquad\les T^{m\delta}\|u_1\|_{L^\infty_TH^s_x}\|u_0\|_{L^\infty_T H^{-s}}\prod_{j=2}^{2m+1}\|u_j\|_{L^{p}_T B^{s}_{q,2}}.
\end{align*} 
This shows \eqref{Strichartznonlinear} with $\eta_\EE(s)=m\delta$ in view of the definitions \eqref{Spq} and \eqref{SpqRd} of $S^{p,+}_TX_q$ and $X_q$, and \eqref{Sb} of $\Sb^s_T$.

Next, we treat the critical case $s=s_c=\frac{d}{2}-\frac1m\ge 0$, and take instead
\begin{align*}
\EE=\Big\{\big(2m+2,\frac{d(2m+2)}{d(m+1)-2},+\big)\Big\}
\end{align*}
so that for $(p,q,+)\in\EE$, $(p,q)$ is admissible. Let also $\wt q = dm(m+1)$, so that
$$d\big(\frac1q-\frac1{\wt q}\big)=s_c\qquad\text{and}\qquad \frac2q+\frac{2m}{\wt q}=1.$$
Then we estimate \eqref{multiRd0} by
\begin{align*}
&\Big|\int_0^T\int_{\R^d}\sum_{N_0\sim N_1}\Big(\prod_{j=0}^1\P_{N_j}u_j^{\pm_j}\Big)\Big(\prod_{j=2}^{2m+1}\P_{\le N_0}u_j^{\pm_j}\Big)dxdt\Big|\\
&\qquad\le \Big\|\sum_{N_0\sim N_1}\Big(\prod_{j=0}^1\big\|\P_{N_j}u_j\big\|_{L^q_x}\Big)\Big(\prod_{j=2}^{2m+1}\big\|\P_{\le N_0}u_j\big\|_{L^{\wt q}_x}\Big)\Big\|_{L^1_T}\\
&\qquad\les \Big\|\sum_{N_0\sim N_1}\Big(\prod_{j=0}^1\big\|\P_{N_j}u_j\big\|_{L^q_x}\Big)\Big(\prod_{j=2}^{2m+1}\big\|u_j\big\|_{B^{s}_{q,2}}\Big)\Big\|_{L^1_T}\\
&\qquad\les \|u_0\|_{L^p_TB^{-s}_{q,2}}\prod_{j=1}^{2m+1}\big\|u_j\big\|_{L^p_TB^{s}_{q,2}}.
\end{align*}
We estimate \eqref{multiRd1} similarly by
\begin{align*}
&\Big|\int_0^T\int_{\R^d}\sum_{N_1\sim N_2}\Big(\prod_{j=1}^2\P_{N_j}u_j^{\pm_j}\Big)\Big(\P_{\le N_2}u_0^{\pm_0}\prod_{j=3}^{2m+1}\P_{\le N_2}u_j^{\pm_j}\Big)dxdt\Big|\\
&\qquad\les \Big\|\sum_{N_1\sim N_2}\|\P_{N_1}u_1\|_{L^q_x}\|\P_{N_2}u_2\|_{L^{\wt q}_x}\|\P_{\le N_2}u_0\|_{L^q_x}\prod_{j=3}^{2m+1}\|u_j\|_{L^{\wt q}_x}\Big\|_{L^1_T}\\
&\qquad\les \Big\|\sum_{N_1\sim N_2}N_1^s\|\P_{N_1}u_1\|_{L^q_x}N_2^s\|\P_{N_2}u_2\|_{L^{q}_x}\|u_0\|_{B^{-s}_{q,2}}\prod_{j=3}^{2m+1}\|u_j\|_{B^{s}_{q,2}}\Big\|_{L^1_T}\\
&\qquad\les  \|u_0\|_{L^p_TB^{-s}_{q,2}}\prod_{j=1}^{2m+1}\big\|u_j\big\|_{L^p_TB^{s}_{q,2}}.
\end{align*}
This finally shows \eqref{Strichartznonlinear} in the case $s=s_c$, with $\eta_\EE(s_c)=0$.

The only remaining case is $s\ge 0 >s_c$, which corresponds only to the $1d$-cubic case, i.e. $d=m=1$. In this case we take $\EE=
\{(8,4,+)\}$ since $(8,4)$ is admissible, and estimate exactly as above
\begin{align*}
&\Big|\int_0^T\int_{\R^d}\sum_{N_0\sim N_1}\Big(\prod_{j=0}^1\P_{N_j}u_j^{\pm_j}\Big)\Big(\prod_{j=2}^{3}\P_{\le N_0}u_j^{\pm_j}\Big)dxdt\Big|\\
&\qquad\les \Big\|\sum_{N_0\sim N_1}N_0^{-s}N_1^s\Big(\prod_{j=0}^1\big\|\P_{N_j}u_j\big\|_{L^4_x}\Big)\Big(\prod_{j=2}^{3}\big\|u_j\big\|_{L^{4}_x}\Big)\Big\|_{L^1_T}\\
&\qquad\les T^{\frac12}\|u_0\|_{L^8_TB^{-s}_{4,2}}\|u_1\|_{L^8_TB^{s}_{4,2}}\prod_{j=2}^3\|u_j\|_{L^8_TL^4_x},
\end{align*}
and similarly
\begin{align*}
&\Big|\int_0^T\int_{\R^d}\sum_{N_1\sim N_2}\Big(\prod_{j=1}^2\P_{N_j}u_j^{\pm_j}\Big)\Big(\P_{\le N_2}u_0^{\pm_0}\P_{\le N_2}u_3^{\pm_3}\Big)dxdt\Big|\\
&\qquad\les \Big\|\sum_{N_1\sim N_2}\|\P_{N_1}u_1\|_{L^4_x}\|\P_{N_2}u_2\|_{L^{4}_x}\|\P_{\le N_2}u_0\|_{L^4_x}\|u_3\|_{L^{4}_x}\Big\|_{L^1_T}\\
&\qquad\les \Big\|\sum_{N_1\sim N_2}N_1^s\|\P_{N_1}u_1\|_{L^4_x}N_2^s\|\P_{N_2}u_2\|_{L^{4}_x}\|u_0\|_{B^{-s}_{4,2}}\|u_3\|_{L^4_x}\Big\|_{L^1_T}\\
&\qquad\les  T^{\frac12}\|u_0\|_{L^8_TB^{-s}_{4,2}}\prod_{j=1}^{3}\big\|u_j\big\|_{L^8_TB^{s}_{4,2}}.
\end{align*}

\subsubsection{\textbf{NLS equation on $\T^d$}}\label{SUBSUB:NLS}
We now consider the case $\M=\T^d$, $\L=i\Delta$, and $\NN(u)=\pm i|u|^{2m}u$ with $m\in\N^*$. We have again $s_c=\frac{d}{2}-\frac1m$ and \eqref{NLS} is well-posed in $H^s(\T^d)$ for any $s>\max(s_c,0)$, and for $s=\max(s_c,0)$ except in the cases $d=1,m=2$ and $d=2,m=1$ \cite{HTT,Wang}. In the following, we focus on the sub-critical case\footnote{In the $1d$-cubic case, this yields the restriction $s>0$. See Remark~\ref{REM:cubicNLS} below.} $s>\max(s_c,0)$.

In this case, \eqref{Strichartz} follows from the localized periodic Strichartz estimate
\begin{align}\label{StrichartzLoc}
\|\Pi_N e^{it\Delta}u_0\|_{L^p_{t,x}(\T\times\T^d)}\les \Lambda_N(d,p)\|\Pi_N u_0\|_{L^2_x},
\end{align}with 
\begin{align*}
\Lambda_N(d,p)=\begin{cases}1,\qquad d=1,~p=4,\\
N^\eps,\qquad 2\le p \le \frac{2(d+2)}{d},\\
N^{\frac{d}2-\frac{d+2}{p}},\qquad p>\frac{2(d+2)}{d}.
\end{cases}
\end{align*}

Using number-theoretic argument, Bourgain \cite{Bo1} first established part of \eqref{StrichartzLoc} for some range of $p$ and $d$. This was then extended to the whole range in \eqref{StrichartzLoc} using the breakthrough of Bourgain and Demeter \cite{BD} regarding the decoupling conjecture, and refined in \cite{KV} in the range $p>\frac{2(d+2)}{d}$.

For $\wt N$ a dyadic integer, let $\CC(\wt N)=\{\wt n+[-\wt N,\wt N]^d\}$ be a collection of cubes of side-length $\wt N$ covering $\R^d$. Note that for $\CC=\wt n+[-\wt N,\wt N]^d\in\CC(\wt N)$, we have 
$$ \P_{\CC}u= e^{i \wt n\cdot x}\P_{[-\wt N,\wt N]^d}\big(e^{-i \wt n\cdot x}u\big)$$
and
$$\P_{\CC}\big(e^{it\Delta}u\big)(t,x)= e^{i \wt n\cdot x +t|\wt n|^2}\P_{[-\wt N,\wt N]^d}\big(e^{-i \wt n\cdot x}u\big)(t,x-2t\wt n),$$
so that \eqref{StrichartzLoc} also gives for any $N,\wt N$ dyadic integers and $\CC\in\CC(\wt N)$:
\begin{align}\label{StrichartzLoc2}
\|\P_{\CC} \P_N e^{it\Delta}u_0\|_{L^p_{t,x}(\T\times\T^d)}\les \Lambda_{\wt N \wedge N}(d,p)\|\P_{\CC}\P_N u_0\|_{L^2_x}.
\end{align}
Thus, for $p\ge 2$ and $0<\eps\ll 1$, letting
$$\sigma(p)=\begin{cases}
2\eps,& 2\le p\le \frac{2(d+2)}{d},\\
\frac{d}{2}-\frac{d+2}{p}+\eps,& p>\frac{2(d+2)}{d},
\end{cases}$$
and using Minkowski's inequality and \eqref{StrichartzLoc2}, we get that for any $p\ge 2$ and $T\in (0,2\pi)$:
\begin{align}\label{StrichartzLoc3}
&\big\|\mathbf{1}_{\wt N\les N}\wt N^{-\sigma(p)}\P_\CC \P_N e^{it\Delta}u_0\big\|_{L^p_T\ell^2_N\ell^1_{\wt N}\ell^2(\CC(\wt N))L^p_x}\notag\\
&\qquad \les\big\|\mathbf{1}_{\wt N\les N}\wt N^{-\sigma(p)}\P_\CC \P_N e^{it\Delta}u_0\big\|_{\ell^1_{\wt N}\ell^2_N\ell^2(\CC(\wt N))L^p_{T,x}}\notag\\
&\qquad\les\big\|\mathbf{1}_{\wt N\les N}\wt N^{-\eps}\P_\CC \P_N u_0\big\|_{\ell^1_{\wt N}\ell^2_N\ell^2(\CC(\wt N))L^2_x}\notag\\
&\qquad\sim \sum_{\wt N}\wt N^{-\eps} \|\P_{\gtrsim \wt N}u_0\|_{L^2} \les \|u_0\|_{L^2}.
\end{align}
To capture the refinement \eqref{StrichartzLoc2}, we thus define the spaces $X_\zeta$ with a multiscale structure:
\begin{align}\label{XNLSper}
\|u\|_{X_p}:=\big\|\mathbf{1}_{\wt N\les N}\wt N^{-\sigma(p)}\P_\CC \P_N u\big\|_{\ell^2_N\ell^1_{\wt N}\ell^2(\CC(\wt N))L^p_x},
\end{align}
so that \eqref{Strichartz} follows from \eqref{StrichartzLoc3} for any $(p,q,+)\in\wt \EE$, with
\begin{align*}
\wt \EE:=\big\{(p,p,+),~~p\in [2,\infty]\big\}.
\end{align*}
Since we have $s>\max(s_c,0)$, taking $0<\eps\ll s-\max(s_c,0)$, we set $p=(d+2)m+\eps$ and ${\displaystyle q=\frac{2(d+2)}{d}}$ so that ${\displaystyle\frac{2}{q}+\frac{2m}{(d+2)m}=1}$. Then with these choices of $p$ and $q$ we restrict $\wt\EE$ to
$$\EE=\big\{(p,p,+),(q,q,+)\big\}.$$
Note that with this choice, we have $p> \frac{2(d+2)}{d}$ unless $d=m=1$, and $q\le\frac{2(d+2)}{d}$, and that in all cases $\sigma(q)=2\eps$ and $\sigma(p)=\max(s_c,0)+O(\eps)$.

To establish \eqref{Strichartznonlinear}, we decompose again
\begin{align*}
\int_0^T\int_{\T^d}\prod_{j= 0}^{2m+1}u_j^{\pm_j}dxdt = \int_0^T\int_{\T^d}\sum_{N_0,...,N_{2m+1}}\prod_{j=0}^{2m+1}\P_{N_j}u_j^{\pm_j}dxdt
\end{align*}
with now $N_1\ge N_2\ge\dots\ge N_{2m+1}$ by symmetry. We have again $N_0\les N_1$ and $N_1\sim \max(N_0,N_2)$. In this case, we make a further decomposition: for example, if $N_0\ge N_2$, since $ \P_{N_0}u_0$ has its Fourier support in $\{|n|\sim N_1\}$, we can cover this region by $O(N_1^dN_2^{-d})$ cubes $\CC_{\ell}\in\CC(N_2)$, $|\ell|\les N_1^dN_2^{-d}$, of side length $N_2$ and centred around some $|n_\ell|\sim N_1$. This leads to the decomposition
\begin{align*}
\sum_{N_0\sim N_1\ge N_2\ge ...\ge N_{2m+1}}\sum_{\CC_\ell,\CC_{\wt\ell}\in\CC(N_2)}\int_0^T\int_{\T^d}\P_{\CC_\ell}\P_{N_0}u_0^{\pm_0}\P_{\CC_{\wt \ell}}\P_{N_1}u_1^{\pm_1}\prod_{j=2}^{2m+1}\P_{N_j}u_j^{\pm_j}dxdt.
\end{align*}
Now, using again the convolution constraint when writing the product on the Fourier side, we see that for any fixed $|\ell|\les N_1^dN_2^{-d}$, there are at most $O(1)$ values of $\wt \ell$ such that the Fourier supports of $\P_{\CC_\ell}\P_{N_0}u_0^{\pm_0}$ and $\P_{\CC_{\wt \ell}}\P_{N_1}u_1^{\pm_1}\prod_{j=2}^{2m+1}\P_{N_j}u_j^{\pm_j}$ are over-lapping, i.e. the integral on $\T^d$ is non-zero. Namely, similarly as in \eqref{multiRd0}-\eqref{multiRd1}, we are left with estimating
\begin{align}\label{multiTd0}
\sum_{N_0\sim N_1\ge N_2\ge...\ge N_{2m+1}}\sum_{\substack{\CC_\ell,\CC_{\wt\ell}\in\CC(N_2)\\|\wt \ell-\ell|\les 1}}\int_0^T\int_{\T^d}\P_{\CC_\ell}\P_{N_0}u_0^{\pm_0}\P_{\CC_{\wt \ell}}\P_{N_1}u_1^{\pm_1}\prod_{j=2}^{2m+1}\P_{N_j}u_j^{\pm_j}dxdt
\end{align}
and
\begin{align}\label{multiTd1}
\sum_{N_2\sim N_1\ge N_0\ge...\ge N_{2m+1}}\sum_{\substack{\CC_\ell,\CC_{\wt\ell}\in\CC(N_0)\\|\wt\ell-\ell|\les 1}}\int_0^T\int_{\T^d}\P_{\CC_\ell}\P_{N_2}u_2^{\pm_2}\P_{\CC_{\wt \ell}}\P_{N_1}u_1^{\pm_1}\prod_{j=0,2}^{2m+1}\P_{N_j}u_j^{\pm_j}dxdt.
\end{align}

For $T\in(0,2\pi)$, we then estimate \eqref{multiTd0} with the help of H\"older' inequality:
\begin{align*}
&\Big|\sum_{N_0\sim N_1\ge N_2\ge...\ge N_{2m+1}}\sum_{\substack{\CC_\ell,\CC_{\wt\ell}\in\CC(N_2)\\|\wt \ell-\ell|\les 1}}\int_{\T^d}\P_{\CC_\ell(N_2)}\P_{N_0}u_0^{\pm_0}\P_{\CC_{\wt \ell}(N_2)}\P_{N_1}u_1^{\pm_1}\prod_{j=2}^{2m+1}\P_{N_j}u_j^{\pm_j}dx\Big|\\
&\qquad\les\sum_{N_0\sim N_1\ge N_2\ge...\ge N_{2m+1}}\sum_{\substack{\CC_\ell,\CC_{\wt\ell}\in\CC(N_2)\\|\wt \ell-\ell|\les 1}}\big\|\P_{\CC_\ell}\P_{N_0}u_0\big\|_{L^q_{x}}\big\|\P_{\CC_{\wt \ell}}\P_{N_1}u_1\big\|_{L^q_x}\prod_{j=2}^{2m+1}\|\P_{N_j}u_j\|_{L^{(d+2)m}_x}.
\end{align*}
Now, using H\"older's inequality with the compactness of $\T^d$ and the definition of $p$, and the definition \eqref{XNLSper} of $X_p$, we have from using Cauchy-Schwarz inequality:
\begin{align*}
\sum_{N}\|\P_{N}u\|_{L^{(d+2)m}_x}&\les \sum_{N}\|\P_{N}u\|_{L^{p}_x} \sim \big\|\mathbf{1}_{\wt N\sim N}\P_{\CC}\P_{N}u\big\|_{\ell^\infty_{\wt N}\ell^1_N\ell^2_{\CC(\wt N)}L^p_x}\\
&\les \big\|\wt N^{\sigma(p)}N^{-s}\mathbf{1}_{\wt N\les N}\wt N^{-\sigma(p)}N^s\P_{\CC}\P_{N}u\big\|_{\ell^1_{\wt N}\ell^1_N\ell^2_{\CC(\wt N)}L^p_x}\\
&\les \|\jb{D}^su\|_{X_p}\big\|\mathbf{1}_{\wt N\les N}\wt N^{\sigma(p)}N^{-s}\big\|_{\ell^\infty_{\wt N}\ell^2_N}\les \|\jb{D}^su\|_{X_p},
\end{align*}
where we used in the second step that $\#\CC(\wt N)\cap \supp\ft\P_N = O(\jb{\frac{N}{\wt N}}^d)=O(1)$ when $\wt N\sim N$, and in the last step that $\sigma(p)=\max(s_c,0)+O(\eps)<s$ with our choice of $p$ and $0<\eps \ll s-\max(s_c,0)$.

Similarly, using Cauchy-Schwarz to sum on both $N_0\sim N_1$ and $\ell,\wt\ell$, we end up with
\begin{align*}
&\Big|\sum_{N_0\sim N_1\ge N_2\ge...\ge N_{2m+1}}\sum_{\substack{\CC_\ell,\CC_{\wt\ell}\in\CC(N_2)\\|\wt \ell-\ell|\les 1}}\int_{\T^d}\P_{\CC_\ell(N_2)}\P_{N_0}u_0^{\pm_0}\P_{\CC_{\wt \ell}(N_2)}\P_{N_1}u_1^{\pm_1}\prod_{j=2}^{2m+1}\P_{N_j}u_j^{\pm_j}dx\Big|\\
&\les\Big(\prod_{j=3}^{2m+1}\|\jb{D}^su_j\|_{X_p}\Big)\sum_{N_2}\Big(N_2^{2\sigma(q)}\|\P_{N_2}u_2\|_{L^{p}_x}\\
&\qquad\qquad\times\big\|N_2^{-\sigma(q)}N_0^{-s}\P_{\CC}\P_{N_0}u_0\big\|_{\ell^2_{N_0}\ell^2(\CC(N_2))L^q_{x}}\big\|N_2^{-\sigma(q)}N_1^s\P_{\CC}\P_{N_1}u_1\big\|_{\ell^2_{N_1}\ell^2(\CC(N_2))L^q_x} \Big)\\
&\qquad\les \|\jb{D}^{-s}u_0\|_{X_q}\|\jb{D}^{s}u_1\|_{X_q}\Big(\prod_{j=3}^{2m+1}\|\jb{D}^su_j\|_{X_p}\Big)\sum_{N_2}N_2^{2\sigma(q)}\|\P_{N_2}u_2\|_{L^{p}_x}.
\end{align*}
Since, with our choice of $q$, we have $\sigma(q)=2\eps$, we can finally sum as above
\begin{align*}
\sum_{N_2}N_2^{2\sigma(q)}\|\P_{N_2}u_2\|_{L^{p}_x}&\les \|\jb{D}^su_2\|_{X_p}\big\|\mathbf{1}_{\wt N\les N_2}\wt N^{\sigma(p)}N_2^{4\eps-s}\big\|_{\ell^\infty_{\wt N}\ell^2_{N_2}}\les \|\jb{D}^su_2\|_{X_p},
\end{align*}
using again that $\sigma(p)=\max(s_c,0)+O(\eps)<s-4\eps$ with our choice of $p$ and $0<\eps \ll s-s_c$.

All in all, taking the $L^1_T$ norm of the previous computations, using H\"older's inequality with the choice of $q$ and $p$, we end up with the following estimate on \eqref{multiTd0}:
\begin{align*}
&\Big|\sum_{N_0\sim N_1\ge N_2\ge...\ge N_{2m+1}}\sum_{\substack{\CC_\ell,\CC_{\wt\ell}\in\CC(N_2)\\|\wt \ell-\ell|\les 1}}\int_0^T\int_{\T^d}\P_{\CC_\ell(N_2)}\P_{N_0}u_0^{\pm_0}\P_{\CC_{\wt \ell}(N_2)}\P_{N_1}u_1^{\pm_1}\prod_{j=2}^{2m+1}\P_{N_j}u_j^{\pm_j}dxdt\Big|\\
&\qquad\qquad\les \Big\|\|\jb{D}^{-s}u_0\|_{X_q}\|\jb{D}^{s}u_1\|_{X_q}\prod_{j=2}^{2m+1}\|\jb{D}^su_j\|_{X_p}\Big\|_{L^1_T}\\
&\qquad\qquad\les \|\jb{D}^{-s}u_0\|_{L^q_TX_q}\|\jb{D}^{s}u_1\|_{L^q_TX_q}\prod_{j=2}^{2m+1}\|\jb{D}^su_j\|_{L^{(d+2)m}_TX_p}\\
&\qquad\qquad\les T^{2m\eps}\|\jb{D}^{-s}u_0\|_{L^q_TX_q}\|\jb{D}^{s}u_1\|_{L^q_TX_q}\prod_{j=2}^{2m+1}\|\jb{D}^su_j\|_{L^{p}_TX_p}.
\end{align*}

Proceeding as above, we find a similar estimate for \eqref{multiTd1}. This finally shows \eqref{Strichartznonlinear} with $\eta_\EE(s)=2m\eps$.

\begin{remark}\label{REM:cubicNLS}
\rm
 Note that in the case $d=m=1$, the argument above only shows well-posedness of \eqref{EQ} for $s>0$, while Bourgain \cite{Bo1} proved global well-posedness in $L^2(\T)$ for \eqref{NLS} using the lossless periodic Strichartz estimate
\begin{align}\label{StrichartzL4}
\|e^{it\dx^2}u\|_{L^4_{t,x}(\T\times\T)}\les \|u\|_{L^2_x}.
\end{align}

In particular, in case $d=m=1$ and $s=0$, the estimate \eqref{Strichartz} with $(p,\zeta,\iota)=(4,4,+)$ and $X_4 = L^4_{x}$ follows from \eqref{StrichartzL4}, while \eqref{Strichartznonlinear} follows from H\"older's inequality through
\begin{align*}
\Big|\int_0^T\int_\T \prod_{j=0}^3 u_j^{\pm_j}dxdt\Big|\les \prod_{j=0}^3\|u_j\|_{L^4_{T,x}}
\end{align*}
for any $T\in(0,2\pi)$, but provides only $\eta_\EE(s)=0$. This can only be used for a small data well-posedness result. For the deterministic cubic NLS in $L^2(\T)$, this is overcome in \cite{Bo1} by using a refined $X^{s,b}$ version of \eqref{StrichartzL4}, namely 
\begin{align}\label{L4Xsb}
\|u\|_{L^4_{T,x}}\les \|u\|_{X^{0,\frac38+}_T}
\end{align} 
for $T\in(0,2\pi)$, allowing to get $\eta_\EE =\frac12-$.

As far as the modulated periodic cubic NLS \eqref{EQ:deterministic} is concerned, it seems possible to capture the refined $L^4$ estimate \eqref{L4Xsb} only through the use of the assumption \ref{A0gr} for some $\gamma>\frac12$ and $\rho>\frac12$ in the spirit of the proof of Theorem~\ref{THM:NLScrit} below; see \cite[Theorem 1.8]{CG1} and the proof of Theorem~\ref{THM:NLSFL} below. However, under assumption \ref{A0*}, we can again use the smallness of $C_W(T)$ in \eqref{CW2} to recover large data global well-posedness in $L^2(\T)$. Note that both assumptions \ref{A0*} and \ref{A0gr} with $\gamma,\rho>\frac12$ barely miss the case $W_t=t$; see Remark~\ref{REM:A0}.
\end{remark}

\subsubsection{\textbf{NLS equation on compact manifolds}}
In the case where $\M$ is a general closed, boundaryless, smooth $d$-dimensional Riemannian manifold, \cite{BGT1,ST} established the Strichartz estimate
\begin{align}\label{StrichartzM}
\|e^{it\Delta}u\|_{L^p_TL^q_x}\les \|u\|_{H^\frac1p}
\end{align}
for any $T\in(0,1]$ and admissible $(p,q)$ satisfying \eqref{admissibility}. In particular, this implies \eqref{Strichartz} with $\zeta=(p,q)\in [2,\infty]^2$, $X_{p,q}=W^{-\frac1p,q}(\M)$, and
$$\wt\EE=\big\{(p,p,q,+),~~(p,q)\text{ admissible}\big\}.$$
Let $s>\frac{d}2-\frac1{2m}$. We then restrict $\wt\EE$ to $\EE=\{(p,p,q,+)\}$ where $p=2m+\eps$ and $q$ given by $\frac2p+\frac{d}{q}=\frac{d}2$, for some $0<\eps\ll 1$ so that $s>\frac{d}2-\frac1p = \frac{d}q-\frac1p$.

Then from the proof of Proposition 3.1 in \cite{BGT1}, it holds
\begin{align*}
\Big|\int_0^T\int_\M\prod_{j=0}^{2m+1}u_jdxdt\Big|&\les \|u_0\|_{L^\infty_TH^{-s}}\big\|\prod_{j=1}^{2m+1}u_j\big\|_{L^1_TH^s}\\
&\les \|u_0\|_{L^\infty_TH^{-s}}\sum_{j=1}^{2m+1}\Big\|\|u_j\|_{H^s}\prod_{\ell\neq j}\|u_j\|_{L^\infty_x}\Big\|_{L^1_T}\\
&\les \|u_0\|_{L^\infty_TH^{-s}}\sum_{j=1}^{2m+1}\|u_j\|_{L^\infty_TH^s}\prod_{\ell\neq j}\|u_j\|_{L^{2m}_TW^{s-\frac1p,q}}\\
&\les T^{2m\eps}\|u_0\|_{L^\infty_TH^{-s}}\sum_{j=1}^{2m+1}\|u_j\|_{L^\infty_TH^s}\prod_{\ell\neq j}\|u_j\|_{L^{p}_TW^{s-\frac1p,q}},
\end{align*}
where we used the fractional Leibniz rule, the Sobolev embedding $W^{s-\frac1p,q}(\M)\subset L^\infty(\M)$ from our choice of $p$ and $q$, and H\"older's inequality. This shows \eqref{Strichartznonlinear} with $\eta_\EE(s)=2m\eps$.

\subsubsection{\textbf{Mass-critical gKdV equation on $\R$}}
In the case of \eqref{gKdV} we have that \eqref{Strichartz} holds with $\zeta=(q,\theta)\in [2,\infty]\times\R$ and $X_{q,\theta}= D^{-\theta}L^q(\R)$. Indeed, from \cite[Section 7.1]{LP}, we have the local smoothing estimate 
\begin{align*}
\big\|D^\theta e^{t\dx^3}u\big\|_{L^q_xL^p_T} \les \|u\|_{L^2}
\end{align*}
 for $(p,q,\theta,\iota)\in\big\{(2,\infty,1,-);(\infty,4,-\frac14,-)\big\}$, uniformly in $T>0$. By interpolation, we deduce that it holds for any $(p,q,\theta,\iota)\in\wt{\mathcal{E}}$, where
\begin{align*}
\wt{\mathcal{E}}=\bigcap_{\delta\in (0,1)}\Big\{&(p,q,\theta,\iota)\in (2,\infty)\times (4,\infty)\times\R\times\{\pm\},\\
&\qquad\iota=-,~~p=\frac2\delta,~~q=\frac{4}{1-\delta},~~\theta = \delta -\frac{1-\delta}{4}\Big\}.
\end{align*}
Since we assumed $\mathcal{E}$ to be finite in Theorem~\ref{THM:noreg}, we restrict $\widetilde{\mathcal{E}}$ by setting
\begin{align*}
\mathcal{E}=\Big\{(\frac{10}{3},10,\frac12,-),(10,5,0,-)\Big\}.
\end{align*}
Then, we can estimate with H\"older's inequality and the fractional Leibniz rule (see \cite[Lemma 7.6]{LP}):
\begin{align*}
\Big|\int_0^T\int_\R \dx u_0\prod_{j=1}^5u_j dxdt\Big|&\les \|D^\frac12u_0\|_{L^{10}_xL^{\frac{10}3}_T}\big\|D^{-\frac12}\dx\prod_{j=1}^5u_j\big\|_{L^{\frac{10}{9}}_xL^{\frac{10}{7}}_T}\\
& \les \|D^\frac12u_0\|_{L^{10}_xL^{\frac{10}3}_T}\sum_{j=1}^5\|D^\frac12u_j\|_{L^{10}_xL^{\frac{10}3}_T}\big\|\prod_{\ell\neq j}u_\ell\big\|_{L^{\frac{5}{4}}_xL^{\frac52}_T}\\
&\les \|D^\frac12u_0\|_{L^{10}_xL^{\frac{10}3}_T}\sum_{j=1}^5\|D^\frac12u_j\|_{L^{10}_xL^{\frac{10}3}_T}\prod_{\ell\neq j}\big\|u_\ell\big\|_{L^{5}_xL^{10}_T}
\end{align*}
This shows \eqref{Strichartznonlinear} when $s=0$ for this model, with $\eta_\EE(s)=0$ in this case, showing that it is indeed a mass-critical model. The case $s>0$ follows from a straightforward adaptation.

\section{The modulated periodic NLS equation at critical regularity}\label{SEC:NLScrit}
We now turn to the proof of Theorem~\ref{THM:NLScrit}. Since, under \ref{A0gr} for $\gamma>\frac12$ and $\rho>1$, local well-posedness for $s>\max(s_c,0)=\max(\frac{d}2-\frac1m,0)$ follows from Theorem~\ref{THM:noreg} as discussed in Subsubsection~\ref{SUBSUB:NLS} and Remark~\ref{REM:A0}~(iii), we focus on the critical case $s=s_c=\frac{d}2-\frac1m>0$; recall the assumption $m>\frac2d$ in Theorem~\ref{THM:NLScrit}.

The main result of this section is therefore the following multilinear estimate.
\begin{proposition}\label{PROP:multilinearT3bis}
Assume that \ref{A0gr} holds for some $\gamma>\frac12$ and $\rho>1$, and recall that $s_c=\frac{d}{2}-\frac1m>0$. For any $T>0$, $u_j\in X^{s_c}_T$, $j=1,...,2m+1$, and $u_0\in Y^{-s_c}_T$, and signs $\pm_j$, it holds
\begin{align}\label{multiT34}
\Big|\int_0^T\int_{\T^d}\prod_{j=0}^{2m+1}u_j^{\pm_j}dxdt\Big|\les T^\gamma\|u_0\|_{Y^{-s_c}_T}\prod_{j=1}^{2m+1}\|u_j\|_{X^{s_c}_T}.
\end{align}
\end{proposition}

Indeed, with this proposition at hand, Theorem~\ref{THM:NLScrit} follows from a fixed-point argument as in the previous section.

\begin{proof}[Proof of Theorem~\ref{THM:NLScrit} (assuming Proposition~\ref{PROP:multilinearT3bis})]
For $u_0\in H^{s_c}(\T^d)$, $R=2\|u_0\|_{H^{s_c}}$, and $T\in(0,1]$, define 
$$B_R(T):=\big\{u\in C([0,T);H^{s_c}(\T^d))\cap X^{s_c}_T,~~\|u\|_{X^{s_c}_T}\le R\big\}$$
which is closed in $X^{s_c}_T$, and
\begin{align*}
\Phi_{u_0}: u\in X^{s_c}_T\mapsto e^{i(W_t-W_0)\Delta}u_0\pm i\int_0^t e^{i(W_t-W_{\tau})\Delta}|u|^{2m}u(\tau)d\tau.
\end{align*}
We will show that $\Phi_{u_0}$ defines a contraction on $B_R(T)$. From Proposition~\ref{PROP:UpVp}, the multilinearity of $\NN(u)=|u|^{2m}u$, and Proposition~\ref{PROP:multilinearT3bis}, we have for any $u,v\in B_R(T)$:
\begin{align*}
\big\|\Phi_{u_0}(u)-\Phi_{v_0}(v)\big\|_{X^{s_c}_T}&\le \|u_0-v_0\|_{H^{s_c}}+\sup_{\|w\|_{Y^{-s_c}_T}\le 1}\Big|\int_0^T\int_{\T^d}\big(|u|^{2m}u-|v|^{2m}v\big)\cj w dxdt\Big|\\
&\le \|u_0-v_0\|_{H^{s_c}} +CT^\gamma\|w\|_{Y^{-s_c}_T}\|u-v\|_{X^{s_c}_T}\big(\|u\|_{X^{s_c}_T}^{2m}+\|v\|_{X^{s_c}_T}^{2m}\big).
\end{align*}
Thus, taking $T(R)\sim \jb{R}^{-\frac{2m}\gamma}$ ensures that $\Phi_{u_0}$ defines a contraction on $B_R(T)$, thus providing a mild solution $u\in C([0,T);H^{s_c}(\T^d))$ to \eqref{EQ} on $[0,T)$, unique in $B_R(T)$. The uniqueness in the whole space $C([0,T);H^{s_c}(\T^d))\cap X^{s_c}_T$ follows from iterating the fixed point argument as in the proof of Theorem~\ref{THM:noreg}, and the estimate above shows the Lipschitz continuity of the flow map. This concludes the proof that Theorem~\ref{THM:NLScrit} follows from Proposition~\ref{PROP:multilinearT3bis}.
\end{proof}

Before going into the details of Proposition~\ref{PROP:multilinearT3bis}, let us briefly describe the strategy of the proof. As in the works \cite{HTT,Wang} on the deterministic critical periodic NLS equation, the key point is to complement the linear Strichartz estimates \eqref{StrichartzLoc}-\eqref{StrichartzLoc2} used in Section~\ref{SEC:Noreg} with a genuinely \emph{multilinear} improvement. The reason for this multilinear estimate is the $\epsilon$-loss in the scale-invariant Strichartz estimates \eqref{StrichartzLoc} on $\T^d$ compared to the Euclidean setting \eqref{StrichartzNLS}. As $\epsilon>0$ can be taken arbitrarily small, this allows to cover the whole sub-critical regime of regularity $s$ by taking $0<\epsilon\ll s-\max(s_c,0)$ as in the proof of Example~\ref{EX:noreg}~(ii) in Section~\ref{SUBSUB:NLS} above. In the critical case $s=s_c>0$, one faces the issue of summing the frequency localized estimates together with an $\epsilon$-loss of derivative. Note that this $\epsilon$-loss can be removed for Strichartz estimates \eqref{StrichartzLoc} with data frequency localized to cubes, when one uses Strichartz estimate with larger, non scale-invariant integrability exponent \cite{Bo1,BD,KV} (there is only the loss due to scaling, which is the same as in the Euclidean setting). Moreover, due to Galilean invariance of the linear Schr\"odinger equation, estimates for data frequency-localized on any cube reduce to estimates for data frequency-localized on a cube of same sidelength but centered at the origin, therefore the estimates are only sensitive to the size of the frequency domain. The smaller the latter is, the better the Strichartz estimate becomes, but the tradeoff is to be able to  resum these localized estimates since there are more and more localized pieces to sum as one localizes on smaller and smaller domains. Then Herr-Tataru-Tzvetkov \cite{HTT} observed (in the deterministic case) that one can indeed improve the Strichartz estimates for data localized on strips with smaller width compared to cubes, and crucially that this decomposition of the frequency domain into smaller (to a certain extent) strips instead of cubes retains the almost orthogonality due to time oscillations (instead of spatial oscillations giving almost orthogonality of the decomposition into cubes). This allows to sum the localized estimates without an extra loss due to the larger number of pieces.

In our modulated setting, to exploit the atomic structure of solutions in $X^s_T$, we thus start by recasting the classical combination of using linear Strichartz estimates for functions localized on cubes together with Hölder's inequality, as a suitable multilinear restriction estimates for the product of functions depending on space only. Indeed, integration in time is performed by combining the atomic decomposition of functions in the nonlinearity with the occupation time formula \eqref{occupation} and the properties of the local time of $W_t$, which leaves us with estimates for the frequency-restricted product of atoms. This is done in Lemma~\ref{LEM:multilinearT31} below. Then we make the same transfer from Strichartz and Hölder to multilinear restriction, but for the refined Strichartz estimate of \cite{HTT} for functions localized on strips of smaller width (Lemma~\ref{LEM:multilinearT32}). Finally, the decomposition into strips for functions originally localized on cubes and its almost orthogonality due to ``time oscillations"\footnote{The whole point of this man\oe uvre is to exploit these time oscillations at the level of a multilinear restriction estimate for functions independent of time, by seeing them as reducing the matter to the contribution where the resonance function is small, namely $|\nu|\ll N_2^2$ in the proof of Lemma~\ref{LEM:multilinearT3} below. This gives a further constraint on the spatial frequencies, allowing to recover the almost orthogonality of the decomposition into strips. Note that there is no time involved here since the integration in time is performed beforehand and functions of space only are involved in the estimate.} is established in Lemma~\ref{LEM:multilinearT3} below, again as a multilinear restriction estimate. We finally use Lemma~\ref{LEM:multilinearT3} to give the proof of Proposition~\ref{PROP:multilinearT3bis}.

\bigskip

We therefore start by recasting the Strichartz estimates \eqref{StrichartzLoc}-\eqref{StrichartzLoc2} as a multilinear restriction estimate.
\begin{lemma}\label{LEM:multilinearT31}
Let ${\displaystyle \frac{2(d+2)}{d}<p<m(d+2)}$ and $q$ such that ${\displaystyle \frac2q+\frac{2m}p=1}$. Then for any $N_0\sim N_1\ge N_2\ge\dots\ge N_{2m+1}$, any $\CC_\ell,\CC_{\wt\ell}\in\CC(N_2)$, and any $U_0,\dots,U_{2m+1}\in L^2(\T^d)$ with non-negative Fourier coefficients, it holds
\begin{multline*}
\sup_{\nu\in\Z}\sum_{\substack{n_0,...,n_{2m+1}\in\Z^d\\\sum_{j=0}^{2m+1}\pm_jn_j=0\\ |n_j|\sim N_j,~~n_0\in\CC_\ell,~n_1\in\CC_{\wt\ell}\\ \sum_{j=0}^{2m+1}\pm_j|n_j|^2=\nu}}\prod_{j=0}^{2m+1}\ft U_j(n_j)\\
\les \|\P_{\CC_\ell}\P_{N_0}U_0\|_{L^2}\|\P_{\CC_{\wt\ell}}\P_{N_1}U_1\|_{L^2}\\
 \times N_2^{s_c+(2m-1)\eps}\|\P_{N_2}U_2\|_{L^2}\prod_{j=3}^{2m+1}N_j^{s_c-\eps}\|\P_{N_j}U_j\|_{L^2}
\end{multline*}
with ${\displaystyle \eps=\frac{d+2}p-\frac1m>0}$.

The same estimate holds exchanging the roles of $N_0$ and $N_2$.
\end{lemma}
\begin{proof}
We observe that for any $\nu\in\Z$,
\begin{align*}
&\sum_{\substack{n_0,...,n_{2m+1}\in\Z^d\\\sum_{j=0}^{2m+1}\pm_jn_j=0\\ |n_j|\sim N_j,~~n_0\in\CC_\ell,~n_1\in\CC_{\wt\ell}\\ \sum_{j=0}^{2m+1}\pm_j|n_j|^2=\nu}}\prod_{j=0}^{2m+1}\ft U_j(n_j)\\
&\qquad= \sum_{\substack{n_0,...,n_{2m+1}\in\Z^d\\\sum_{j=0}^{2m+1}\pm_jn_j=0\\ |n_j|\sim N_j,~~n_0\in\CC_\ell,~n_1\in\CC_{\wt\ell}}}\prod_{j=0}^{2m+1}\ft U_j(n_j)\int_0^{2\pi}e^{it\big(\sum_{j=0}^{2m+1}\pm_j|n_j|^2-\nu\big)}dt\\
&\qquad=\int_0^{2\pi}\int_{\T^d}e^{-it\nu}\big(e^{it\Delta}\P_{\CC_\ell}\P_{N_0}U_0\big)^{\pm_0}\big(e^{it\Delta}\P_{\CC_{\wt\ell}}\P_{N_1}U_1\big)^{\pm_1}\prod_{j=2}^{2m+1}\big(e^{it\Delta}\P_{N_j}U_j\big)^{\pm_j}dxdt.
\end{align*}
Thus, since we have $p,q>\frac{2(d+2)}d$ from our choice of $p$ and $q$, we have from \eqref{StrichartzLoc}-\eqref{StrichartzLoc2} that
\begin{align*}
&\sup_{\nu\in\Z}\sum_{\substack{n_0,...,n_{2m+1}\in\Z^d\\\sum_{j=0}^{2m+1}\pm_jn_j=0\\ |n_j|\sim N_j,~~n_0\in\CC_\ell,~n_1\in\CC_{\wt\ell}\\ \sum_{j=0}^{2m+1}\pm_j|n_j|^2=\nu}}\prod_{j=0}^{2m+1}\ft U_j(n_j)\\
&\qquad\les \big\|e^{it\Delta}\P_{\CC_\ell}\P_{N_0}U_0\big\|_{L^q(\T\times\T^{d})}\big\|e^{it\Delta}\P_{\CC_{\wt\ell}}\P_{N_1}U_1\big\|_{L^q(\T\times\T^{d})}\prod_{j=2}^{2m+1}\big\|e^{it\Delta}\P_{N_j}U_j\big\|_{L^p(\T\times\T^{d})}\\
&\qquad\les \big\|\P_{\CC_\ell}\P_{N_0}U_0\big\|_{L^2}\big\|\P_{\CC_{\wt\ell}}\P_{N_1}U_1\big\|_{L^2}\\
&\qquad\qquad\times N_2^{2\big(\frac{d}2-\frac{d+2}q\big)+\big(\frac{d}2-\frac{d+2}p\big)}\|\P_{N_2}U_2\|_{L^2}\prod_{j=3}^{2m+1}N_j^{\frac{d}2-\frac{d+2}p}\|\P_{N_j}U_j\|_{L^2}.
\end{align*}
The case $N_1\sim N_2\gtrsim N_0$ is estimated similarly. This proves Lemma~\ref{LEM:multilinearT31} since ${ \frac{d}2-\frac{d+2}p = s_c -\big(\frac{d+2}p-\frac1m\big)}$ and ${ 2\big(\frac{d}2-\frac{d+2}q\big)+\big(\frac{d}2-\frac{d+1}p\big) = s_c+(2m-1)\big(\frac{d+2}p-\frac1m\big)}$ by the definition of $s_c=\frac{d}2-\frac1m$ and our choice of $q$ such that $\frac2q+\frac{2m}p=1$.
\end{proof}

We now establish a refined version of Lemma~\ref{LEM:multilinearT31} in the spirit of \cite[Corollary 3.4]{HTT}.
\begin{lemma}\label{LEM:multilinearT32}
Let $ p>m(d+2)$ and ${\displaystyle \frac{2(d+2)}d<q<\frac{3m(d+2)}{md+1}}$ be such that ${\displaystyle \frac3q+\frac{2m-1}p=1}$. Then there exists $0<\theta<1-\frac{2(d+2)}{dq}$ such that for any $N_0\sim N_1\ge N_2\ge\dots\ge N_{2m+1}$, any $\CC_\ell,\CC_{\wt\ell}\in\CC(N_2)$, let $M=(N_2^2N_1^{-1})\vee 1$ and $\{\RR_\alpha\}$, $|\alpha|\sim M^{-1}N_1$, be a partition of $\CC_\ell$ into almost disjoint strips of width $M$ orthogonal to the center $n_\ell$ of $\CC_\ell$:
$$\RR_\alpha = \big\{n\in \CC_\ell,~~n\cdot n_\ell\in[\alpha|n_\ell|M,(\alpha+1)|n_\ell|M)\big\},$$
and similarly for $\CC_{\wt\ell}$. Then for any $U_0,\dots,U_{2m+1}\in L^2(\T^d)$ with non-negative Fourier coefficients, it holds
\begin{align*}
&\sup_{\nu\in\Z}\sum_{\substack{|\alpha|,|\beta|\sim M^{-1}N_1\\ |\pm_0 \alpha\pm_1\beta|\les 1}}\sum_{\substack{n_0,...,n_{2m+1}\in\Z^d\\\sum_{j=0}^{2m+1}\pm_jn_j=0\\ |n_j|\sim N_j,~~n_0\in\RR_\alpha,~n_1\in\RR_{\beta}\\ \sum_{j=0}^{2m+1}\pm_j|n_j|^2=\nu}}\prod_{j=0}^{2m+1}\ft U_j(n_j)\\
&\qquad\les \|\P_{\CC_\ell}\P_{N_0}U_0\|_{L^2}\|\P_{\CC_{\wt\ell}}\P_{N_1}U_1\|_{L^2}\\
&\qquad\qquad\times N_2^{s_c-(2m-1)\eps}\Big(\frac1{N_2}+\frac{N_2}{N_1}\Big)^{\theta}\|\P_{N_2}U_2\|_{L^2}\prod_{j=3}^{2m+1}N_j^{s_c+\eps}\|\P_{N_j}U_j\|_{L^2}
\end{align*}
with $\eps:=\frac1m-\frac{d+2}{p}>0$.
The same estimate holds exchanging the roles of $N_0$ and $N_2$.
\end{lemma}

\begin{proof}
We proceed as in Lemma~\ref{LEM:multilinearT31}, but on top of \eqref{StrichartzLoc}-\eqref{StrichartzLoc2} we also use the refined Strichartz estimate from \cite[Corollary 3.4]{HTT}. Namely, using Bernstein inequality, we have
\begin{align}\label{Bernstein}
\big\|\P_{\RR_\alpha} e^{it\Delta}\phi\big\|_{L^\infty(\T\times\T^d)}\les M^\frac12 N_2^{\frac{d-1}{2}}\|\P_{\RR_\alpha}\phi\|_{L^2_x}.
\end{align}
Letting $p,q$ be as in the statement of Lemma~\ref{LEM:multilinearT32} and taking $\frac{2(d+2)}d<q_0<q$ we can interpolate between \eqref{Bernstein} and \eqref{StrichartzLoc2} for $q_0$ to get
\begin{align}\label{StrichartzLoc4}
\big\|\P_{\RR_\alpha} e^{it\Delta}\phi\big\|_{L^q(\T\times\T^d)}&\les N_2^{\frac{q_0}{q}\big(\frac{d}2-\frac{d+2}{q_0}\big)+\frac{d-1}{2}(1-\frac{q_0}q)}M^{\frac12(1-\frac{q_0}q)}\|\P_{\RR_\alpha}\phi\|_{L^2_x}\notag\\
&\sim N_2^{\frac{d}2-\frac{d+2}q}\Big(\frac1{N_2}+\frac{N_2}{N_1}\Big)^{\frac12(1-\frac{q_0}q)}\|\P_{\RR_\alpha}\phi\|_{L^2_x}.
\end{align}

Thus, with $p$ and $q$ as in the statement, proceeding as in the proof of Lemma~\ref{LEM:multilinearT31}, we estimate using \eqref{StrichartzLoc}-\eqref{StrichartzLoc4} and then Cauchy-Schwarz inequality in $\alpha,\beta$:
\begin{align*}
&\sup_{\nu\in\Z}\sum_{\substack{|\alpha|,|\beta|\sim M^{-1}N_1\\ |\pm_0 \alpha\pm_1\beta|\les 1}}\sum_{\substack{n_0,...,n_{2m+1}\in\Z^d\\\sum_{j=0}^{2m+1}\pm_jn_j=0\\ |n_j|\sim N_j,~~n_0\in\RR_\alpha,~n_1\in\RR_{\beta}\\ \sum_{j=0}^{2m+1}\pm_j|n_j|^2=\nu}}\prod_{j=0}^{2m+1}\ft U_j(n_j)\\
&\qquad\les \sum_{\substack{|\alpha|,|\beta|\sim M^{-1}N_1\\ |\pm_0 \alpha\pm_1\beta|\les 1}}\big\|e^{it\Delta}\P_{\RR_\alpha}\P_{N_0}U_0\big\|_{L^q(\T\times\T^d)}\big\|e^{it\Delta}\P_{\RR_{\beta}}\P_{N_1}U_1\big\|_{L^q(\T\times\T^d)}\\
&\qquad\qquad\times\big\|e^{it\Delta}\P_{N_2}U_2\big\|_{L^q(\T\times\T^d)}\prod_{j=3}^{2m+1}\big\|e^{it\Delta}\P_{N_j}U_j\big\|_{L^p(\T\times\T^d)}\\
&\qquad\les \sum_{\substack{|\alpha|,|\beta|\sim M^{-1}N_1\\ |\pm_0 \alpha\pm_1\beta|\les 1}}\big\|\P_{\RR_{\alpha}}\P_{N_0}U_0\big\|_{L^2}\big\|\P_{\RR_{\beta}}\P_{N_1}U_1\big\|_{L^2}\\
&\qquad\qquad\times N_2^{3\big(\frac{d}2-\frac{d+2}q\big)}\Big(\frac1{N_2}+\frac{N_2}{N_1}\Big)^{1-\frac{q_0}q}\|\P_{N_2}U_2\|_{L^2}\prod_{j=3}^{2m+1}N_j^{\frac{d}2-\frac{d+2}p}\|\P_{N_j}U_j\|_{L^2}\\
&\qquad\les \|\P_{\CC_\ell}\P_{N_0}U_0\|_{L^2}\|\P_{\CC_{\wt\ell}}\P_{N_1}U_1\|_{L^2}\\
&\qquad\qquad\times N_2^{3\big(\frac{d}2-\frac{d+2}q\big)}\Big(\frac1{N_2}+\frac{N_2}{N_1}\Big)^{1-\frac{q_0}q}\|\P_{N_2}U_2\|_{L^2}\prod_{j=3}^{2m+1}N_j^{\frac{d}2-\frac{d+2}p}\|\P_{N_j}U_j\|_{L^2}.
\end{align*}
The case $N_1\sim N_2\gtrsim N_0$ is treated similarly. This proves Lemma~\ref{LEM:multilinearT32} with $\theta=1-\frac{q_0}q$.
\end{proof}

We can now recast \cite[Proposition 3.5]{HTT} and \cite[Theorem 3.4]{Wang} also as a suitable multilinear restriction estimate.
\begin{lemma}\label{LEM:multilinearT3}
For any $U_0\in H^{-s_c}(\T^d)$ and $U_1\dots,U_{2m+1}\in H^{s_c}(\T^d)$ with non-negative Fourier coefficients, it holds 
\begin{align*}
&\sum_{N_0\sim N_1\gtrsim N_2\ge\dots\ge N_{2m+1}}\sum_{\CC_\ell,\CC_{\wt\ell}\in\CC(N_2),~|\ell-\wt\ell|\les 1}\sum_{\substack{n_0,...,n_{2m+1}\in\Z^d\\\sum_{j=0}^{2m+1}\pm_jn_j=0\\ |n_j|\sim N_j,~~n_0\in\CC_\ell,~n_1\in\CC_{\wt\ell}}}\langle\sum_{j=0}^{2m+1}\pm_j|n_j|^2\rangle^{-\rho}\prod_{j=0}^{2m+1}\ft U_j(n_j)\\
&\qquad\qquad\les \|U_0\|_{H^{-s_c}}\prod_{j=1}^{2m+1}\|U_j\|_{H^{s_c}}
\end{align*}
and
\begin{align*}
&\sum_{N_1\sim N_2\gtrsim N_0,N_2\ge\dots\ge N_{2m+1}}\sum_{\CC_\ell,\CC_{\wt\ell}\in\CC(N_0),~|\ell-\wt\ell|\les 1}\sum_{\substack{n_0,...,n_{2m+1}\in\Z^d\\\sum_{j=0}^{2m+1}\pm_jn_j=0\\ |n_j|\sim N_j,~~n_1\in\CC_\ell,~n_2\in\CC_{\wt\ell}}}\langle\sum_{j=0}^{2m+1}\pm_j|n_j|^2\rangle^{-\rho}\prod_{j=0}^{2m+1}\ft U_j(n_j)\\
&\qquad\qquad\les \|U_0\|_{H^{-s_c}}\prod_{j=1}^{2m+1}\|U_j\|_{H^{s_c}}.
\end{align*}
\end{lemma}
\begin{proof}
We start by rewriting
\begin{align*}
\sum_{\substack{n_0,...,n_{2m+1}\in\Z^d\\\sum_{j=0}^{2m+1}\pm_jn_j=0\\ |n_j|\sim N_j,~~n_0\in\CC_\ell,~n_1\in\CC_{\wt\ell}}}\langle\sum_{j=0}^{2m+1}\pm_j|n_j|^2\rangle^{-\rho}\prod_{j=0}^{2m+1}\ft U_j(n_j)=\sum_{\nu\in\Z}\sum_{\substack{n_0,...,n_{2m+1}\in\Z^d\\\sum_{j=0}^{2m+1}\pm_jn_j=0\\ |n_j|\sim N_j,~~n_0\in\CC_\ell,~n_1\in\CC_{\wt\ell}\\ \sum_{j=0}^{2m+1}\pm_j|n_j|^2=\nu}}\jb{\nu}^{-\rho}\prod_{j=0}^{2m+1}\ft U_j(n_j).
\end{align*}
Since $\rho>1$, taking $0<\delta<\rho-1$ and $p,q$ as in Lemma~\ref{LEM:multilinearT31} with $p$ close enough to $m(d+2)$ such that $\eps:=\frac{d+2}{p}-\frac1m$ satisfies $(2m-1)\eps<2\delta$, we can estimate the sum on $|\nu|\gtrsim N_2^2$ by using Lemma~\ref{LEM:multilinearT31}:
\begin{align*}
&\sum_{|\nu|\gtrsim N_2^2}\sum_{\substack{n_0,...,n_{2m+1}\in\Z^d\\\sum_{j=0}^{2m+1}\pm_jn_j=0\\ |n_j|\sim N_j,~~n_0\in\CC_\ell,~n_1\in\CC_{\wt\ell}\\ \sum_{j=0}^{2m+1}\pm_j|n_j|^2=\nu}}\jb{\nu}^{-\rho}\prod_{j=0}^{2m+1}\ft U_j(n_j)\\
&\qquad\les N_2^{-2\delta}\sup_{\nu\in\Z}\sum_{\substack{n_0,...,n_{2m+1}\in\Z^d\\\sum_{j=0}^{2m+1}\pm_jn_j=0\\ |n_j|\sim N_j,~~n_0\in\CC_\ell,~n_1\in\CC_{\wt\ell}\\ \sum_{j=0}^{2m+1}\pm_j|n_j|^2=\nu}}\prod_{j=0}^{2m+1}\ft U_j(n_j)\\
&\qquad\les \|\P_{\CC_\ell}\P_{N_0}U_0\|_{L^2}\|\P_{\CC_{\wt\ell}}\P_{N_1}U_1\|_{L^2}N_2^{s_c+(2m-1)\eps-2\delta}\|\P_{N_2}U_2\|_{L^2}\prod_{j=3}^{2m+1}N_j^{s_c-\eps}\|\P_{N_j}U_j\|_{L^2}.
\end{align*}
With our choice of $\eps,\delta$ we can thus use Cauchy-Schwarz inequality to sum on $\ell,\wt\ell$ then $N_0\sim N_1\gtrsim N_2\ge\dots\ge N_{2m+1}$ to get
\begin{align*}
&\sum_{\substack{N_0\sim N_1\gtrsim N_2\ge\dots\ge N_{2m+1}\\\CC_\ell,\CC_{\wt\ell}\in\CC(N_2)\\|\ell-\wt\ell|\les 1}}\sum_{|\nu|\gtrsim N_2^2}\sum_{\substack{n_0,...,n_{2m+1}\in\Z^d\\\sum_{j=0}^{2m+1}\pm_jn_j=0\\ |n_j|\sim N_j,~~n_0\in\CC_\ell,~n_1\in\CC_{\wt\ell}\\ \sum_{j=0}^{2m+1}\pm_j|n_j|^2=\nu}}\jb{\nu}^{-\rho}\prod_{j=0}^{2m+1}\ft U_j(n_j)\\
&\qquad\les \sum_{\substack{N_0\sim N_1\gtrsim N_2\ge\dots\ge N_{2m+1}\\\CC_\ell,\CC_{\wt\ell}\in\CC(N_2)\\|\ell-\wt\ell|\les 1}}\|\P_{\CC_\ell}\P_{N_0}U_0\|_{L^2}\|\P_{\CC_{\wt\ell}}\P_{N_1}U_1\|_{L^2}\\
&\qquad\qquad\times N_2^{s_c+(2m-1)\eps-2\delta}\|\P_{N_2}U_2\|_{L^2}\prod_{j=3}^{2m+1}N_j^{s_c-\eps}\|\P_{N_j}U_j\|_{L^2}\\
&\qquad\les \sum_{N_0\sim N_1\gtrsim N_2\ge\dots\ge N_{2m+1}}N_0^{-s_c}\|\P_{N_0}U_0\|_{L^2}N_1^{s_c}\|\P_{N_1}U_1\|_{L^2}\\
&\qquad\qquad\times N_2^{s_c+(2m-1)\eps-2\delta}\|\P_{N_2}U_2\|_{L^2}\prod_{j=3}^{2m+1}N_j^{s_c-\eps}\|\P_{N_j}U_j\|_{L^2}\\
&\qquad\les \|U_0\|_{H^{-s_c}}\prod_{j=1}^{2m+1}\|U_j\|_{H^{s_c}}.
\end{align*}

It remains to estimate the contribution from $|\nu|\ll N_2^2$. Following \cite{HTT}, in order to use Lemma~\ref{LEM:multilinearT32}, for $\CC_\ell\in\CC(N_2)$, we write $\CC_\ell = n_\ell+[-N_2,N_2]^d$ for some $|n_\ell|\sim N_1$, and for $$M=(N_2^2N_1^{-1})\vee 1,$$ we let $\{\RR_\alpha\}$, $|\alpha|\sim M^{-1}N_1$, be a partition of $\CC_\ell$ into almost disjoint strips orthogonal to $n_\ell$ of width $M$:
$$\RR_\alpha = \big\{n\in \CC_\ell,~~n\cdot n_\ell\in[\alpha|n_\ell|M,(\alpha+1)|n_\ell|M)\big\},$$
and similarly for a partition $\{\RR_\beta\}$ of $\CC_{\wt\ell}$.

For $|\nu|\ll N_2^2$, we can thus decompose
\begin{align*}
\sum_{\substack{n_0,...,n_{2m+1}\in\Z^d\\\sum_{j=0}^{2m+1}\pm_jn_j=0\\ |n_j|\sim N_j,~~n_0\in\CC_\ell,~n_1\in\CC_{\wt\ell}\\ \sum_{j=0}^{2m+1}\pm_j|n_j|^2=\nu}}\prod_{j=0}^{2m+1}\ft U_j(n_j)&=\sum_{\alpha,\beta}\sum_{\substack{n_0,...,n_{2m+1}\in\Z^d\\\sum_{j=0}^{2m+1}\pm_jn_j=0\\ |n_j|\sim N_j,~~n_0\in\RR_\alpha,~n_1\in\RR_{\beta}\\ \sum_{j=0}^{2m+1}\pm_j|n_j|^2=\nu}}\prod_{j=0}^{2m+1}\ft U_j(n_j).
\end{align*}
Now, the point is that, as in \cite[Proposition 3.5]{HTT} and \cite[Theorem 3.4]{Wang}, for a given $\alpha$, there are at most $O(1)$ values of $\beta$ for which the sum above is non empty. Indeed, for $n_0\in\RR_\alpha$, it holds
\begin{align*}
|n_0|^2 &= |n_0-n_\ell|^2+2n_0\cdot n_\ell+|n_\ell|^2\\
&=\underset{O(N_2^2)}{\underbrace{|n_0-n_\ell|^2}}+\underset{= \alpha^2M^2+O(\alpha M^2)}{\underbrace{\frac1{|n_\ell|^2}(n_0\cdot n_\ell)^2}} -\underset{=O(N_2^2)}{\underbrace{\frac1{|n_\ell|^2}\big((n_0-n_\ell)\cdot n_\ell)\big)^2}}\\
&=\alpha^2M^2+O(\alpha M^2)
\end{align*}
since $N_2^2\les |\alpha|M^2$ with our choice of $M$ and that $|\alpha|\sim M^{-1}N_1$. Similarly, we have $|n_1|^2=\beta^2M^2+O(\beta M^2)$ for $n_1\in\RR_{\beta}$, and thus from the constraint ${\displaystyle \sum_{j=0}^{2m+1}\pm_j|n_j|^2=\nu}$ with $|\nu|\ll N_2^2$, we see that we must have $\pm_0|n_0|^2\pm_1|n_1|^2=O(N_2^2)$, i.e.
$$(\pm_0 \alpha^2\pm_1\beta^2)M^2=O((|\alpha|+|\beta|)M).$$
Thus $|\pm_0 \alpha\pm_1\beta|\les 1$. This yields
\begin{align*}
\sum_{\substack{n_0,...,n_{2m+1}\in\Z^d\\\sum_{j=0}^{2m+1}\pm_jn_j=0\\ |n_j|\sim N_j,~~n_0\in\CC_\ell,~n_1\in\CC_{\wt\ell}\\ \sum_{j=0}^{2m+1}\pm_j|n_j|^2=\nu}}\prod_{j=0}^{2m+1}\ft U_j(n_j)&=\sum_{\substack{|\alpha|,|\beta|\sim M^{-1}N_1\\ |\pm_0 \alpha\pm_1\beta|\les 1}}\sum_{\substack{n_0,...,n_{2m+1}\in\Z^d\\\sum_{j=0}^{2m+1}\pm_jn_j=0\\ |n_j|\sim N_j,~~n_0\in\RR_\alpha,~n_1\in\RR_{\beta}\\ \sum_{j=0}^{2m+1}\pm_j|n_j|^2=\nu}}\prod_{j=0}^{2m+1}\ft U_j(n_j).
\end{align*}
Thus, taking $p,q,\eps,\theta$ as in Lemma~\ref{LEM:multilinearT32}, we can use Lemma~\ref{LEM:multilinearT32} and then Cauchy-Schwarz inequality to sum on $N_j$:
\begin{align*}
&\sum_{\substack{N_0\sim N_1\gtrsim N_2\ge\dots\ge N_{2m+1}\\\CC_\ell,\CC_{\wt\ell}\in\CC(N_2)\\|\ell-\wt\ell|\les 1}}\sum_{|\nu|\ll N_2^2}\sum_{\substack{n_0,...,n_{2m+1}\in\Z^d\\\sum_{j=0}^{2m+1}\pm_jn_j=0\\ |n_j|\sim N_j,~~n_0\in\CC_\ell,~n_1\in\CC_{\wt\ell}\\ \sum_{j=0}^{2m+1}\pm_j|n_j|^2=\nu}}\jb{\nu}^{-\rho}\prod_{j=0}^{2m+1}\ft U_j(n_j)\\
&\qquad\les \sum_{N_0\sim N_1\gtrsim N_2\ge\dots\ge N_{2m+1}}\|\P_{N_0}U_0\|_{L^2}\|\P_{N_1}U_1\|_{L^2}\\
&\qquad\qquad\qquad\times N_2^{s_c-(2m-1)\eps}\Big(\frac1{N_2}+\frac{N_2}{N_1}\Big)^{\theta}\|\P_{N_2}U_2\|_{L^2}\prod_{j=3}^{2m+1}N_j^{s_c+\eps}\|\P_{N_j}U_j\|_{L^2}\\
&\les \sum_{N_0\sim N_1\gtrsim N_2}N_0^{-s_c}\|\P_{N_0}U_0\|_{L^2}N_1^{s_c}\|\P_{N_1}U_1\|_{L^2}\\
&\qquad\qquad\qquad\times N_2^{s_c-(2m-1)\eps}\Big(\frac1{N_2}+\frac{N_2}{N_1}\Big)^{\theta}\|\P_{N_2}U_2\|_{L^2}\prod_{j=3}^{2m+1}N_2^\eps\|U_j\|_{H^{s_c}}\\
&\qquad\les \|U_0\|_{H^{-s_c}}\prod_{j=1}^{2m+1}\|U_j\|_{H^{s_c}}.
\end{align*}
The case $N_1\sim N_2\gtrsim N_0$ is treated similarly. This proves Lemma~\ref{LEM:multilinearT3}.
\end{proof}

With the previous estimates at hand, we can finally give the proof of the main multilinear estimate \eqref{multiT34}.

\begin{proof}[Proof of Proposition~\ref{PROP:multilinearT3bis}]
Following the argument in Subsubsection~\ref{SUBSUB:NLS} it suffices to estimate
\begin{align}\label{multiT31}
\sum_{N_0\sim N_1\ge N_2\ge...\ge N_{2m+1}}\sum_{\substack{\CC_\ell,\CC_{\wt\ell}\in\CC(N_2)\\|\wt \ell-\ell|\les 1}}\int_0^T\int_{\T^d}\P_{\CC_\ell}\P_{N_0}u_0^{\pm_0}\P_{\CC_{\wt \ell}}\P_{N_1}u_1^{\pm_1}\prod_{j=2}^{2m+1}\P_{N_j}u_j^{\pm_j}dxdt
\end{align}
and
\begin{align}\label{multiT32}
\sum_{N_2\sim N_1\ge N_0\ge...\ge N_{2m+1}}\sum_{\substack{\CC_\ell,\CC_{\wt\ell}\in\CC(N_0)\\|\wt\ell-\ell|\les 1}}\int_0^T\int_{\T^d}\P_{\CC_\ell}\P_{N_2}u_2^{\pm_2}\P_{\CC_{\wt \ell}}\P_{N_1}u_1^{\pm_1}\prod_{j=0,2}^{2m+1}\P_{N_j}u_j^{\pm_j}dxdt.
\end{align}

From the embedding $Y^{-s_c}_T\subset X^{-s_c,b,2}_T$ with $b=\big(\frac1\gamma\big)'>2$, i.e. $b'=\frac1\gamma<2$ (recall that $\gamma>\frac12$), and the atomic structure of $U^2$ and $U^b$, it suffices to consider the case where for $j=1,...,2m+1$, $u_j$ are $X^{s}_T$-atoms and $u_0$ is a $X^{-s_c,b,2}$-atom:
\begin{align*}
u_j(t,x)=\sum_{k=1}^{K^{(j)}}\mathbf{1}_{[t_{k-1}^{(j)};t_k^{(j)})}e^{iW_t\Delta}\phi_k^{(j)},\qquad \{t_k^{(j)}\}_k\in\ZZ(T),\qquad\|\phi_k^{(j)}\|_{\ell^2_kH^{s_c}}\le 1
\end{align*}
for any $j=1,\dots,2m+1$, and
\begin{align*}
u_0(t,x)=\sum_{k=1}^{K^{(0)}}\mathbf{1}_{[t_{k-1}^{(0)};t_k^{(0)})}e^{iW_t\Delta}\phi_k^{(0)}, \qquad \{t_k^{(0)}\}_k\in\ZZ(T),\qquad \|\phi_k^{(0)}\|_{\ell^b_k H^{-s_c}}\le 1.
\end{align*}

Then, for given $k_0,\dots,k_{2m+1}$, letting 
$$I_{k_0,\dots,k_{2m+1}}=\bigcap_{j=0}^{2m+1}[t_{k_j-1}^{(j)};t_{k_j}^{(j)}),$$
using Plancherel theorem and the occupation time formula \eqref{occupation}, we estimate for any $N_0,\dots,N_{2m+1}$, $\ell,\wt\ell$:
\begin{align*}
&\Big|\int_{I_{k_0,\dots,k_{2m+1}}}\int_{\T^d}\big(\P_{\CC_\ell}\P_{N_0}e^{iW_t\Delta}\phi_{k_0}^{(0)}\big)^{\pm_0}\big(\P_{\CC_{\wt \ell}}\P_{N_1}e^{iW_t\Delta}\phi_{k_1}^{(1)}\big)^{\pm_1}\prod_{j=2}^{2m+1}\big(\P_{N_j}e^{iW_t\Delta}\phi_{k_j}^{(j)}\big)^{\pm_j}dxdt\Big|\\
&=\Big|\sum_{\substack{n_0,...,n_{2m+1}\in\Z^d\\\sum_{j=0}^{2m+1}\pm_jn_j=0}}\int_{I_{k_0,\dots,k_{2m+1}}} e^{iW_t\sum_{j=0}^{2m+1}\pm_j|n_j|^2}dt\\
&\qquad\times\widehat{\P_{\CC_\ell}\P_{N_0}\phi_{k_0}^{(0)}}(n_0)^{\pm_0}\widehat{ \P_{\CC_{\wt\ell}}\P_{N_1}\phi_{k_1}^{(1)}}(n_1)^{\pm_1}\prod_{j=2}^{2m+1}\widehat{\P_{N_j}\phi_{k_j}^{(j)}}(n_j)^{\pm_j}\Big|\\
&= \Big|\sum_{\substack{n_0,...,n_{2m+1}\in\Z^d\\\sum_{j=0}^{2m+1}\pm_jn_j=0}}\widehat{\frac{d\mu_{I_{k_0,\dots,k_{2m+1}}}}{dz}}\big(\sum_{j=0}^{2m+1}\pm_j|n_j|^2\big)\\
&\qquad\times\widehat{\P_{\CC_\ell}\P_{N_0}\phi_{k_0}^{(0)}}(n_0)^{\pm_0}\widehat{\P_{\CC_{\wt\ell}}\P_{N_1}\phi_{k_1}^{(1)}}(n_1)^{\pm_1}\prod_{j=2}^{2m+1}\widehat{ \P_{N_j}\phi_{k_j}^{(j)}}(n_j)^{\pm_j}\Big|\\
&\les |I_{k_0,\dots,k_{2m+1}}|^\gamma \sum_{\substack{n_0,...,n_{2m+1}\in\Z^d\\\sum_{j=0}^{2m+1}\pm_jn_j=0\\ |n_j|\sim N_j,~~n_0\in\CC_\ell,~n_1\in\CC_{\wt\ell}}}\langle\sum_{j=0}^{2m+1}\pm_j|n_j|^2\rangle^{-\rho}\prod_{j=0}^{2m+1}|\ft \phi_{k_j}^{(j)}(n_j)|,
\end{align*}
where the last step follows from assumption \ref{A0gr}. Now, using Lemma~\ref{LEM:multilinearT3}, we can sum on $N_j$ and $\ell,\wt\ell$, to get
\begin{multline*}
\sum_{\substack{N_0\sim N_1\ge N_2\ge...\ge N_{2m+1}\\\CC_\ell,\CC_{\wt\ell}\in\CC(N_2)\\|\wt \ell-\ell|\les 1}}\sum_{\substack{n_0,...,n_{2m+1}\in\Z^d\\\sum_{j=0}^{2m+1}\pm_jn_j=0\\ |n_j|\sim N_j,~~n_0\in\CC_\ell,~n_1\in\CC_{\wt\ell}}}\langle\sum_{j=0}^{2m+1}\pm_j|n_j|^2\rangle^{-\rho}\prod_{j=0}^{2m+1}|\ft \phi_{k_j}^{(j)}(n_j)|\\
\les \|\phi_{k_0}^{(0)}\|_{H^{-s_c}}\prod_{j=1}^{2m+1}\|\phi_{k_j}^{(j)}\|_{H^{s_c}},
\end{multline*}
and similarly for the case $N_1\sim N_2\gtrsim N_0$. It remains to sum on $k_j$ using H\"older's inequality to get
\begin{align*}
\Big|\int_0^T\int_{\T^d}\prod_{j=0}^{2m+1}u_j^{\pm_j}dxdt\Big|&\les \sum_{k_0,\dots,k_{2m+1}}|I_{k_0,\dots,k_{2m+1}}|^\gamma \|\phi_{k_0}^{(0)}\|_{H^{-s_c}}\prod_{j=1}^{2m+1}\|\phi_{k_j}^{(j)}\|_{H^{s_c}}\\
&\les \Big\||I_{k_0,\dots,k_{2m+1}}|^\gamma\Big\|_{\ell^2_{k_1,\dots,k_{2m+1}}\ell^{b'}_{k_0}}\|\phi_{k_0}^{(0)}\|_{\ell^b_{k_0} H^{-s_c}}\prod_{j=1}^{2m+1}\|\phi_{k_j}^{(j)}\|_{\ell^2_{k_j} H^{s_c}}\\
&\les \Big\||I_{k_0,\dots,k_{2m+1}}|^\gamma\Big\|_{\ell^{\frac1\gamma}_{k_0,\dots,k_{2m+1}}}=T^\gamma.
\end{align*}
In the last two steps we used that $b'=\frac1\gamma<2$, together with the fact that the intervals $\{I_{k_0,\dots,k_{2m+1}}\}$ have disjoint interiors and cover $[0,T)$. This finally proves \eqref{multiT34} in view of the definition of the atomic spaces $X^{s_c}_T$ and $X^{-s_c,b,2}_T$.
\end{proof}

\begin{remark}\label{REM:NLSsphere}
\rm
A straightforward adaptation of the arguments in \cite{BGT2,Herr} following the proof of Theorem~\ref{THM:NLScrit} above allows to extend Theorem~\ref{THM:NLScrit} to the cases $\M=\SS^2$, $m=1$, $s>\frac14$; and $\M=\SS^3$, $m=2$, $s\ge s_c=1$, respectively.
\end{remark}

\section{The periodic cubic Wick-ordered modulated NLS equation}\label{SEC:NLSFL}
In this last section we give the proof of Theorem~\ref{THM:NLSFL}. First, as usual \cite{GrunrockHerr,GuoOh}, we rewrite the Wick-ordered cubic nonlinearity in Fourier as
\begin{align*}
\NN(u)&:=(|u|^2-2\|u\|_{L^2}^2)u= \sum_{n_0\in\Z}e^{in_0x}\Big\{\sum_{\substack{n_1-n_2+n_3=n_0\\
n_2\neq n_1,n_3}}\ft u_{n_1}\cj{\ft u_{n_2}}\ft u_{n_3} + |\ft u_{n_0}|^2\ft u_{n_0}\Big\}\\
& =: \NN_0(u)+\RR(u).
\end{align*}
Abusing notations, we also write $\NN_0(u,v,w)$ and $\RR(u,v,w)$ for the trilinear forms such that $\NN_0(u)=\NN_0(u,u,u)$ and $\RR(u)=\RR(u,u,u)$.
 
Let then $s\ge 0$ and $r\in[1,\infty]$. As in the proof of Theorem~\ref{THM:NLScrit}, using assumption \ref{A0gr} with $\gamma>\frac12$ and $\rho>\frac1{r'}$, we get that for any interval $I\subset\R$, any $u_j\in \F L^{s,r}(\T)$, $j=1,2,3$, and $u_0\in\F L^{-s,r'}(\T)$, it holds
\begin{align*}
&\Big|\int_I\int_\T \NN_0\big(e^{iW_t\dx^2}u_1,e^{iW_t\dx^2}u_2,e^{iW_t\dx^2}u_3\big)\cj{e^{iW_t\dx^2}u_0} dxdt\Big|\\
&\qquad= \Big|\sum_{\substack{n_0,...,n_3\in\Z\\\sum_{j=0}^3\pm_jn_j=0\\n_2\neq n_1,n_3}}\widehat{\frac{d\mu_{I}}{dz}}\big(\sum_{j=0}^3\pm_jn_j^2\big)\prod_{j=0}^3\ft u_j(n_j)^{\pm_j}\Big|\\
&\qquad\les |I|^\gamma\sum_{\substack{n_0,...,n_3\in\Z\\\sum_{j=0}^3\pm_jn_j=0\\n_2\neq n_1,n_3}}\jb{(n_1-n_2)(n_0-n_2)}^{-\rho}\prod_{j=0}^3|\ft u_j(n_j)|
\intertext{where we used the exact factorization of ${\displaystyle \sum_{j=0}^3\pm_jn_j^2}$ under the convolution constraint ${\displaystyle\sum_{j=0}^3\pm_jn_j=0}$. Using now H\"older and Young inequalities together with ${\displaystyle \jb{n_0}^s\les \sum_{j=1}^3\jb{n_j}^s}$ under the convolution constraint since $s\ge 0$, we can continue with}
&\qquad\les |I|^\gamma\|u_0\|_{\F L^{-s,r'}}\Big\|\jb{n_0}^s\sum_{\substack{n_1,n_2,n_3\in\Z\\n_1-n_2+n_3=n_0\\n_2\neq n_1,n_3}}\jb{(n_1-n_2)(n_0-n_2)}^{-\rho}\prod_{j=1}^3\ft u_j(n_j)\Big\|_{\ell^r_{n_0}}\\
&\qquad\les |I|^\gamma\|u_0\|_{\F L^{-s,r'}}\Big(\sum_{j=1}^3\|u_j\|_{\F L^{s,r}}\prod_{\substack{\ell=1,2,3\\\ell\neq j}}\|u_\ell\|_{\F L^{0,r}}\Big)\\
&\qquad\qquad\times\Big\{\sup_{n_0\in\Z}\sum_{\substack{n_1,n_2,n_3\in\Z\\n_1-n_2+n_3=n_0\\n_2\neq n_1,n_3}}\jb{(n_1-n_2)(n_0-n_2)}^{-r'\rho}\Big\}^{\frac1{r'}}.
\end{align*}

Thus it remains to bound uniformly in $n_0\in\Z$:
\begin{align*}
\sum_{\substack{n_1,n_2,n_3\in\Z\\n_1-n_2+n_3=n_0\\n_2\neq n_1,n_3}}\jb{(n_1-n_2)(n_0-n_2)}^{-r'\rho}&= \sum_{m_1,m_2\neq 0}\jb{m_1m_2}^{-r'\rho}\les 1
\end{align*}
from our assumption that $\rho>\frac1{r'}$. This shows that
\begin{align}\label{NN0}
&\Big|\int_I\int_\T \NN_0\big(e^{iW_t\dx^2}u_1,e^{iW_t\dx^2}u_2,e^{iW_t\dx^2}u_3\big)\cj{e^{iW_t\dx^2}u_0} dxdt\Big|\notag\\
&\qquad\qquad\les |I|^\gamma\|u_0\|_{\F L^{-s,r'}}\Big(\sum_{j=1}^3\|u_j\|_{\F L^{s,r}}\prod_{\substack{\ell=1,2,3\\\ell\neq j}}\|u_\ell\|_{\F L^{0,r}}\Big).
\end{align}

We estimate the contribution of $\RR$ similarly (note that ${\displaystyle \sum_{j=0}^3\pm_j n_j^2=0}$ on its support):
\begin{align}\label{RR}
&\Big|\int_I\int_\T \RR\big(e^{iW_t\dx^2}u_1,e^{iW_t\dx^2}u_2,e^{iW_t\dx^2}u_3\big)\cj{e^{iW_t\dx^2}u_0} dxdt\Big|\notag\\
&\qquad\qquad\les |I|^\gamma \sum_{n_0\in\Z}\prod_{j=0}^3|\ft u_j(n_0)|\les |I|^\gamma\|u_0\|_{\F L^{-s,r'}}\Big(\sum_{j=1}^3\|u_j\|_{\F L^{s,r}}\prod_{\substack{\ell=1,2,3\\\ell\neq j}}\|u_\ell\|_{\F L^{0,r}}\Big).
\end{align}

Now, for $R>0$ and $T\in(0,1]$, we again consider $$B_R(T):=\big\{u\in C([0,T);\F L^{s,r}(\T))\cap X^{s,r}_T,~~\|u\|_{X^{s,r}_T}\le R\big\}$$
which is closed in $X^{s,r}_T$.

Given $u_0\in \F L^{s,r}(\T)$, we set $R=2\|u_0\|_{\F L^{s,r}}$, and consider the map
$$\Phi_{u_0} : u\in B_R(T)\mapsto e^{i(W_t-W_0)\dx^2}u_0 \pm i \int_0^te^{i(W_t-W_{t'})\dx^2}\NN(u)(t')dt'.$$

By Proposition~\ref{PROP:UpVp}~(ii) and~(iii), we have
\begin{align*}
\big\|\Phi_{u_0}(u)-\Phi_{v_0}(v)\big\|_{X^{s,r}_T}\le \|u_0-v_0\|_{\F L^{s,r}}+\sup_{\substack{w\in Y^{-s,r'}_T\\\|w\|_{Y^{-s,r'}_T}\le 1}}\Big|\int_0^T\int_\M \big(\NN(u)-\NN(v)\big)\cj w dxdt\Big|.
\end{align*}
Since $u,v\in X^{s,r}_T$ and from the embedding of Proposition~\ref{PROP:UpVp}~(i), $w\in Y^{-s,r'}_T\subset X^{s,b,r}_T$ for $b=\big(\frac1\gamma\big)'>2$, using the atomic structure of $U^2$ and $U^b$ and the multilinearity of $\NN$, we can assume that $u,v$ (respectively $w$) are $U^2$ (respectively $U^b$) atoms:
\begin{align*}
u=\sum_{k=1}^{K^{(1)}}\mathbf{1}_{[t_{k-1}^{(1)};t_{k}^{(1)})}e^{iW_t\dx^2}\phi_{k}^{(1)}\text{ with }\|\phi_{k}^{(1)}\|_{\ell^2_{k}\F L^{s,r}}\le 1,
\end{align*}
\begin{align*}
v=\sum_{k=1}^{K^{(2)}}\mathbf{1}_{[t_{k-1}^{(2)};t_{k}^{(2)})}e^{iW_t\dx^2}\phi_{k}^{(2)}\text{ with }\|\phi_{k}^{(2)}\|_{\ell^2_{k}\F L^{s,r}}\le 1,
\end{align*}
and
\begin{align*}
w=\sum_{k=1}^{K^{(0)}}\mathbf{1}_{[t_{k-1}^{(0)};t_{k}^{(0)})}e^{iW_t\dx^2}\phi_{k}^{(0)}\text{ with }\|\phi_{k}^{(0)}\|_{\ell^b_{k}\F L^{-s,r'}}\le 1,
\end{align*}
We can also assume
\begin{align*}
u-v=\sum_{k=1}^{K^{(3)}}\mathbf{1}_{[t_{k-1}^{(3)};t_{k}^{(3)})}e^{iW_t\dx^2}\phi_{k}^{(3)}\text{ with }\|\phi_{k}^{(3)}\|_{\ell^2_{k}\F L^{s,r}}\le 1.
\end{align*}
Then, writing the nonlinearity as
\begin{align*}
\NN(u)-\NN(v)=\NN(u-v,u,u)+\NN(v,u-v,u)+\NN(v,v,u-v)
\end{align*}
and plugging the atomic decompositions into the Duhamel formula and using \eqref{NN0}-\eqref{RR} together with H\"older's inequality, we find
\begin{align*}
&\Big|\int_0^T\int_\M \NN(u-v,u,u)\cj w dxdt\Big|\\&\le \sum_{k_0,k_1,k_1',k_3}\Big|\int_{I_{k_0,k_1,k_1',k_3}}\int_\M \NN\big(e^{iW_t\dx^2}\phi_{k_3}^{(3)},e^{iW_t\dx^2}\phi_{k_1}^{(1)},e^{iW_t\dx^2}\phi_{k_1'}^{(1)}\big)\cj{e^{iW_t\dx^2}\phi_{k_0}^{(0)}} dxdt\Big|\\
&\les \sum_{k_0,k_1,k_1',k_3}\big|I_{k_0,k_1,k_1',k_3}\big|^\gamma \|\phi_{k_0}^{(0)}\|_{\F L^{-s,r'}}\|\phi_{k_3}^{(3)}\|_{\F L^{s,r}}\|\phi_{k_1}^{(1)}\|_{\F L^{s,r}}\|\phi_{k_1'}^{(1)}\|_{\F L^{s,r}}\\
&\les \Big\|\big|I_{k_0,k_1,k_1',k_3}\big|^\gamma\Big\|_{\ell^2_{k_1,k_1',k_3}\ell^{b'}_{k_0}}
\end{align*}
where
\begin{align*}
I_{k_0,k_1,k_1',k_3}=[t_{k_1'-1}^{(1)};t_{k_1'}^{(1)})\bigcap_{j=0,1,3}[t_{k_j-1}^{(j)};t_{k_j}^{(j)}).
\end{align*}
Since $\{t_{k_j}^{(j)}\}\in\ZZ(T)$, with our choice of $b=\big(\frac1\gamma\big)'>2$, we have
\begin{align*}
\Big\|\big|I_{k_0,k_1,k_1',k_3}\big|^\gamma\Big\|_{\ell^2_{k_1,k_1',k_3}\ell^{b'}_{k_0}}\le \Big\|\big|I_{k_0,k_1,k_1',k_3}\big|^\gamma\Big\|_{\ell^{b'}_{k_0,k_1,k_1',k_3}}=T^\gamma.
\end{align*}
This finally leads to
\begin{align*}
\Big|\int_0^T\int_\M \big(\NN(u)-\NN(v)\big)\cj w dxdt\Big|\les T^\gamma\|u-v\|_{X^{s,r}_T}\big(\|u\|_{X^{s,r}_T}+\|v\|_{X^{s,r}_T}\big)^2\|w\|_{Y^{-s,r'}_T},
\end{align*}
hence
\begin{align*}
\big\|\Phi_{u_0}(u)-\Phi_{v_0}(v)\big\|_{X^{s,r}_T}&\le \|u_0-v_0\|_{\F L^{s,r}}+CT^\gamma\|u-v\|_{X^{s,r}_T}\big(\|u\|_{X^{s,r}_T}+\|v\|_{X^{s,r}_T}\big)^2\\
&\le \|u_0-v_0\|_{\F L^{s,r}}+CT^\gamma (2R)^2\|u-v\|_{X^{s,r}_T}
\end{align*}
as $u,v\in B_R(T)$. Taking $T=T(R)=T(\|u_0\|_{\F L^{s,r}})\sim(1+\|u_0\|_{\F L^{s,r}})^{-\frac2\gamma}$, the estimate above shows that $\Phi_{u_0}$ defines a contraction on $B_R(T)$, thus having a unique fixed point in $B_R(T)$ which solves the mild formulation \eqref{Duhamel}. Uniqueness of the solution in the whole class $X^{s,r}_T$ follows from a usual argument as in the proof of Theorem~\ref{THM:noreg}. At last, the estimate above shows the Lipschitz continuity of the flow with respect to initial data.

\appendix
\section{About singular occupation measures}\label{SEC:appendix}
In Theorems~\ref{THM:noreg},~\ref{THM:NLScrit}, and~\ref{THM:NLSFL} above, the main assumptions on the occupation measure, \ref{A0} or \ref{A0gr} with $\rho$ large enough, imply that we are in a setting where this latter is always \emph{absolutely continuous with respect to the Lebesgue measure}, with bounded (under \ref{A0} or \ref{A0gr} with $\rho>1$) or at least square integrable (under \ref{A0gr} with $\rho>\frac12$) Radon-Nikodym derivative. 

While the main goal of this work is to address the case of a rough modulation $W_t$, for which the occupation measure is indeed regular enough, we can say a word on the situation where the latter is, on the contrary, a singular measure. This leads to the following generalization of assumption \ref{A0} above, where we include the possibility for $\mu_I$ to be a Frostman measure instead of being absolutely continuous.

Let $\theta \in [0,1)$ and define the following requirement.
\begin{assumption}{\textbf{(A$\theta$)}}\label{Atheta}
\rm $W_t$ is a continuous function such that there exists $R>0$ and for all $T>0$, there exists $C_W(T)>0$ non-decreasing such that for all $0<r\leq R$ and any interval $I\subset [0,T]$, it holds
\begin{equation}\label{SBE}
\sup_{a\in\R}\mu_I\big((a-r,a+r)\big)\leq C_W(T)r^{1-\theta}.
\end{equation}
 \end{assumption}
 
 Assumption \ref{Atheta} is closely linked to the small-ball estimates of Romito-Tolomeo \cite{RT}, but with uniformity of the constant with respect to the center $a$.
 
 \begin{remark}\label{REM:Atheta}
 \rm ~\\
 \noi (i) We can see that when $\theta=0$, assumption \ref{Atheta} is equivalent to \ref{A0} since the bound \eqref{SBE} implies that $\mu_I\ll dz$ for any $I$ with $\|d\mu_I/dz\|_{L^\infty}\leq C$.\\
\noi (ii) We also have that
 $$\ref{A0gr}\text{ with }0<\rho<1 \Longrightarrow \ref{Atheta}\text{ with }\theta=1-\rho \text{ and }C(T)=c\|\mu_{[0,T]}\|_{\F L^{\rho,\infty}}.$$
 Indeed, take a smooth non-negative bump function $\chi\in C^\infty_0(\R,\R^+)$ such that $\chi\equiv 1$ on $[-1,1]$, then for any interval $I\subset [0,T]$, any $a\in\R$ and $r>0$,
 \begin{align*}
 \mu_I\big((a-r,a+r)\big)&\leq \int_\R\chi\big(\frac{z-a}{r}\big)d\mu_I(z) = c\int_\R \ft\chi(\xi)\ft{\mu_I}\big(\frac{-\xi}{r}\big)e^{-i\frac{a}{r}\xi}d\xi\\
 &\les \|\mu_I\|_{\F L^{\rho,\infty}}\int_\R|\ft\chi(\xi)|\jb{r^{-1}\xi}^{-\rho}d\xi \les r^\rho \|\mu_I\|_{\F L^{\rho,\infty}} \|\chi\|_{\dot{W}^{-\rho,1}},
 \end{align*}
 and the last homogeneous Sobolev norm of $\chi$ is finite for $\rho<1$.
 \end{remark}
 
 The following proposition shows that under the weaker assumption \ref{Atheta}, one can recover Strichartz estimates \textbf{with loss} of derivatives. We only state the case of $\M=\T^d$, the case $\M=\R^d$ being easier\footnote{Though one needs the full strength of Strichartz estimates in arbitrary time interval in this case.}.
 \begin{proposition}\label{PROP:StrichartzLoss} 
Assume that \ref{Atheta} holds for some $\theta\in[0,1)$, and that $p\in[2,\infty]$ is such that there exists $C>0$ such that for any dyadic $N\in\N$, there is $\Lambda(N)>0$ such that for any $f\in L^2(\T^d)$ it holds $\|e^{it\Delta}\P_{\leq N} f\|_{L^p_{t,x}(\T\times\T^d)}\leq C \Lambda_N\|\P_{\leq N} f\|_{L^2}$. Then there is $C_W>0$ non-decreasing such that for any $T\in(0,1]$ and any cube $Q=n_0+[-N,N]^d\subset\Z^d$ of size $N$, it holds
 \begin{equation}\label{StrichartzLoss}
 \|e^{i(W_t-W_0)\Delta}\P_Q f\|_{L^p_TL^p_x}\leq C_W(T) N^{\frac{2}{p}\theta}\Lambda_N\|\P_Q f\|_{L^2}.
 \end{equation}
 \end{proposition}
\begin{remark}
\rm This may be adapted to more general dispersion relations, as this is again a transference result; here we rely on the Galilean invariance of the free Schrödinger flow on $\T^d$ to have an estimate for any cube instead of cube centered at the origin. When $\theta=0$, we recover the lossless transference of Theorem~\ref{THM:noreg}.
\end{remark}

The proof of this estimate relies on a transfer of spatial localization to time localization due to the structure of the linear equation, and on combining the Strichartz estimate for the usual propagator together with \eqref{SBE} and the time localization to deal with the roughness of the occupation measure. Before giving the proof of Proposition~\ref{PROP:StrichartzLoss}, we record that as a consequence (together with Remark~\ref{REM:Atheta}), for mass-subcritical NLS, we can slightly improve the result of Tanaka \cite{Tanaka} is the regime not covered by Theorem~\ref{THM:noreg} (see the discussion in Example~\ref{EX:noreg}~(2)), i.e. get the same range of well-posedness on $\T^d$ under the slightly weaker assumption \ref{Atheta}. In the mass-critical cases (cubic NLS on $\T^d$ and quintic NLS on $\T$), the loss in \eqref{StrichartzLoss} seems to give a smaller range of well-posedness than that of \cite{Tanaka}, and a multilinear refinement over Proposition~\ref{PROP:StrichartzLoss} might be needed. 
\begin{corollary}
Assume that \ref{Atheta} holds for some $\theta\in[0,1)$. Then for $(d,m)\notin \{(1,1),(1,2),(2,1)\}$, \eqref{NLS} is locally well-posed on $H^s(\T^d;\C)$ for any $s>\frac{d}{2}-\frac{1-\theta}{m}$.
\end{corollary}
\begin{proof}
We only sketch the argument, as this is a slight modification of the proofs of Theorem~\ref{THM:noreg} and of Example~\ref{EX:noreg}~(2) above: first, from \eqref{StrichartzLoc} with Proposition~\ref{PROP:StrichartzLoss}, the transference of Proposition~\ref{PROP:UpVp}~(iv) gives
\begin{equation}\label{StrichartzLossX}
\|\P_Q u\|_{L^p_{T,x}}\leq C_W(T)\Lambda_N(d,p)N^{\frac{2}{p}\theta}\|\P_Q u\|_{X^{0,p,2}_T}.
\end{equation} 
We choose the exponent $p=\frac{2m}{1-2m\epsilon}>2m$ with $0<\epsilon\ll s-(\frac{d}2-\frac{1-\theta}{m})$. Note that the mass-subcritical case corresponds to $m\geq 1+\frac2d$, so that $p>\frac{2(d+2)}{d}$. Therefore, proceeding similarly as in the proof of Example~\ref{EX:noreg}, we have for any $u_1,\dots,u_{2m+1}\in X^s_T$ and $u_0\in Y^{-s}_T$ that (see \eqref{multiTd0} and \eqref{multiTd1})
\begin{align*}
&\Big|\int_0^T\int_{\T^d}\prod_{j=0}^{2m+1}u_j^{\pm_j}dxdt\Big|\\
&\les  
\sum_{N_0\sim N_1\ge N_2\ge...\ge N_{2m+1}}
\big\|\P_{N_0}u_0\big\|_{L^\infty_TL^2_x}\big\|\P_{N_1}u_1\big\|_{L^\infty_TL^2_x}\prod_{j=2}^{2m+1}\|\P_{N_j}u_j\|_{L^{2m}_TL^\infty_x}\\
&+
\sum_{N_2\sim N_1\ge N_0\ge...\ge N_{2m+1}}\big\|\P_{N_2}u_2\big\|_{L^\infty_TL^2_x}\big\|\P_{N_1}u_1\big\|_{L^\infty_TL^2_x}\prod_{j\neq 1,2}\|\P_{N_j}u_j\|_{L^{2m}_TL^\infty_x}.
\end{align*}
Using Proposition~\ref{PROP:StrichartzLoss} with Sobolev inequality with the choice of $p$ above, and the definition of $\Lambda_N(d,p)$ in \eqref{StrichartzLoc}, we find
\begin{align*}
&\sum_{N_0\sim N_1\ge N_2\ge...\ge N_{2m+1}}
\big\|\P_{N_0}u_0\big\|_{L^\infty_TL^2_x}\big\|\P_{N_1}u_1\big\|_{L^\infty_TL^2_x}\prod_{j=2}^{2m+1}T^\epsilon\|\P_{N_j}u_j\|_{L^{p}_TL^\infty_x}\\
&\les \sum_{N_0\sim N_1\ge N_2\ge...\ge N_{2m+1}}
\big\|\P_{N_0}u_0\big\|_{X^0_T}\big\|\P_{N_1}u_1\big\|_{Y^0_T}T^{2m\epsilon}\prod_{j=2}^{2m+1}N_j^{\frac{d}p}\|\P_{N_j}u_j\|_{L^{p}_{T,x}}\\
&\les \sum_{N_0\sim N_1\ge N_2\ge...\ge N_{2m+1}}
\big\|\P_{N_0}u_0\big\|_{X^0_T}\big\|\P_{N_1}u_1\big\|_{Y^0_T}T^{2m\epsilon}\prod_{j=2}^{2m+1}C_W(T)N_j^{\frac{d}p+\frac2p\theta+\frac{d}{2}-\frac{d+2}{p}}\|\P_{N_j}u_j\|_{X^{0,p,2}_T}.
\end{align*}
This is enough to sum on the $N_j$'s provided that $\epsilon\ll s-(\frac{d}{2}-\frac{1-\theta}{m})$ so that $\frac{d}{p}+\frac{2}{p}\theta+\frac{d}{2}-\frac{d+2}{p}=\frac{d}2-\frac2p(1-\theta) < s$. Since $C_W$ is non-decreasing by assumption, we can drop the factor $C_W(T)$ for $T\in (0,1]$. With the multilinear estimate at hand with the factor $T^{2m\epsilon}$, we can run a fixed point argument in a ball of $C([0,T);H^s(\T^d))\cap X^s_T$ as in the previous proofs.
\end{proof}

Finally, we give the proof of Proposition~\ref{PROP:StrichartzLoss}.
\begin{proof}[Proof of Proposition~\ref{PROP:StrichartzLoss}]
As in the proof of Theorem~\ref{THM:noreg}, we use the occupation time formula \eqref{occupation} to write
\begin{align*}
\big\|e^{i(W_t-W_0)\Delta} \P_Q f\big\|_{L^p_TL^p_x}^p = \int_{I_W(T)}\big\|e^{iz\Delta}\P_Q(e^{-iW_0\Delta}f)\big\|_{L^p_x}^pd\mu_{[0,T]}(z).
\end{align*}
Using Galilean invariance, we can rewrite
\begin{align*}
e^{iz\Delta}\P_Q(e^{-iW_0\Delta}f)(x)=e^{in_0x-iz|n_0|^2}\P_{\leq N}\big(e^{iz\Delta}g\big)(x-2n_0 z),
\end{align*}
where $g(x)=e^{-in_0 x}e^{-iW_0\Delta}f(x)$. Therefore,
\begin{align*}
\big\|e^{i(W_t-W_0)\Delta} \P_Q f\big\|_{L^p_TL^q_x}^p&=\int_{I_W(T)}\big\|\P_{\leq N}\big(e^{iz\Delta}g\big)\big\|_{L^q_x}^pd\mu_{[0,T]}(z).
\end{align*}
Now observe that $e^{iz\Delta}\P_{\leq N}g$ is also periodic in $z$, so that we have in the sense of periodic space-time distributions
\begin{align*}
\F_{t,x}\big(e^{iz\Delta}\P_{\leq N}g\big)(k,n)=\mathbf{1}_{k=|n|^2}\psi_0(N^{-1}n)\ft g_n = \psi_0(N^{-2}k)\mathbf{1}_{k=|n|^2}\psi_0(N^{-1}n)\ft g_n, 
\end{align*}
where $\psi_0$ is the symbol of $\P_{\leq 1}$, so that
\begin{align*}
e^{iz\Delta}\P_{\leq N}g = (\psi_{N^2}\star_t e^{it\Delta}\P_{\leq N}g)(z),
\end{align*}
where $\psi_{N^2}(t)=\sum_{k\in\Z}e^{ikt}\psi_0(N^{-2}k)$ and $\star_t$ denotes convolution on $\T$ in time.
Then, using Minkowski's and H\"older's inequalities,
\begin{align*}
\big\|\P_{\leq N}\big(e^{iz\Delta}g\big)\big\|_{L^q_x}^p&\leq \Big(\int_\T|\psi_{N^2}|(z-t)\big\|e^{it\Delta}\P_{\leq N}g\big\|_{L^q_x}dt\Big)^p\\
& \leq \|\psi_{N^2}(z-t)^{1-\frac1p}\|_{L^{p'}_t}^p\big\|\psi_{N^2}(z-t)^\frac1p\|e^{it\Delta}\P_{\leq N}g\|_{L^q_x}\big\|_{L^p_t}^p\\
&\les \|\psi_{N^2}\|_{L^1_t}^p\int_\T|\psi_{N^2}|(z-t)\|e^{it\Delta}\P_{\leq N}g\|_{L^q_x}^pdt,
\end{align*}
uniformly in $N$ and $z$.

Note that $\|\psi_{N^2}\|_{L^1_t}=O(1)$ uniformly in $N$. Indeed, by Poisson summation formula (see e.g. \cite{Grafakos}), we have for all $t\in\T$ that
\begin{align}\label{PoissonPsi}
\psi_{N^2}(t)=C\sum_{k\in\Z}\widehat{\psi_0(N^{-2}\cdot)}(t+2\pi k).
\end{align}
Therefore,
\begin{align*}
\|\psi_{N^2}\|_{L^1(\T)} &\les \sum_{k\in\Z}\int_{2\pi k}^{2\pi (k+1)}|\widehat{\psi_0(N^{-2}\cdot)}(\lambda)|d\lambda = \|\widehat{\psi_0(N^{-2}\cdot)}\|_{L^1(\R)} = \|\widehat{\psi_0}\|_{L^1(\R)}.
\end{align*}

Integrating the previous bound on $I_W(T)$ against $\mu_{[0,T]}$ and using Fubini, we thus get
\begin{align*}
\big\|e^{i(W_t-W_0)\Delta} \P_Q f\big\|_{L^p_TL^p_x}^p&\les \Big(\sup_{t\in\T}\int_{I_W(T)}|\psi_{N^2}|(z-t)d\mu_{[0,T]}(z)\Big)\|e^{it\Delta}\P_{\leq N}g\|_{L^p_{t,x}(\T\times\T^d)}^p.
\end{align*}

Now, since $\psi_0\in C^\infty_0(\R)$, we can use successive integrations by parts to get for any $A\in\N$ that
\begin{align*}
\big|\widehat{\psi_0(N^{-2}\cdot)}(\lambda)\big|\les N^{2-2A}|\lambda|^{-A}\|\psi_0^{(A)}\|_{L^\infty},
\end{align*}
uniformly in $\lambda\in\R$ and $N$. In particular, since $T\in(0,1]$ in Proposition~\ref{PROP:StrichartzLoss}, we have that $I_W(T)\subset [-C,C]$ for some $C>0$ independent of $T$, therefore there is $C>0$ such that for any $t\in\T$ and $z\in I_W(T)$, $|z-t|\leq C$. For any $k\in\Z$ with $|k|\geq 2 C$, then $|z-t+2\pi k|\geq \frac12 |k|$ and the previous bound with $A=2$ gives
\begin{align*}
\sup_{t\in\T}\int_{I_W(T)}\sum_{\substack{k\in\Z\\ |k|\geq 2C}} \big|\widehat{\psi_0}(N^{-2}\cdot)(z-t+2\pi k)\big|d\mu_{[0,T]}(z) \les N^{-2}\int_{I_W(T)}d\mu_{[0,T]}(z) = TN^{-2}.
\end{align*}
On the other hand, we also have
\begin{align*}
&\sup_{t\in\T}\sum_{|k|\leq 2C}\int_{I_W(T)}\big|\widehat{\psi_0(N^{-2}\cdot)}(z-t+2\pi k)\big|d\mu_{[0,T]}(z)\\
&\les \sup_{t\in\T,|k|\leq 2C} \int_{|z-t+2\pi k|\leq N^{-2}}N^2d\mu_{[0,T]}(z)\\
&\qquad+\sum_{\substack{j\in\Z\\ -c\leq j \leq 2\log(N)}}\int_{2^{-j-1}\leq |z-t+2\pi k|<2^{-j}}N^{2-2A}2^{-Aj}d\mu_{[0,T]}(z).
\end{align*}
Invoking the bound \eqref{SBE} of assumption \ref{Atheta}, we can continue with
\begin{align*}
&\sup_{t\in\T}\sum_{|k|\leq 2C}\int_{I_W(T)}\big|\widehat{\psi_0(N^{-2}\cdot)}(z-t+2\pi k)\big|d\mu_{[0,T]}(z)\\
&\les C_W(T)N^2N^{-2(1-\theta)}+\sum_{\substack{j\in\Z\\ -c\leq j \leq 2\log(N)}}N^{2-2A}2^{-Aj}C_W(T)2^{-j(1-\theta)}\les C_W(T)N^{2\theta}.
\end{align*}
Putting everything together and using the usual Strichartz estimate from the assumption of Proposition~\ref{PROP:StrichartzLoss}, we finally find
\begin{align*}
\big\|e^{i(W_t-W_0)\Delta} \P_Q f\big\|_{L^p_TL^p_x}^p&\les C_W(T) N^{2\theta}\|e^{it\Delta}\P_{\leq N}g\|_{L^p_{t,x}(\T\times\T^d)}^p\\
&\les C_W(T)N^{2\theta}\Lambda_N^p\|\P_{\leq N}g\|_{L^2}^p=C_W(T)N^{2\theta}\Lambda_N^p\|\P_Q f\|_{L^2}^p.
\end{align*}
This proves \eqref{StrichartzLoss}.
\end{proof}

\end{document}